\documentclass[12pt,reqno]{amsart}

\usepackage{amsmath}
\usepackage{amssymb}
\usepackage{amscd}
\usepackage{graphicx}
\usepackage{enumerate}
\usepackage{color}
\usepackage{bbm}
\usepackage{tikz, verbatim}
\usepackage[T1]{fontenc}

\usetikzlibrary{matrix,calc}

\newtheorem{thm}{Theorem}[section]
\newtheorem{lem}[thm]{Lemma}
\newtheorem{prop}[thm]{Proposition}

\newtheorem{ques}[thm]{Question}

\newtheorem{cor}[thm]{Corollary}
\newtheorem{defn}[thm]{Definition}

\newtheorem{rem}[thm]{Remark}

\newtheorem{exam}[thm]{Example}

\newcommand{\inv}{invariant}
\newcommand{\im}{invariant measure}
\newcommand{\sq}{sequence}

\newcommand{\z}{\mathbb Z}
\newcommand{\na}{\mathbb N}
\newcommand{\re}{\mathbb R}
\newcommand{\s}{\mathbb S}

\newcommand{\p}{\mathcal P}

\newcommand{\PP}{\mathbb P}
\newcommand{\B}{\mathcal B}

\newcommand{\tl}{topological}

\newcommand{\xbm}{$(X,\mathcal B,\mu)$}

\numberwithin{equation}{section}

\begin{document}

\baselineskip=12pt
\title{A fresh look at the notion of normality}

\author{Vitaly Bergelson}
\address{Department of Mathematics, Ohio State University, Columbus,
  OH 43210, USA}
\email{vitaly@math.ohio-state.edu}

\author{Tomasz Downarowicz}
\address{Faculty of Pure and Applied Mathematics, Wroc\l aw University
  of Science and Technology, Wybrze\.ze Wyspia\'nskiego 21, 50-370
  Wroc\l aw, Poland}
\email{Tomasz.Downarowicz@pwr.edu.pl}

\author{Micha{\l} Misiurewicz}
\address{Department of Mathematical Sciences,
Indiana University-Purdue University Indianapolis,
402 N. Blackford Street, Indianapolis, IN 46202, USA}
\email{mmisiure@math.iupui.edu}

\date{\today}

\begin{abstract}
Let $G$ be a countably infinite cancellative amenable semigroup and let $(F_n)$ be a (left) F{\o}lner sequence in $G$. We
introduce the notion of an $(F_n)$-normal set in $G$ and an
$(F_n)$-normal element of $\{0,1\}^G$. When $G$ = $(\na,+)$ and $F_n
= \{1,2,...,n\}$, the $(F_n)$-normality coincides with the classical
notion. We prove several results about $(F_n)$-normality, for example:
\begin{itemize}
	\item If $(F_n)$ is a F{\o}lner sequence in $G$, such that for
          every $\alpha\in(0,1)$ we have $\sum_n \alpha^{|F_n|}<\infty$,
          then almost every
          (in the sense of the uniform product measure $(\frac12,\frac12)^G$)
          $x\in\{0,1\}^G$ is $(F_n)$-normal.
        \item For any F{\o}lner sequence $(F_n)$ in $G$, there exists an
          effectively defined Cham\-per\-nowne-like $(F_n)$-normal set.
	\item There is a rather natural and sufficiently wide class of
          F{\o}lner sequences $(F_n)$ in $(\na,\times)$, which we call
          ``nice'', for which the Champernowne-like construction can
          be done in an algorithmic way. Moreover, there exists a
          Champernowne-like set which is $(F_n)$-normal for every nice
          F{\o}lner \sq\ $(F_n)$.
\end{itemize}
We also investigate and juxtapose combinatorial and Diophantine
properties of normal sets in semigroups $(\na,+)$ and $(\na,\times)$.
Below is a sample of results that we obtain:
\begin{itemize}
	\item Let $A\subset\na$ be a classical normal set. Then, for
          any F{\o}lner sequence $(K_n)$ in $(\na,\times)$ there exists
          a set $E$ of $(K_n)$-density $1$, such that for any finite
          subset $\{n_1,n_2,\dots,n_k\}\subset E$, the intersection
          $A/{n_1}\cap A/{n_2}\cap\ldots\cap A/{n_k}$ has positive
          upper density in $(\na,+)$. As a consequence, $A$ contains
          arbitrarily long geometric progressions, and, more
          generally, arbitrarily long ``geo-arithmetic''
          configurations of the form $\{a(b+ic)^j,0\le i,j\le k\}$.
	\item For any F{\o}lner \sq\ $(F_n)$ in $(\na,+)$ there exist
          uncountably many $(F_n)$-normal Liouville numbers.
	\item For any nice F{\o}lner sequence $(F_n)$ in $(\na,\times)$
          there exist uncountably many $(F_n)$-normal Liouville
          numbers.
\end{itemize}
\end{abstract}

\maketitle

\section{Introduction}

It follows from the classical law of large numbers
(\cite{Bo09}\footnote{See Appendix for historical notes.}) that
given a fair coin whose sides are labeled 0 and 1, the infinite binary
sequence $(x_n)$, obtained by independent tossing of the coin, is
almost surely \emph{normal}, meaning that, for any
$k\in\na=\{1,2,\dots\}$, any 0-1 word of length $k$, $w=\langle
w_1,w_2,\dots,w_k\rangle\in\{0,1\}^k$, appears in $(x_n)$ with
frequency $2^{-k}$. This provides a proof of existence of normal
sequences (note that \emph{a priori} it is not even clear whether
normal sequences exist!). There are also numerous explicit
constructions of normal sequences (see for instance \cite{Ch33, M33,
  DE52}). For example, the \emph{Champernowne sequence}
$1\,10\,11\,100\,101\,110\,\dots$, which is formed by the sequence
$1,2,3,4,5,6,\dots$ written in base 2, is a normal sequence.

Any 0-1 sequence $(x_n)\in\{0,1\}^\na$ may be viewed as the sequence
of digits in the binary expansion of the real number
$x=\sum_{n=1}^\infty x_n2^{-n}\in[0,1]$, which leads to an equivalent
formulation of the above fact: almost every $x\in[0,1]$ is normal in
base 2. Similarly, due to the natural bijection between 0-1 sequences
and subsets of $\na$ (any subset of $\na$ is identified with its
indicator function which is a 0-1 \sq), one can talk about
\emph{normal sets} in $\na$ (more accurately, in $(\na,+)$; see the
discussion below).

The peculiar combinatorial and Diophantine properties of normal \sq
s/sets/ numbers, together with the fact that they are ``typical'' (in
the sense of measure), make them a natural object of interest and a
source of various generalizations, see~\cite{Ko15,PV15,BB13,Fi05}.

The classical definition of normality of a 0-1
\sq\ $x=(x_n)_{n\in\na}\in\{0,1\}^\na$ is formulated as
follows.\footnote{In \cite{Bo09}, this formulation appears not as the
  definition but as a ``characterization'' of normality, see
  Appendix for more details.}

\begin{defn}\label{clnor}
For $n,k\in\na$ ($k\le n$), and a 0-1 word $w\in\{0,1\}^k$, we let
$\mathsf N(w,x,n)$ be the number of times the word $w$ occurs as a
subword of the word $\langle x_1,x_2,\dots,x_n\rangle\in\{0,1\}^n$:
	\[
	\mathsf{N}(w,x,n)=|\{m\in\{1,\dots,n-k+1\}:
	\langle x_m,x_{m+1},\dots,x_{m+k-1}\rangle=w\}|
	\]
        (here $|\cdot|$ denotes the cardinality of a set). A
        \sq\ $x\in\{0,1\}^\na$ is \emph{normal} if for every
        $k\in\na$ and every $w\in\{0,1\}^k$ we have
\begin{equation}\label{00}
\lim_{n\to\infty} \frac1n\mathsf N(w,x,n)=2^{-k}.
\end{equation}
\end{defn}

One may ask a naive but in some sense natural question whether
replacing the \sq\ of ``averaging intervals'' $\{1,2,\dots,n\}$ (which
are implicit in the above definition because one can write $\langle
x_1,x_2,\dots,x_n\rangle=x|_{\{1,2,\dots,n\}}$) by a more general
\sq\ $(F_n)$ of (a priori arbitrary) finite subsets of $\na$ leads to
a meaningful generalization of the notion of normality. More
precisely, one would like to count the number of times the word $w$
occurs as a subword of $x|_{F_n}$:
\begin{multline*}
\mathsf{N}(w,x,F_n)=\\|\{m\in\na: \{m,m+1,\dots,m+k-1\}\subset F_n
\text{ and } \langle x_m,x_{m+1},\dots,x_{m+k-1}\rangle = w\}|
\end{multline*}
and call a \sq\ $x\in\{0,1\}^\na$ \emph{$(F_n)$-normal} if, for every
$k\in\na$ and any $w\in\{0,1\}^k$, one has
\begin{equation}\label{fnnor}
\lim_{n\to\infty} \frac1{|F_n|}\mathsf N(w,x,F_n)=2^{-k}.
\end{equation}

It turns out that in order for the above notion of $(F_n)$-normality
to be nonvoid, the \sq\ of sets $(F_n)$ has to be a \emph{F{\o}lner
  \sq}, i.e., satisfy the so-called \emph{F{\o}lner condition}:
\begin{equation}\label{33}
\forall k\in\na\ \ \lim_{n\to\infty}\frac{|F_n\cap (F_n-k)|}{|F_n|}=1
\end{equation}
(in particular, it must hold that $|F_n|\to\infty$). As a matter of
fact, the F{\o}lner condition is implied by a rather mild requirement
that there exists an $x\in\{0,1\}^\na$ such that, for each $k\in\na$,
\begin{equation}\label{44}
 \lim_{n\to\infty} \frac1{|F_n|}\sum_{w\in\{0,1\}^k}\mathsf N(w,x,F_n) = 1.
\end{equation}
The proof will be given later (see Theorem \ref{26} below) in a more
general context.

Our next observation is that if $x\in\{0,1\}^\na$ is $(F_n)$-normal
then not only words, but in fact all \emph{0-1 blocks}, occur in $x$
with ``correct frequencies'', by which we mean the following. Let $K$
be a nonempty finite subset of $\na$. Any element (function)
$B\in\{0,1\}^K$ will be called a \emph{block}. We will say that
\emph{a shift of a block $B\in\{0,1\}^K$ occurs} in the block
$x|_{F_n}\in\{0,1\}^{F_n}$ at a position $m\in\na\cup\{0\}$ if
\[
(\forall i\in K)\ i+m\in F_n\text{ and } x_{i+m}=B(i).
\]
We let $\mathsf N(B,x,F_n)$ be the number of shifts of the block $B$
occurring in $x|_{F_n}$, i.e.,
	\[
	\mathsf{N}(B,x,F_n)=|\{m\in\na\cup\{0\}: (\forall i\in
        K)\ i+m\in F_n\text{ and }
	x_{i+m}=B(i)\}|.
	\]
Then $(F_n)$-normality of $x$ implies
\begin{equation}\label{fnbnor}
\lim_{n\to\infty} \frac1{|F_n|}\mathsf N(B,x,F_n)=2^{-|K|},
\end{equation}
for any nonempty finite set $K$ and every block $B\in\{0,1\}^K$. We
will prove this implication in Section \ref{pre} using the language of
dynamics (see Lemma \ref{nn}).

Once we are driven into considering F{\o}lner \sq s of the ``averaging
sets'', the natural context for continuing our discussion of normality
becomes that of countably infinite amenable cancellative semigroups\footnote{
	A semigroup $G$ is (two-sided) \emph{cancellative}
  if, for any $a,b,c\in G$, $ab=ac\implies b=c$ and $ba=ca\implies
  b=c$.}
$G$ (it is known that such semigroups admit (left) F{\o}lner \sq s
see \cite[Theorem 3.5 and Corollary 4.3]{Na64}, see also Definition \ref{fol}).
In order to adapt the
definition of $(F_n)$-normality to this context, we pick a F{\o}lner
\sq\ $(F_n)$ in $G$, fix a 0-1-valued function $x=(x_g)_{g\in
  G}\in\{0,1\}^G$, and, for each finite set $K\subset G$ and a block $B\in\{0,1\}^K$, 
  denote
\begin{equation}\label{enen}
	\mathsf N(B,x,F_n)=\{g\in G\cup\{e\}: (\forall h\in K)\ hg\in
        F_n \text{ and }x_{hg}=B(h)\},
\end{equation}
where $e$ is the formal identity element added to $G$ in case $G$
lacks an identity. We will say that $x$ is $(F_n)$-normal if for any
nonempty finite $K\subset G$ and every $B\in\{0,1\}^K$, one has (as in
the case of $(\na,+)$),
\begin{equation}\label{fnbnorm}
\lim_{n\to\infty} \frac1{|F_n|}\mathsf N(B,x,F_n)=2^{-|K|}.
\end{equation}

Let us now examine closer the dynamical underpinnings of the notion of
normality. Let $G$ be a countably infinite amenable cancellative
semigroup. The semigroup $G$ acts naturally on the symbolic space
$\{0,1\}^G$ by shifts, as follows: for $g\in G$ and $x=(x_h)_{h\in
  G}$, $\sigma_g(x)=(x_{hg})_{h\in G}$.\footnote{In the classical case
  $G=(\na,+)$, the action is given by
  $\sigma^m(x)=(x_{n+m})_{n\in\na}$ (where $m\in\na$ and
  $x=(x_n)_{n\in\na}$).}

For any nonempty finite set $K\subset G$, each block $B\in\{0,1\}^K$
determines a \emph{cylinder}
\[
[B]=\{x\in\{0,1\}^G: x|_{K}=B\}.
\]
As we will explain later (see Theorem \ref{tff}), if
  $(F_n)$ is a F{\o}lner \sq\ then $(F_n)$-normality can be expressed
  in terms of the shift action and cylinder sets in the following way:
\begin{itemize}
\item An element $x\in\{0,1\}^G$ is $(F_n)$-normal if and only if for
  every nonempty finite set $K$ and every block $B\in\{0,1\}^K$ one
  has
\[
\lim_{n\to\infty}\frac1{|F_n|}|\{g\in F_n: \sigma_g(x)\in[B]\}|=2^{-|K|}.
\]
\end{itemize}

When dealing with a general amenable semigroup $G$ and a F{\o}lner
\sq\ $(F_n)$, it is not a priori obvious whether $(F_n)$-normal
elements $x\in\{0,1\}^G$ exist. We solve this problem in the
affirmative by showing, in Theorem~\ref{B1} below, that for any
countably infinite cancellative amenable semigroup $G$ and any F{\o}lner \sq\ $(F_n)$ in $G$, with $|F_n|$ strictly
increasing, $\lambda$-almost every $x\in\{0,1\}^G$ is $(F_n)$-normal,
where $\lambda$ is the uniform product measure $(\frac12,\frac12)^G$
on $\{0,1\}^G$.
In an equivalent form (see Theorem \ref{co}), our result can be
interpreted as a sort of pointwise ergodic theorem for Bernoulli
shifts. Namely, for any F{\o}lner \sq\ $(F_n)$ with $|F_n|$ strictly
increasing\footnote{Actually, our assumption in Theorems \ref{B1} and
  \ref{co} on the F{\o}lner \sq\ $(F_n)$ is even weaker: for any
  $\alpha\in(0,1)$, $\sum_{n\in\na}\alpha^{|F_n|}<\infty$.}, any
\emph{continuous} function $f$ on $\{0,1\}^G$ and $\lambda$-almost
every $x\in\{0,1\}^G$ we have
\begin{equation*}\label{cesaro1}
\lim_{n\to\infty}\frac1{|F_n|}\sum_{g\in F_n}f(\sigma_gx)=\int f\,d\lambda.
\end{equation*}
We emphasize that the pointwise ergodic theorem for general actions of
amenable groups (and measurable functions) holds only for
\emph{tempered} F{\o}lner \sq s which satisfy the so-called Shulman's
condition:
\begin{equation}\label{tempered}
\left|\bigcup_{i=1}^n F_i^{-1}F_{n+1}\right|\le C|F_{n+1}|
\end{equation}
(see \cite[p.\,83]{Li01}, see also \cite{AJ75} for the necessity of
Shulman's condition).

The set $\na$ of natural numbers has two natural semigroup operations:
addition and multiplication. This leads to two parallel notions of
normality of subsets of $\na$, which will be referred to as
\emph{additive }and \emph{multiplicative} normality, respectively. The
possibility of juxtaposing the Diophantine and combinatorial
properties of additively and multiplicatively normal subsets of $\na$
served as the initial motivation for this paper.

The notion of $(F_n)$-normality in $(\na,+)$ and $(\na,\times)$ allows
one to reconsider, from the more general point of view, the classical
results dealing with the existence of normal Liouville numbers\footnote{
Let us recall that an irrational number $x$ is called a Liouville number
if for every natural $k$ there exists a rational number $\frac pq$ such
that $|x-\frac pq| < \frac1{q^k}$.}
(see \cite{Bu02}\footnote{In fact, in \cite{Bu02} Bugeaud proves the
  existence of \emph{absolutely normal} (i.e., classical normal with
  respect to any base) Liouville numbers. We are interested in
  $(F_n)$-normality in base $2$, but for a general F{\o}lner
  \sq\ $(F_n)$ in $(\na,+)$, as well as in $(\na,\times)$.}). While it
is true that the set of Liouville numbers is residual (i.e.,
topologically large, see \cite[Theorem 5]{Gr83}), the set of
$(F_n)$-normal numbers is, as we will show in subsection \ref{tsmall},
of the first category (i.e., topologically small). This holds for any
F{\o}lner \sq\ $(F_n)$ in either $(\na,+)$ or $(\na,\times)$. As for
the largeness in the sense of measure, the situation is reversed: as
we have already mentioned, the set of $(F_n)$-normal numbers is (for
any F{\o}lner \sq\ $(F_n)$, in either $(\na,+)$ or $(\na,\times)$, with
$|F_n|$ strictly increasing) of full Lebesgue measure, while it is
well known that the set of Liouville numbers has Lebesgue measure zero
(see for example~\cite{Ox80}). So, using just the criteria of
topological or measure-theoretic largeness it is impossible to decide
whether the sets of Liouville numbers and of $(F_n)$-normal numbers
have nonempty intersection.

Below is a brief description of results obtained in this paper.

\smallskip
$\bullet$ Section \ref{pre} is devoted to reviewing or establishing basic
facts about amenable groups and semigroups, which are needed in the
sequel. In particular we prove an auxilliary theorem which shows that
in many situations one can deal, without loss of generality, with
amenable groups rather than semigroups.

\smallskip
$\bullet$ In Section \ref{two} we establish left invariance of the
class of $(F_n)$-normal sets in countably infinite amenable cancellative semigroups.

\smallskip
$\bullet$ Section \ref{for} contains our ``ergodic theorem for
Bernoulli shifts and continuous functions'' which says that the set
$\mathcal N((F_n))$ of $(F_n)$-normal elements of $\{0,1\}^G$ is large
in the sense of measure. By the way of contrast, we also prove that
the set $\mathcal N((F_n))$ is small in the sense of topology (is of
first category).

\smallskip
$\bullet$ In Section \ref{cha} we give a general Champernowne-like
construction of an $(F_n)$-normal element $x\in\{0,1\}^G$ for any
countably infinite amenable cancellative semigroup $G$ and
any F{\o}lner \sq\ $(F_n)$ in $G$.

\smallskip
$\bullet$ Section \ref{mnor} focuses on the notion of normality in the
semigroup $(\na,\times)$ of multiplicative positive integers. We
introduce a natural class of F{\o}lner \sq s which we call ``nice''.
For any nice F{\o}lner \sq\ $(F_n)$ we construct a Champernowne-like
$(F_n)$-normal element $x\in\{0,1\}^\na$. Due to monotileability of
the semigroup $(\na, \times)$ and properties of a nice F{\o}lner \sq,
the construction resembles that of the classical Champernowne number
and is much more transparent than the one described in the preceding
section.

We also study the class of elements $x\in\{0,1\}^\na$ which are normal
with respect to all nice F{\o}lner \sq s in $(\na,\times)$. We call
these elements \emph{net-normal}. We prove that the set of net-normal
elements has measure zero but is nonempty (to this end we use a
modification of the Champernowne-like construction from the preceding
section).

\smallskip
$\bullet$ Section \ref{s4} is devoted to the study of combinatorial
and Diophantine properties of additively and multiplicatively normal
subsets of $\na$. In particular, we prove the following results:

\begin{itemize}
\item[--] Let $(F_n)$ be a F{\o}lner sequence in $(\na,+)$. Then any
  $(F_n)$-normal set $S$ contains solutions of any partition-regular
  system of linear equations.\footnote{A system of equations is called
    \emph{partition-regular} if for any finite coloring of $\na$ there exists
    a monochromatic solution.}
\item[--] Let $(F_n)$ be a F{\o}lner sequence in $(\na,\times)$. Then
  any $(F_n)$-normal set $S$ contains solutions of any homogeneous
  system of polynomial equations which has solutions in $\na$.
\item[--] Let $S$ be any classical normal set in $(\na,+)$. Then
\begin{enumerate}[(i)]
\item $S$ contains solutions $a,b,c$ of any equation $ia+jb=kc$, where
  $i,j,k$ are arbitrary positive integers,
\item $S$ contains pairs $\{n+m,nm\}$ with arbitrary large $n,m$,
\item $S$ contains arbitrarily long geometric progressions, and, more
generally, arbitrarily long ``geo-arithmetic'' configurations of the
form \hfill\break $\{a(b+ic)^j,\ 0\le i,j\le k\}$.
\end{enumerate}
\end{itemize}

\smallskip
$\bullet$ In Section \ref{fifa} we show that for any F{\o}lner
\sq\ $(F_n)$ in $(\na,+)$ there exists an $(F_n)$-normal Liouville
number (actually, we construct a Cantor set of such numbers) and,
likewise, for any nice F{\o}lner \sq\ $(F_n)$ in $(\na,\times)$ there
exists an $(F_n)$-normal Liouville number (and indeed a Cantor set of
$(F_n)$-normal Liouville numbers).

\section{Preliminaries}\label{pre}
We now present some background material concerning properties of
F{\o}lner \sq s in countably infinite amenable cancellative
semigroups, and tilings in countably infinite amenable groups.

Let $G$ be a cancellative semigroup. Recall
that given $g\in G$ and a finite subset $F\subset G$, $g^{-1}F$ stands
for $\{h\in G: gh\in F\}$.

\begin{defn}\label{fol}
A \sq\ $(F_n)$ of finite subsets of $G$ is a \emph{F\o lner \sq} if it
satisfies the \emph{F{\o}lner condition}:
\[
\forall g\in G\ \ \lim_{n\to\infty}\frac{|F_n\cap g^{-1}F_n|}{|F_n|}=1.
\]
\end{defn}
We have $g^{-1}F\cap F = \{f\in F: gf\in F\}$. By
cancellativity, $f\in g^{-1}F\cap F \iff gf\in gF\cup F$, and thus
\begin{equation}\label{gF}
|g^{-1}F\cap F| = |gF\cap F|.
\end{equation}

It follows that the F\o lner condition is equivalent to
\[
\forall g\in G\ \ \lim_{n\to\infty}\frac{|gF_n\cap F_n|}{|F_n|}=1.
\]
Another useful equivalent form of the F{\o}lner condition utilizes the
notion of a $(K,\varepsilon)$-\inv\ set.

\begin{defn}
Given a nonempty finite set $K\subset G$ and $\varepsilon>0$ we will
say that a finite set $F\subset G$ is
\emph{$(K,\varepsilon)$-invariant} if
\[
\frac{|KF\triangle F|}{|F|}\le\varepsilon
\]
($\triangle$ stands for the symmetric difference of sets).
\end{defn}

It is not hard to see that a \sq\ of finite sets $(F_n)$ is F{\o}lner
if and only if for any nonempty finite $K\subset G$ and
$\varepsilon>0$, the sets $F_n$ are eventually $(K,\varepsilon)$-\inv.

We remark that a general F{\o}lner \sq\ need not be increasing with
respect to inclusion (in particular, it can consist of disjoint sets),
the cardinalities $|F_n|$ need not increase (but, of course
$|F_n|\to\infty$), and the union $\bigcup_{n\ge 1} F_n$ need not equal
the whole semigroup.

Given a F{\o}lner sequence $(F_n)$ in $G$ and a set $V\subset G$, one
defines the \emph{upper and lower $(F_n)$-densities} of $V$ by the
formulas
\begin{align*}
\overline d_{(F_n)}(V)&=\limsup_{n\to\infty}\frac{|F_n\cap V|}{|F_n|},
\\
\underline d_{(F_n)}(V)&=\liminf_{n\to\infty}\frac{|F_n\cap V|}{|F_n|}.
\end{align*}
If $\overline d_{(F_n)}(V)=\underline d_{(F_n)}(V)$, then we denote
the common value by $d_{(F_n)}(V)$ and call it the
\emph{$(F_n)$-density} of $V$. The F{\o}lner property of $(F_n)$ and
cancellativity immediately imply that for any $V\subset G$ and any
$g\in G$,
\begin{equation}\label{transl}
\overline d_{(F_n)}(V) = \overline d_{(F_n)}(gV) = \overline d_{(F_n)}(g^{-1}V)
\end{equation}
(analogous equalities hold for $\underline d_{(F_n)}(\cdot)$ and
$d_{(F_n)}(\cdot)$).

\begin{defn} Let $K$ and $F$ be nonempty finite subsets of $G$.
\begin{enumerate}
\item The \emph{$K$-core} of $F$ is the set $F_K=\{h\in G:Kh\subset
  F\}=\bigcap_{g\in K}g^{-1}F$.
\item $K$ is called an \emph{$\varepsilon$-modification} of $F$ if
  $\frac{|K\triangle F|}{|F|}\le\varepsilon$, \ ($\varepsilon>0$).
\end{enumerate}
\end{defn}

The following elementary lemma is a slightly more general form of
Lemma~2.6 in~\cite{DHZ17}. We include the proof for reader's
convenience.

\begin{lem}\label{estim}
For any $\varepsilon>0$ and any nonempty finite subset $K$ of an
amenable cancellative semigroup $G$, there exists $\delta>0$ (in fact
$\delta = \frac\varepsilon{2|K|}$), such that if $F\subset G$ is
finite and $(K,\delta)$-invariant then the $K$-core of $F$ is an
$\varepsilon$-modification of $F$.
\end{lem}

\begin{proof} Note that $(K,\delta)$-invariance of $F$ implies that
\[
(\forall g\in K)\ \ |gF\setminus F|\le\delta|F|,
\]
i.e., using \eqref{gF},
\[
(\forall g\in K)\ \ |g^{-1}F\cap F| = |gF\cap F| \ge(1-\delta)|F|,
\]
in particular, $|g^{-1}F\setminus F|\le\delta|F|$.
Using the above, we get
\[
|F_K\cap F|=\left|\bigcap_{g\in K}(g^{-1}F\cap
F)\right|\ge(1-|K|\delta)|F|,
\]
while $|F_K\cup F|\le\sum_{g\in K}|g^{-1}F\setminus F|+|F|
\le(1+\delta|K|)|F|$. Combining the two estimates above, we obtain
$|F_K\triangle F|=|F_K\cup F|-|F_K\cap
F|\le2\delta|K||F|=\varepsilon|F|$.
\end{proof}

\begin{defn}\label{equi}
We will say that two F{\o}lner \sq s $(F_n)$ and $(F'_n)$ in an
amenable semigroup $G$ are \emph{equivalent} if $\frac{|F'_n\triangle
  F_n|}{|F_n|}\to 0$ (equivalently, $\frac{|F'_n\triangle
  F_n|}{|F'_n|}\to 0$).
\end{defn}

Note that if $(F_n)$ is a F{\o}lner \sq\ in $G$ and, for each $n$,
$F_n'$ is an $\varepsilon_n$-modification of $F_n$, where
$\varepsilon_n\to 0$, then $(F_n')$ is a F{\o}lner \sq\ equivalent to
$(F_n)$.

\begin{rem}\label{i-ii}
It is not hard to see that if $(F_n)$ and $(F_n')$ are
equivalent F{\o}lner \sq s then
\begin{enumerate}
	\item[(i)] the notions of (upper/lower) $(F_n)$-density and
          $(F_n')$-density coincide,
	\item[(ii)] the notions of $(F_n)$-normality and
          $(F_n')$-normality coincide.
\end{enumerate}
\end{rem}
Invoking Lemma \ref{estim}, we obtain the following lemma.

\begin{lem}\label{25}
If $(F_n)$ is a F{\o}lner \sq\ in a countably infinite cancellative
semigroup $G$ and $K\subset G$ is nonempty finite then the
\sq\ $(F_{n,K})$ (of the $K$-cores of $F_n$), is a F{\o}lner
\sq\ equivalent to $(F_n)$.
\end{lem}
In particular, for any $g\in G$, $(g^{-1}F_n)$ is a F{\o}lner
\sq\ equivalent to $(F_n)$. By~\eqref{gF}, the same holds for
$(gF_n)$.\footnote{It is easy to see that if $(F_n)$ is a F{\o}lner
  \sq\ in a countably infinite cancellative semigroup $G$ and $g\in G$
  then the \sq\ $(F_ng)$ satisfies the F{\o}lner condition. However,
  unless $G$ is commutative, $(F_ng)$ need not be equivalent to
  $(F_n)$.}

We can now rephrase slightly the definition of $(F_n)$-normality using
a ``dynamical'' modification $\mathsf{\tilde N}(B,x,F_n)$ of the
quantity $\mathsf N(B,x,F_n)$ introduced in \eqref{enen}. For two
nonempty finite sets $F,K\subset G$, a 0-1-valued function
$x\in\{0,1\}^G$, and a block $B\in\{0,1\}^K$ let us denote by
$\mathsf{\tilde N}(B,x,F)$ the number of visits of the orbit of $x$ to
the cylinder $[B]$ at ``times'' belonging to $F$:
\[
\mathsf{\tilde N}(B,x,F)=|\{g\in F: \sigma_g(x)\in[B]\}|.
\]
For comparison, as easily verified, $\mathsf N(B,x,F)$ (see
\eqref{enen}) counts the visits of the orbit of $x$ in $[B]$ at
``times'' belonging to the $K$-core of $F$ in the extended semigroup
$G\cup\{e\}$:
\[
\mathsf{N}(B,x,F)=|\{g\in F^o_K: \sigma_g(x)\in[B]\}|,
\]
where
\[
F^o_K=\{g\in G\cup\{e\}: Kg\subset F\}.
\]
(Clearly, if $G$ has an identity element then $F^o_K=F_K$. In any
case, $F_K\subset F^o_K\subset F_K\cup\{e\}$). The difference between
$\mathsf{\tilde N}(B,x,F)$ and $\mathsf{N}(B,x,F)$ is best seen in the
classical case of $(\na,+)$. Let $F=\{1,2,\dots,n\}$ and let $B$ be a
word $w$ of length $k$. Then we have
\begin{gather*}
	\mathsf{\tilde N}(w,x,F)=|\{m\in\{1,\dots,n\}:\langle
        x_m,x_{m+1},\dots,x_{m+k-1}\rangle=w\}|,\\ \mathsf{N}(w,x,F)=
        |\{m\in\{1,\dots,n-k+1\}:\langle
        x_m,x_{m+1},\dots,x_{m+k-1}\rangle=w\}|.
\end{gather*}

Here is now a reformulation of the definition of normality in terms of
$\mathsf{\tilde N}(B,x,F_n)$.

\begin{thm}\label{tff}
Let $G$ be a countably infinite amenable cancellative semigroup and
let $(F_n)$ be a F{\o}lner \sq\ in $G$. An element $x\in\{0,1\}^G$ is
$(F_n)$-normal of and only if, for any nonempty finite $K\subset G$
and each $B\in\{0,1\}^K$, we have
\begin{equation}\label{tfnbnorm}
\lim_{n\to\infty} \frac1{|F_n|}\mathsf{\tilde N}(B,x,F_n)=2^{-|K|}.
\end{equation}
\end{thm}

\begin{proof}
Notice that $\mathsf N(B,x,F_n)$ equals either $\mathsf{\tilde
  N}(B,x,F_{n,K})$ (if $G$ contains an identity element) or at most
$\mathsf{\tilde N}(B,x,F_{n,K})+1$ otherwise. Now apply Remark
\ref{i-ii} (ii) and Lemma \ref{25} above.
\end{proof}

One of two main advantages of using the function
$B\mapsto\mathsf{\tilde N}(B,x,F)$ over $B\mapsto \mathsf N(B,x,F)$ is
its finite additivity on cylinders (even when the domains of the
blocks involved in the summation do not necessarily
coincide)\footnote{To see that the function $B\mapsto \mathsf
  N(B,x,F)$ is not finitely additive, recall that a word is an element
  of $\{0,1\}^{\{1,2,\dots,k\}}$, where $k$ is its length, and
  consider the words $u=\langle 0\rangle,\ v=\langle
  0,1\rangle\ w=\langle 0,0,0\rangle, y =\langle 0,0,1\rangle$. Notice
  that $[u]=[v]\cup[w]\cup[y]$ is a disjoint union of cylinders
  corresponding to words of lengths 2 and 3. Let
  $x=(0,0,0,0,0,\dots)\in\{0,1\}^\na$ and consider the set
  $F=\{1,2,3\}$. We have $\mathsf N(u,x,F)=3$, $\mathsf N(v,x,F)=0$,
  $\mathsf N(w,x,F)=1$, $\mathsf N(y,x,F)=0$, and $\mathsf
  N(u,x,F)\neq\mathsf N(v,x,F)+\mathsf N(w,x,F)+\mathsf N(y,x,F)$. }.

\begin{lem}\label{add}
Let $K_0,K_1,\dots,K_r$ be nonempty finite subsets of a countably
infinite semigroup $G$ and let $B_i\in\{0,1\}^{K_i}$, $i=0,1,\dots,r$,
be blocks such that $[B_0]=[B_1]\cup[B_2]\cup\cdots\cup[B_r]$ is a
disjoint union. Let $F\subset G$ be finite and let $x\in\{0,1\}^G$.
Then
\[
\mathsf{\tilde N}(B_0,x,F)= \sum_{i=1}^r\mathsf{\tilde N}(B_i,x,F).
\]
\end{lem}

\begin{proof}
Recall that for $i=0,1,2\dots,r$ we have $\mathsf{\tilde N}(B_i,x,F)=|\{g\in
F:\sigma_g(x)\in[B_i]\}|$, and notice that by the assumption
\[
\{g\in F:\sigma_g(x)\in[B_0]\} = \bigcup_{i=1}^r \{g\in
F:\sigma_g(x)\in[B_i]\}
\]
is a disjoint union. Since cardinality is a finitely additive
function, we are done.
\end{proof}

Another advantage of working with $\mathsf{\tilde
    N}(B,x,F)$ (rather than $\mathsf N(B,x,F)$) is that it can be
  represented as an \emph{ergodic sum} of the indicator function of
  $[B]$:
\begin{equation}\label{ergs}
\mathsf{\tilde N}(B,x,F) = \sum_{g\in F_n}\mathbbm
1_{[B]}(\sigma_g(x)).
\end{equation}
We will be referring to this interpretation later.

\medskip

We shall now fulfill the promise made in the introduction and prove
that the rather mild condition \eqref{44} forces the \sq\ $(F_n)$ to
be F{\o}lner.

\begin{thm}\label{26}
Let $(F_n)$ be a \sq\ of arbitrary finite subsets of a countably
infinite semigroup $G$. Suppose that there exists a 0-1-valued
function $x\in\{0,1\}^G$ such that for every two-element set $K\subset
G$ we have
\begin{equation}\label{asabove}
\lim_n\frac1{|F_n|}\sum_{B\in\{0,1\}^K}\mathsf N(B,x,F_n)=1
\end{equation}
(this assumption weakens and generalizes \eqref{44}). Then $(F_n)$ is a F{\o}lner
\sq\ in $G$.
\end{thm}

\begin{proof}
Let $x\in\{0,1\}^G$ satisfy \eqref{asabove} for every nonempty finite
$K\subset G$. Notice that, for any nonempty finite sets $K$ and $F$,
we have
\[
\sum_{B\in\{0,1\}^K}\mathsf N(B,x,F) = |F_K|.
\]

Now, let $g\in G$ be arbitrary. We fix some $h\in G$ and we let
$K=\{h,gh\}$. By assumption, given any $\varepsilon>0$, for $n$ large
enough we have $|F_{n,K}|\ge (1-\varepsilon)|F_n|$. For our particular
$K$, we have $F_{n,K}=h^{-1}F_n\cap {h^{-1}g^{-1}}F_n$, so, we obtain
\[
(1-\varepsilon)|F_n|\le |h^{-1}F_n\cap h^{-1}g^{-1}F_n| =
|h^{-1}(F_n\cap g^{-1}F_n)|\le |F_n\cap g^{-1}F_n|.
\]
This implies the F{\o}lner condition $\frac{|F_n\cap
  g^{-1}F_n|}{|F_n|}\to 1$.
\end{proof}

We remark that, without assuming the F{\o}lner condition, Theorem
\ref{tff} may fail very badly, i.e., the notion of $(F_n)$-normality
via the sets $\mathsf N(B,x,F_n)$ may differ drastically from the
notion involving the sets $\mathsf{\tilde N}(B,x,F_n)$. By Theorem
\ref{26}, if $(F_n)$ is not F{\o}lner, the set of $(F_n)$-normal
elements $x\in\{0,1\}^G$ (i.e., elements satisfying \eqref{fnbnorm})
is empty. On the other hand, taking for example $F_n=\{2,4,\dots,2n\}$
in $(\na,+)$ we see that even though $(F_n)$ is not a F{\o}lner \sq,
almost every element $x\in\{0,1\}^\na$ satisfies, for every word $w$,
the condition \eqref{tfnbnorm}: $\lim_n\frac1{|F_n|}\mathsf{\tilde
  N}(w,x,F_n)=2^{-|w|}$.

The following lemma shows that the formulas \eqref{fnnor} and
\eqref{fnbnor} in the Introduction lead to the same notion of
$(F_n)$-normality in $(\na,+)$.

\begin{lem}\label{nn}
Let $(F_n)$ be a F{\o}lner \sq\ in $(\na,+)$. If $x\in\{0,1\}^\na$ is
$(F_n)$-normal, i.e., satisfies \eqref{fnnor} (for words) then it
satisfies \eqref{fnbnor} (for blocks), i.e., for every nonempty finite
$K\subset\na$ and every block $B\in\{0,1\}^K$, we have
\[
\lim_n \frac1{|F_n|}\mathsf N(B,x,F_n)=2^{-|K|}.
\]
\end{lem}

\begin{proof}
By Theorem \ref{tff}, in \eqref{fnnor}, we can replace $\mathsf
N(w,x,F_n)$ by $\mathsf {\tilde N}(w,x,F_n)$, and in~\eqref{fnbnor} we
can replace $\mathsf N(B,x,F_n)$ by $\mathsf {\tilde N}(B,x,F_n)$. If
$B\in\{0,1\}^K$ then, letting $I$ be the shortest interval in $\na$ of
the form $\{1,2,\dots,r\}$, $r\in\na$, which contains $K$, we have the
disjoint union representation of the cylinder $[B]$:
\[
[B]=\bigcup_{w\in \{0,1\}^I,\ w|_K=B}[w].
\]
By Lemma \ref{add}, the function $[B]\mapsto
\lim_n\frac1{|F_n|}\mathsf{\tilde N}(B,x,F_n)$ is finitely additive on
cylinders for which the limits exist. By \eqref{fnnor}, for each
$w\in\{0,1\}^I$, we have
\[
\lim_n\frac1{|F_n|}\mathsf{\tilde
  N}(w,x,F_n)=\lim_n\frac1{|F_n|}\mathsf N(w,x,F_n)=2^{-|I|}.
\]
Thus,
\begin{multline*}
\lim_n \frac1{|F_n|}\mathsf N(B,x,F_n)=\lim_n \frac1{|F_n|}\mathsf
    {\tilde N}(B,x,F_n)=\\ \sum_{w\in\{0,1\}^I,\ w|_K=B} \lim_n
    \frac1{|F_n|}\mathsf {\tilde
      N}(w,x,F_n)=2^{|I|-|K|}\cdot2^{-|I|}=2^{-|K|}.
\end{multline*}
\end{proof}

The next theorem will allow us to reduce the proofs of some results
pertaining to countably infinite amenable cancellative semigroups to
the setup of countably infinite amenable groups (in particular, we
will be able to use the machinery of tilings). The fact given below appears independently in \cite[Corollary 2.12]{DFG19}. 

\begin{thm}\label{sgrvsgr}
Let $G$ ba a countably infinite amenable cancellative semigroup $G$. Then there exists an amenable group $\widetilde G$ containing $G$ as a subsemigroup, such that any
F{\o}lner \sq\ $(F_n)$ in $G$ is a F{\o}lner \sq\ in $\widetilde G$.
\end{thm}

\begin{proof}
First of all, any amenable cancellative semigroup is embeddable in a group $H$ (see \cite{Pa88}).  Then each element $g\in G$ has an inverse $g^{-1}\in H$. Let $\widetilde G\subset H$ be the set of all finite products $g_1g_2^{-1}\cdots g_{2k-1}g_{2k}^{-1}$, $k\in\na$,
where all the terms $g_i$ belong to $G\cup\{e\}$ ($e$ is the identity
element of $H$ and must be added only in case $G$ does not have an
identity element). Clearly, $\widetilde G$ is a subgroup of $H$ and it
contains $G$ (alternatively, $\widetilde G$ can be defined as the smallest
subgroup of $H$ containing $G$). Let $(F_n)$ be a F{\o}lner \sq\ in
$G$. Fix an element $\bar g\in \widetilde G$ and write it as a product
$g_1g_2^{-1}\cdots g_{2k-1}g_{2k}^{-1}$. Also fix an $\varepsilon>0$
and denote $\varepsilon'=\frac\varepsilon{2k}$. For large $n$, the set
$F_n$ is $(g_i,\varepsilon')$-invariant for each $i=1,2,\dots,2k$.
Then
\[
|F_n\triangle g_iF_n|=|F_n\triangle g^{-1}_iF_n|\le \varepsilon'|F_n|
\]
(the set $g_i^{-1}F_n$ is understood in $\widetilde G$). Mutiplying both
sets in $F_n\triangle g^{-1}_iF_n$ by $g_j$ on the left (for some
$j\in\{1,2,\dots,2k\}$), we obtain
\[
|g_jF_n\triangle g_jg^{-1}_iF_n|\le \varepsilon'|F_n|.
\]
Now, by the triangle inequality for the metric $|\cdot\triangle\cdot|$
(on finite sets), we have
\[
|F_n\triangle g_jg^{-1}_iF_n|\le |F_n\triangle
g_jF_n|+|g_jF_n\triangle g_jg^{-1}_iF_n|\le 2\varepsilon'|F_n|.
\]
Repeating this argument $k$ times (with the appropriate order of
indices) we get
\[
|F_n\triangle\bar gF_n|=|F_n\triangle g_1g_2^{-1}\cdots
g_{2k-1}g_{2k}^{-1}F_n|\le 2k\varepsilon'|F_n|=\varepsilon|F_n|,
\]
which means that $F_n$ is $(\bar g,\varepsilon)$-invariant.
We have shown that $(F_n)$ is a F{\o}lner \sq\ in $\widetilde G$. This also
implies that the group $\widetilde G$ is amenable.
\end{proof}

In this context we have the following fact.

\begin{lem}\label{ll}
Let $(F_n)$ be a F{\o}lner \sq\ in a countably infinite amenable
semigroup $G$ which is embeddable in a group $\widetilde G$ such that
$(F_n)$ is a F{\o}lner \sq\ in $\widetilde G$. A subset $A\subset G$ is
$(F_n)$-normal in $G$ if and only if it is $(F_n)$-normal, viewed as a
subset of the group $\widetilde G$ (in other words, $\mathbbm
1_A\subset\{0,1\}^G$ is $(F_n)$-normal if and only if $\mathbbm
1_A\subset\{0,1\}^{\widetilde G}$ is $(F_n)$-normal).
\end{lem}

\begin{proof} In the proof we will use Theorem \ref{defn1.2} which will be proved later and is independent from Lemma \ref{ll}.
If $A$ is $(F_n)$-normal in $\widetilde G$ then clearly it is $(F_n)$-normal
in $G$ (because every nonempty finite set $K\subset G$ is also a
subset of $\widetilde G$). Suppose $A$ is $(F_n)$-normal in $G$ and let $K$
be a nonempty finite subset of $\widetilde G$. For large enough $n_0$, the
intesection $F_{n_0}\cap F_{n_0,K}$ is nonempty, i.e., there exists
$g\in F_n\subset G$ and a bijection $h\mapsto f_h$ from $K$ onto some
$K'\subset F_{n_0}$, such that $hg=f_h$ for each $h\in K$. Then, in
the group $\widetilde G$, we have
\[
d_{(F_n)}\Bigl(\bigcap_{h\in
  K}h^{-1}A\Bigr)=d_{(F_n)}\Bigl(g^{-1}\bigcap_{h\in
  K}h^{-1}A\Bigr)=d_{(F_n)}\Bigl(\bigcap_{h\in
  K}f_h^{-1}A\Bigr)=d_{(F_n)}\Bigl(\bigcap_{f\in K'}f^{-1}A\Bigr).
\]
The meaning of $f^{-1}A$ is different in $\widetilde G$ and in $G$ (in $G$
it means $f^{-1}A\cap G$), however, since the sets $F_n$ are contained
in $G$, the value of $d_{(F_n)}\Bigl(\bigcap_{f\in K'}f^{-1}A\Bigr)$
does not depend on whether it is considered in $\widetilde G$ or in $G$.
Since $K'\subset G$ and $A$ is $(F_n)$-normal as a subset of $G$, the
equivalence (1)\,$\iff$\,(3) in Theorem \ref{defn1.2} and formula
\eqref{1.2} yield that $d_{(F_n)}\Bigl(\bigcap_{h\in
  K}h^{-1}A\Bigr)=2^{-|K'|}=2^{-|K|}$. By invoking Theorem
\ref{defn1.2} again, we obtain $(F_n)$-normality of $A$ as a subset of
$\widetilde G$.
\end{proof}

Throughout the remainder of this section we assume that $G$ is a
countably infinite amenable \emph{group}. Our key tool for handling
$(F_n)$-normality in $G$ is a special \emph{system of tilings}
$(\mathcal T_k)_{k\ge 1}$ of $G$ which was constructed in
\cite{DHZ17}. (We could employ instead an older concept of quasi-tilings
introduced in \cite{OW87}, but the system $(\mathcal T_k)$ is a more
convenient tool for our purposes.)

Let $\mathcal S$ be a collection of finite subsets of $G$, each
containing the identity element, which we will call \emph{shapes}. To
each $S\in\mathcal S$ we associate a \emph{set of translates} (of
$S$), $C_S\subset G$. We require that the sets $C_S$ be pairwise
disjoint and write $\mathcal C=\{C_S:S\in\mathcal S\}$. If the family
\[
\mathcal T=\{Sc: S\in\mathcal S, c\in C_S\},
\]
is a partition of $G$, we call it the \emph{tiling} of $G$ associated
with the pair $(\mathcal S,\mathcal C)$.

An element $Sc$ of this partition will be called a \emph{tile of shape
  $S$ centered at $c$}. By disjointness of the tiles, the assignment
$(S,c)\mapsto Sc$ is a bijection from $\{(S,c): S\in\mathcal S, c\in
C_S\}$ to $\mathcal T$, i.e., each tile has a uniquely determined
center and shape.

Given a tiling $\mathcal T$ and a set $F\subset G$, the
\emph{$\mathcal T$-saturation} of $F$ is defined as
\[
F^{(\mathcal T)}=\bigcup\{Sc\in\mathcal T:Sc\cap F\neq\emptyset\}.
\]
Let $K$ be the union of the sets $SS^{-1}$ over all shapes $S$ of
those tiles of $\mathcal T$ which have nonempty intersections with $F$.
Formally,
\[
K=\bigcup\{SS^{-1}: S\in\mathcal S, (\exists c\in C_S) Sc\cap
F\neq\emptyset\}.
\]
It is an easy observation that if $F$ is a finite and
$(K,\varepsilon)$-invariant set, then $|F^{(\mathcal T)}\setminus
F|\le\varepsilon|F|$.

A tiling whose set of shapes is finite will be called \emph{proper}.

A \sq\ of proper tilings $(\mathcal T_k)_{k\ge 1}$ is called \emph{a
  congruent system of tilings} if for each $k$ every tile of $\mathcal
T_{k+1}$ is a union of some tiles of $\mathcal T_k$.

A congruent system of tilings is \emph{deterministic}, if, for each
$k\ge 1$, all tiles of $\mathcal T_{k+1}$ having the same shape are
partitioned into the tiles of $\mathcal T_k$ the same way. More
precisely, we require that whenever $T'_1=S'c_1$ and $T'_2=S'c_2$ are
two tiles of $\mathcal T_{k+1}$ of the same shape $S'$ (note that then
$c_1,c_2\in \mathcal C_{S'}$) and $T'_1=\bigcup_{i=1}^l T_{1,i}$ is
the partition of $T'_1$ into the tiles of $\mathcal T_k$, then the
sets $T_{2,i}=T_{1,i}c_1^{-1}c_2$ (with $i=1,2,\dots,l$) are also
tiles of $\mathcal T_k$ (and clearly they partition $T'_2$). It
follows that in the deterministic case, the tiling $\mathcal T_{k+1}$
\emph{determines} all the tilings $\mathcal T_1,\dots,\mathcal T_k$.
Also note that, with the above notation, the family
$\{T_{1,i}c_1^{-1}:i=1,2\dots,l\}$ (which is the same as
$\{T_{2,i}c_2^{-1}:i=1,2\dots,l\}$) is a partition of the shape $S'$
into shifted shapes of the tiling $\mathcal T_k$. We will call this
partition the \emph{standard tiling of $S'$ by the tiles of $\mathcal
  T_k$} (although formally, the sets $T_{1,i}c_1^{-1}$ need not be
tiles of $\mathcal T_k$).

We will say that a system of proper tilings $(\mathcal T_k)_{k\ge 1}$
is \emph{F{\o}lner} if for every nonempty finite set $K\subset G$
and every $\varepsilon>0$, for large enough $k$, all shapes of $\mathcal
T_k$ (and thus also all tiles) are $(K,\varepsilon)$-invariant (in
other words, if $(S_j)_{j\in\na}$ is obtained by enumerating the
collection $\bigcup_k\mathcal S_k$ of all shapes used in the system of
tilings, then $(S_j)$ is a F{\o}lner \sq).

A proper tiling is called \emph{syndetic} if for every shape $S$ the
set of translates $C_S$ is (left) syndetic, i.e., such that $KC_S=G$
for some finite set $K$ (depending on $S$).

It is proved in \cite{DHZ17} that every countably infinite amenable
group $G$ admits a congruent, deterministic, F{\o}lner system of proper
tilings $(\mathcal T_k)_{k\ge 1}$. One can actually obtain a system of
\emph{syndetic} tilings with all the above properties, as follows.
First, as noted in \cite{DHZ17}, any proper tiling $\mathcal T$ of $G$
can be represented as an element of the symbolic space $(\mathcal
S\cup\{0\})^G$ (where the role of the alphabet is played by the finite
collection $\mathcal S$ of shapes of $\mathcal T$ with an additional
symbol $0$). Now, a system of proper tilings $(\mathcal T_k)_{k\ge 1}$
becomes an element $\mathbf T$ of the space $\prod_{k\ge 1}(\mathcal
S_k\cup\{0\})^G = (\prod_{k\ge 1}(\mathcal S_k\cup\{0\}))^G$, on which
$G$ acts by shifts. Let $\overline O(\mathbf T)$ denote the orbit
closure of $\mathbf T$ with respect to the shift action. Every element
$\mathbf T'\in\overline O(\mathbf T)$ is again a system of proper
tilings $(\mathcal T'_k)_{k\ge 1}$ with the respective sets of shapes
$\mathcal S'_k$ satisfying $\mathcal S'_k\subset\mathcal S_k$ for each
$k$.\footnote{In fact, if for every element $\mathcal T'_k\in\overline
  O(\mathcal T_k)$ we have $\mathcal S'_k=\mathcal S_k$ then $\mathcal
  T_k$ is already syndetic.} Note that if $\mathbf T$ is a F{\o}lner
system of tilings, so is every $\mathbf T'\in\overline O(\mathbf T)$.
Also, the properties of being congruent and deterministic pass from
$\mathbf T$ to all members of $\overline O(\mathbf T)$. The system
$\overline O(\mathbf T)$ (with the shift action) has a minimal
subsystem. Any element of this minimal subsystem, in addition to the
preceding properties, is a system of syndetic tilings, which follows
by a standard characterization of minimality in symbolic dynamics.

We define $\mathcal T_0$ to be the tiling all tiles of which are
singletons ($\mathcal T_0$ has one shape $S=\{e\}$ and the
corresponding set of translates $C_S$ is the whole group).

\section{Left invariance of $(F_n)$-normality}\label{two}

We call a subset $A\subset G$ $(F_n)$-normal if its indicator function
$\mathbbm 1_A$, viewed as an element of $\{0,1\}^G$, is
$(F_n)$-normal. The goal of this section is to prove that if $G$ is a
countably infinite amenable cancellative semigroup and $(F_n)$ is a F{\o}lner \sq\ in $G$ then a set $A\subset G$ is $(F_n)$-normal if and only if so is $gA$, and also if and only if so is $g^{-1}A$.

The following theorem provides a characterization of normal sets in
terms of ``combinatorial independence''.

\begin{thm}\label{defn1.2}
Let $G$ be a countably infinite amenable cancellative semigroup and
let $(F_n)$ be a F{\o}lner sequence in $G$. Let $A\subset G$. We will
use the following notation: $A^1=A$, $A^0=G\setminus A$. Consider the
following five conditions:
\begin{enumerate}
\item $A$ is $(F_n)$-normal,
\item for any nonempty finite set $K$ and any 0-1 block
  $B\in\{0,1\}^K$ we have
\begin{equation}\label{1.0}
d_{(F_n)}\Bigl(\bigcap_{h\in K}h^{-1}A^{B(h)}\Bigr) = 2^{-|K|},
\end{equation}
\item for any nonempty finite set $K$ we have
\begin{equation}\label{1.2}
d_{(F_n)}\Bigl(\bigcap_{h\in K}h^{-1}A\Bigr) = 2^{-|K|},
\end{equation}
\item for any nonempty finite set $K$ and any 0-1 block
  $B\in\{0,1\}^K$ we have
\begin{equation}\label{1.3}
d_{(F_n)}\Bigl(\bigcap_{h\in K}hA^{B(h)}\Bigr) = 2^{-|K|},
\end{equation}
\item for any nonempty finite set $K$ we have
\begin{equation}\label{1.4}
d_{(F_n)}\Bigl(\bigcap_{h\in K}hA\Bigr) = 2^{-|K|}.
\end{equation}
\end{enumerate}
Then (1)$\iff$(2)$\iff$(3)$\implies$(4)$\iff$(5). If $G$ is a group or
$G$ is commutative then all conditions (1)--(5) are equivalent.
\end{thm}

\begin{rem}
If we enumerate the set $K$ as $\{h_1,h_2,\dots,h_k\}$ then the blocks
$B\in\{0,1\}^K$ stand in 1-1 correspondence to 0-1 words of length
$k$. Then, we can rewrite conditions (2)--(5) as follows
\begin{enumerate}
\item[(2)] for any nonempty finite set $K=\{h_1,h_2,\dots,h_k\}$ and
  any 0-1 word $w$ of length $k$, we have
\begin{equation*}
d_{(F_n)}(h_1^{-1}A^{w_1}\cap h_2^{-1}A^{w_2}\cap\cdots\cap
h_k^{-1}A^{w_k}) = 2^{-k},
\end{equation*}
\item[(3)] for any nonempty finite set $K=\{h_1,h_2,\dots,h_k\}$ we have
\begin{equation*}
d_{(F_n)}(h_1^{-1}A\cap h_2^{-1}A\cap\cdots\cap h_k^{-1}A) = 2^{-k},
\end{equation*}
\item[(4)] for any nonempty finite set $K=\{h_1,h_2,\dots,h_k\}$ and
  any 0-1 word $w$ of length $k$, we have
\begin{equation*}
d_{(F_n)}(h_1A^{w_1}\cap h_2A^{w_2}\cap\cdots\cap h_kA^{w_k}) = 2^{-k},
\end{equation*}
\item[(5)] for any nonempty finite set $K=\{h_1,h_2,\dots,h_k\}$ we
  have
\begin{equation*}
d_{(F_n)}(h_1A\cap h_2A\cap\cdots\cap h_kA) = 2^{-k}.
\end{equation*}
\end{enumerate}
\end{rem}

\begin{proof}[Proof of Theorem \ref{defn1.2}]
In view of Theorem \ref{tff}, $(F_n)$-normality can be defined via the
condition \eqref{tfnbnorm}. Observe that
\[
g\in\bigcap_{h\in K}h^{-1}A^{B(h)}\iff(\forall_{h\in K})\ hg\in
A^{B(h)}\iff\sigma_g(\mathbbm 1_A)\in[B].
\]
Thus \eqref{1.0} is just \eqref{tfnbnorm} written in terms of
$(F_n)$-density, which immediately gives the equivalence (1)$\iff$(2).
Next, (2) implies (3) because \eqref{1.2} is the particular case of
\eqref{1.0} for the block $B$ equal to the constant function $1$ on
$K$. By the same argument (4) implies (5).

We pass to proving that (3)$\implies$(2). Suppose that some two sets
$A_1,A_2\subset G$ have well defined $(F_n)$-densities and satisfy the
``independence condition'':
\[
d_{(F_n)}(A_1\cap A_2)=d_{(F_n)}(A_1)\cdot d_{(F_n)}(A_2).
\]
Then, by finite additivity of $(F_n)$-density, we have
\begin{multline*}\label{3.5}
d_{(F_n)}(A_1\cap A_2^0)=
d_{(F_n)}(A_1)-d_{(F_n)}(A_1\cap A_2)=\\
d_{(F_n)}(A_1)-d_{(F_n)}(A_1)\cdot
d_{(F_n)}(A_2)=d_{(F_n)}(A_1)(1-d_{(F_n)}(A_2))=\\
d_{(F_n)}(A_1)\cdot d_{(F_n)}(A_2^0).
\end{multline*}
Iterating the above calculation one shows that if a finite family
$\{A_1,A_2,\dots,A_k\}$ of subsets of $G$ satisfies the ``independence
condition'':
\begin{itemize}
\item for any subset $E\subset\{1,2,\dots,k\}$ one has
  $d_{(F_n)}\Bigl(\bigcap_{i\in E}A_i\Bigr) = \prod_{i\in
  E}d_{(F_n)}(A_i)$,
\end{itemize}
then, for any 0-1-word $w\in\{0,1\}^k$, the family
$\{A^{w_1}_1,A^{w_2}_2,\dots,A^{w_k}_k\}$ also satisfies the
independence condition. Next, notice that condition (3) applied to
all possible nonempty subsets of $K$ is precisely the independence
condition for the family $\{h^{-1}A:h\in K\}$. This, combined with the
preceding observation, implies (2).

The same argument proves the implication (5)$\implies$(4).

To prove the implication (3)$\implies$(5), we note that for large $n$
the ``$K^{-1}$-core'' of $F_n$, i.e., the set $\bigcap_{h\in K}hF_n$
is nonempty (like the $K$-core, it is eventually an
$\varepsilon$-modification of $F_n$). Thus there exists an
$n_0\in\na$, a $g\in G$ and a bijection $h\mapsto f_h$ from $K$ onto
some $K'\subset F_{n_0}$, such that $g=hf_h$ for each $h\in K$. By
\eqref{transl}, the $(F_n)$-density of $\bigcap_{h\in K}hA$ is the
same as that of $g^{-1}\bigcap_{h\in K}hA=\bigcap_{h\in K}g^{-1}hA =
\bigcap_{h\in K}f_h^{-1}A = \bigcap_{f\in K'}f^{-1}A$. By~(3), this
density equals $2^{-|K'|}=2^{-|K|}$, as needed.

If $G$ is a group then the equivalence (3)$\iff$(5) is obvious: the
family $\{hA:h\in K\}$ is the same as $\{h^{-1}A:h\in K^{-1}\}$.

Suppose $G$ is commutative and assume (4). Let $K$ be a nonempty
finite subset of $G$. As before, there exists $g\in G$ and a bijection
$h\mapsto f_h$ from $K$ onto some $K'\subset G$, such that $g=hf_h$
for each $h\in K$. By \eqref{transl} we have
\begin{equation*}\label{1.5}
d_{(F_n)}\Bigl(\bigcap_{h\in K}h^{-1}A^{B(h)}\Bigr) =
d_{(F_n)}\Bigl(\bigcap_{h\in K}gh^{-1}A^{B(h)}\Bigr).
\end{equation*}
We would like to replace $gh^{-1}$ by $f_h$ (using commutativity),
however, in general $gh^{-1}A$ is only a subset of $h^{-1}gA=f_hA$ (an
analogous inclusion holds for $A^0=G\setminus A$). Thus
\[
\frac1{|F_n|}\Bigl|F_n\cap\bigcap_{h\in K}h^{-1}A^{B(h)}\Bigr|\le
\frac1{|F_n|}\Bigl|F_n\cap\bigcap_{h\in
  K}f_hA^{B(h)}\Bigr|=\frac1{|F_n|}\Bigl|F_n\cap\bigcap_{f\in
  K'}fA^{B'(f)}\Bigr|,
\]
where $B'$ is defined on $K'$ by $B'(f_h)=B(h)$. By (4), the right
hand side tends to $2^{-|K'|}=2^{-|K|}$. But since the sum of the left
hand sides over all blocks $B\in\{0,1\}^K$ equals 1, we have
convergence of the left hand side to $2^{-|K|}$ for every block, i.e.,
\eqref{1.0}. We have proved that (4)$\implies$(2).
\end{proof}

\begin{thm}\label{linv}
Let $G$ be a countably infinite amenable cancellative semigroup and
let $(F_n)$ be a F{\o}lner \sq\ in $G$. For any $A\subset G$ and $g\in
G$ we have the equivalences
\begin{equation}\label{imp}
\text{$gA$ is $(F_n)$-normal $\iff$ $A$ is $(F_n)$-normal
  $\iff$ $g^{-1}A$ is $(F_n)$-normal.}
\end{equation} 
\end{thm}

\begin{proof}
The first half of the proof relies on the equivalence (1)$\iff$(3) in
Theorem~\ref{defn1.2}. Assume that $gA$ is $(F_n)$-normal and let $K$
be a nonempty finite subset of $G$. Condition \eqref{1.2} applied
to $K'=gK$ and the set $gA$ reads
\[
2^{-|K|}= d_{(F_n)}\Bigl(\bigcap_{h\in K}(gh)^{-1}gA\Bigr) =
d_{(F_n)}\Bigl(\bigcap_{h\in K}h^{-1}A\Bigr),
\]
i.e., we have obtained \eqref{1.2} for $K$ and $A$. Now assume that
$A$ is $(F_n)$-normal and let $K$ be a nonempty finite subset of $G$.
Condition \eqref{1.2} applied to $K'=gK$ and the set $A$ is
\[
2^{-|K|}= d_{(F_n)}\Bigl(\bigcap_{h\in K}(gh)^{-1}A\Bigr) =
d_{(F_n)}\Bigl(\bigcap_{h\in K}h^{-1}g^{-1}A\Bigr),
\]
which gives \eqref{1.2} for $K$ and $g^{-1}A$.

If $G$ is a group or is commutative then we can use the equivalence
(1)$\iff$(4) to reverse the implications: Assume that $g^{-1}A$ is
$(F_n)$-normal and let $K$ be a nonempty finite subset of $G$ and
$B\in\{0,1\}^K$. We have
\[
\frac1{|F_n|}\Bigl|F_n\cap\bigcap_{h\in K}hg^2g^{-1}A^{B(h)}\Bigr|\le
\frac1{|F_n|}\Bigl|F_n\cap\bigcap_{h\in K}hgA^{B(h)}\Bigr|.
\]
Te condition \eqref{1.3} applied to $K'=Kg^2$, the block
$B'\in\{0,1\}^{Kg^2}$ defined by $B'(hg^2)=B(h)$, and the set
$g^{-1}A$ implies that the left hand side tends to $2^{-|K|}$. But
since the sum of the right hand sides over all blocks $B\in\{0,1\}^K$
equals 1, we have convergence of the right hand side to $2^{-|K|}$ for
every block, i.e., \eqref{1.3} holds for $gA$. Since (4)$\implies$(1)
we have proved $(F_n)$-normality of $gA$.

Finally, we can use the fact that $G$ is cancellative. By Theorem \ref{sgrvsgr}, 
it can be embedded in a group $\widetilde G$ such that $(F_n)$ is a F{\o}lner \sq\ in $\widetilde G$. Suppose $g^{-1}A$ is
$(F_n)$-normal as a subset of $G$. By definition, the set $g^{-1}A$
regarded as a subset of $G$ is equal, in $\widetilde G$, to $g^{-1}A\cap G$.
By Lemma~\ref{ll}, $g^{-1}A\cap G$ is $(F_n)$-normal as a subset of
$\widetilde G$. Since all sets $F_n$ are contained in $G$, also the set
$g^{-1}A$ is $(F_n)$-normal in $\widetilde G$. In the group $\widetilde G$,
$(F_n)$-normality of $g^{-1}A$ implies $(F_n)$-normality $gA$.
Finally, by the trivial direction of Lemma~\ref{ll}, $gA$ is also
$(F_n)$-normal when viewed as a subset of $G$.
\end{proof}

\begin{rem}
We were unable to prove the implication (4)$\implies$(2) in Theorem
\ref{defn1.2} for semigroups embeddable in groups.
\end{rem}

\begin{rem}\label{3.6}
In general, even if $G$ is a group, $(F_n)$-normality is {\bf not}
right invariant: if $A$ is $(F_n)$-normal then $Ag$ is not guaranteed
to be $(F_n)$-normal\footnote{For instance, a
  counterexample can be constructed in the group
  $G=\langle\sigma,\tau\rangle$ of transformations of the symbolic
  space $\{0,1\}^\z$, generated by the shift $\sigma$ and the flip
  $\tau$ of the zero-coordinate symbol (note that $\tau^{-1}=\tau$).
  This group is solvable: the subset
  $H=\langle\sigma^{-k}\tau\sigma^k:k\in\z\rangle$ (consisting of
  flips at finitely many coordinates, with no shift) is a normal
  subgroup of $G$ and $G/H=\langle\sigma\rangle$ is Abelian. In
  particular, $G$ is amenable. Each $g\in G$ is representable in a
  unique way as $\sigma^{k_g}h_g$ with $h_g\in H$. For each $h\in H$
  denote by $m_h\in\z$ the rightmost coordinate on which $h$ applies
  the flip. Let $(F'_n)$ be a F{\o}lner \sq\ in $G$. Let
  $m_n=\max\{m_{h_g}:g\in F_n'\}$. Now we create a new F{\o}lner
  \sq\ $(F_n)$ by setting $F_n=F_n'\sigma^{m_n+1}$. Notice that any
  $g\in F_n$ does not flip the zero-coordinate symbol (but perhaps
  shifts it). This implies that $F_{n_1}$ and $F_{n_2}\tau$ are
  disjoint for any $n_1,n_2\in\na$. As we know, there exist an
  $(F_n)$-normal set $A'\subset G$ and its intersection with the union
  $A=\bigcup_nF_n$ is also $(F_n)$-normal. The set $A$ is disjoint
  from $F_n\tau$ for all $n\ge 1$, which implies that $A\tau$ has
  $(F_n)$-density zero and hence cannot be $(F_n)$-normal.}. For
this reason, $(F_n)$-normality of the elements of $\{0,1\}^G$ is {\bf
  not} preserved by the shift-action: $\sigma_g(x)$ need not be
$(F_n)$-normal if $x$ is (if $x$ is the indicator function of a set
$A$ then $\sigma_g(x)$ is the indicator function of $Ag^{-1}$).
Nevertheless, under very mild assumptions on $(F_n)$, this may happen
only with probability zero, see Corollary \ref{totnor} below.
\end{rem}

\section{Properties of the family of $(F_n)$-normal sets}\label{for}

\subsection{Ergodic interpretation of normality}\label{three}

Fix a countably infinite amenable cancellative semigroup $G$ and a
F{\o}lner \sq\ $(F_n)$ in $G$. Suppose that $G$ acts by continuous maps
$T_g$ on a compact metric space $X$, preserving a Borel probability
measure $\mu$. We will tacitly assume that, when convenient or
necessary, the identity element (always denoted by $e$) is attached to
the semigroup, and $T_e$ is the identity mapping. A point $x\in X$ is
called \emph{$(F_n)$-generic for $\mu$} if for any continuous function
$f\in C(X)$ one has
\begin{equation}\label{cesaro}
\lim_{n\to\infty}\frac1{|F_n|}\sum_{g\in F_n}f(T_gx)=\int f\,d\mu,
\end{equation}
in other words, if the measures $\frac1{|F_n|}\sum_{g\in
  F_n}\delta_{T_g x}$ converge to $\mu$ in the weak-star topology.

Note that the shift action on the symbolic space $X=\{0,1\}^G$ preserves
(among many other measures) the product measure $m^G$, where $m$ is
the $(\frac12,\frac12)$-measure on $\{0,1\}$. The measure $m^G$ will
be henceforth denoted by $\lambda$ and called the (uniform) Bernoulli
measure.

We have the following equivalent formulation of normality of a set
$A\subset G$, in dynamical terms.

\begin{prop}\label{ro}
A set $A\subset G$ is $(F_n)$-normal if and only if its indicator
function $\mathbbm 1_A$ is $(F_n)$-generic for the Bernoulli measure
$\lambda$ on $\{0,1\}^G$.
\end{prop}

\begin{proof}
First of all, note that for any nonempty finite set $K\subset G$ and
any $B\in\{0,1\}^K$, we have $2^{-|K|}=\lambda([B])$. Thus, using
\eqref{tfnbnorm} and \eqref{ergs}, we can see that the
$(F_n)$-normality of $x$ can be equivalently expressed by the
condition \eqref{cesaro} (with $\mu=\lambda$ and $T_g=\sigma_g$) for
all functions of the form $f=\mathbbm 1_{[B]}$. Finally, observe that
the indicator functions of cylinders are linearly dense in the space
$C(\{0,1\}^G)$ of continuous functions on $\{0,1\}^G$, which clearly
ends the proof.
\end{proof}

As was already mentioned in the Introduction, the existence of
$(F_n)$-normal 0-1 \sq s (and the fact that the set of such \sq s has
full measure) is often derived with the help of the pointwise ergodic
theorem, which, in general, holds only along rather special (tempered)
F{\o}lner \sq s. However, in the specific case of the Bernoulli measure
and continuous functions, the conventional pointwise ergodic theorem
can be replaced by Theorem~\ref{B1} below (more precisely, by its
equivalent version Theorem~\ref{co}), which is valid under much weaker
restrictions on F{\o}lner \sq s.

\begin{thm}\label{B1}
Let $G$ be a countably infinite amenable cancellative semigroup. Let $(F_n)_{n\ge 1}$ be a F{\o}lner \sq\ in $G$
such that for any $\alpha\in(0,1)$ we have $\sum_{n=1}^\infty
\alpha^{|F_n|}<\infty$. Then $\lambda$-almost every $x\in\{0,1\}^G$ is
$(F_n)$-normal, i.e., for any nonempty finite set $K\subset G$ and any
block $B\in\{0,1\}^K$, one has
\begin{equation}\label{nnn}
\lim_{n\to\infty}\frac1{|F_n|}{|\{g\in F_n:\sigma_g(x)\in [B]\}|}=2^{-|K|}.
\end{equation}
\end{thm}

\begin{proof}
By Theorem \ref{sgrvsgr} and Lemma \ref{ll}, it suffices to consider
the case where $G$ is a group. Because there are countably many blocks
over finite subsets of $G$, it suffices to prove that for any nonempty
finite set $K\subset G$ and any block $B\in\{0,1\}^K$, \eqref{nnn}
holds for $\lambda$-almost every $x\in\{0,1\}^G$.

Given $\varepsilon>0$, we will partition the group $G$ into finitely
many sets $D_0,D_1,\dots,D_r$, such that $\overline
d_{(F_n)}(D_0)\le\varepsilon$ (the set $D_0$ may be empty), and for
every $i>0$, we have
\begin{enumerate}
	\item $\underline d_{(F_n)}(D_i)>0$,
	\item for all distinct $g_1,g_2\in D_i$, $Kg_1\cap
          Kg_2=\emptyset$.
\end{enumerate}

We start by showing that the existence of the sets $D_0,D_1,\dots,D_r$
as above implies the assertion of the theorem. Choose a positive
$\beta<\min\{\underline d_{(F_n)}(D_i),i=1,2,\dots,r\}$. Let $n_0$ be
such that for every $n\ge n_0$,
\[
\frac{|F_n\cap D_0|}{|F_n|}<2\varepsilon,\ \text{ and, for each }
i\in\{1,2,\dots,r\},\ \ \frac{|F_n\cap D_i|}{|F_n|}>\beta.
\]

Let $\Omega=(\{0,1\}^G,\mathcal B,\lambda)$ where $\mathcal B$ denotes
the Borel $\sigma$-algebra in $\{0,1\}^G$. Fix an $n\ge n_0$ and
consider the finite \sq\ of $\{0,1\}$-valued random variables defined
on $\Omega$ by
\[
\mathsf Y_g(x)=\mathbbm1_{[B]}(\sigma_g(x)), \ \ g\in F_n.
\]
Also, for each $i=0,1,\dots,r$ define
\[
\bar{\mathsf Y}_i = \frac1{|F_n\cap D_i|}\sum_{g\in F_n\cap
  D_i}\mathsf Y_g.
\]
By (2), for each $i>0$ the variable $\bar{\mathsf Y}_i$ is the average
of finitely many independent random variables $\mathsf Y_g$, each
assuming the value $1$ with probability $2^{-|K|}$. Clearly, the
expected value of $\bar{\mathsf Y}_i$ equals $2^{-|K|}$. Now, the
classical Bernstein's inequality (see, e.g., \cite{Ber27}) implies that
\[
\lambda(\{x:|\bar{\mathsf Y}_i(x)-2^{-|K|}|>\varepsilon\}) \le
\gamma^{|F_n\cap D_i|}<\gamma^{\beta|F_n|},
\]
where $\gamma\in(0,1)$ is some constant (not depending on $n$). Then,
denoting by $X_\varepsilon=\{x:\exists{i=1,2,\dots,r}:|\bar{\mathsf
  Y}_i(x)-2^{-|K|}|>\varepsilon\}$, we have
\begin{multline*}
\lambda(X_\varepsilon) =
\lambda(\bigcup_{i=1,2,\dots,r}\{x:|\bar{\mathsf
  Y}_i(x)-2^{-|K|}|>\varepsilon\})\le\\ \sum_{i=1,2,\dots,r}\lambda(\{x:|\bar{\mathsf
  Y}_i(x)-2^{-|K|}|>\varepsilon\})\le r\gamma^{\beta|F_n|}.
\end{multline*}
On the complementary set $\{0,1\}^G\setminus X_\varepsilon$, for each
$i=1,2,\dots,r$, we have the inequality $|\bar{\mathsf
  Y}_i(x)-2^{-|K|}|\le\varepsilon$, i.e.,
\begin{equation}\label{lili}
\frac1{|F_n\cap D_i|}\sum_{g\in F_n\cap D_i}\mathsf Y_g(x)\in
[2^{-|K|}-\varepsilon,2^{-|K|}+\varepsilon].
\end{equation}
For $i=0$, recall that $\frac{|F_n\cap D_0|}{|F_n|}<2\varepsilon$, and
we have the trivial estimate
\begin{equation}\label{lolo}
\frac1{|F_n\cap D_0|}\sum_{g\in F_n\cap D_0}\mathsf Y_g(x)\in[0,1].
\end{equation}
Averaging the left hand sides of \eqref{lili} and \eqref{lolo} over
$i=0,1,2,\dots,r$ (with weights $\frac{|F_n\cap D_i|}{|F_n|}$) we
obtain
\[
\frac1{|F_n|}\sum_{g\in F_n}\mathsf Y_g(x)\in
      [2^{-|K|}-3\varepsilon,2^{-|K|}+3\varepsilon].
\]
So, the set on which the inequality
\[
\left|\frac1{|F_n|}\sum_{g\in F_n}\mathsf
Y_g(x)-2^{-|K|}\right|>3\varepsilon
\]
holds is contained in $X_\varepsilon$, thus has measure at most
$r\gamma^{\beta|F_n|}$. Summarizing, we have shown that
\[
\lambda\left\{\left|\frac{|\{g\in F_n:\sigma_g(x)\in [B]\}|}{|F_n|} -
2^{-|K|}\right|> 3\varepsilon\right\}\le r\gamma^{\beta|F_n|}.
\]
Let $\alpha=\gamma^\beta$ and note that $\alpha\in(0,1)$. By the
assumption, $\sum_n \alpha^{|F_n|}<\infty$. The Borel-Cantelli Lemma
now yields that for $\lambda$-almost every $x$, the numbers
\[
\frac{|\{g\in F_n:\sigma_g(x)\in
  [B]\}|}{|F_n|}=\frac1{|F_n|}\mathsf{\tilde N}(B,x,F_n)
\]
eventually remain within $3\varepsilon$ from $2^{-|K|}$. Since
$\varepsilon$ is arbitrary, we have proved the desired almost
everywhere convergence.

It remains to define the sets $D_i$. We will do that with the help of
tilings. As we have mentioned earlier, $G$ admits a congruent,
deterministic, F{\o}lner system of proper, syndetic tilings $(\mathcal
T_k)_{k\ge 1}$. Let $k$ be such that all shapes $S\in\mathcal S$ of
the tiling $\mathcal T=\mathcal T_k=(\mathcal S,\mathcal C)$ are
$(K,\delta)$-invariant, where $\delta =\frac\varepsilon{2|K|}$. Then,
for each $S\in\mathcal S$, the \emph{$K$-core} of $S$, i.e., the set
$S_K=\{g\in G:Kg\subset S\}$ satisfies
\[
\frac{|S_K|}{|S|}\ge 1-\varepsilon
\]
(see Lemma \ref{estim} and notice since that $K$
contains the identity element, we have $S_K\subset S$). Also, if $T$
is any tile of $\mathcal T$ and $T_K$ denotes the $K$-core of $T$ then
\[
\frac{|T_K|}{|T|}\ge 1-\varepsilon
\]
(recall that $T=Sc$ where $S\in\mathcal S$, $c\in C_S$, in which case
$T_K=S_Kc$). Let now
\[
D_0 =\bigcup_{T\in\mathcal T} T\setminus T_K.
\]
We claim that $\overline
d_{(F_n)}(D_0)\le\varepsilon$. Indeed, this inequality is obvious if
  the F{\o}lner \sq\ $(F_n)$ is replaced by the \sq\ $(F_n^{(\mathcal
    T)})$ of the $\mathcal T$-saturations of the sets $F_n$. But the
  F{\o}lner \sq s $(F_n)$ and $(F_n^{(\mathcal T)})$ are equivalent
  (see Definition \ref{equi}) and hence they define the same upper
  densities of sets.
For $S\in\mathcal S$ and $g\in S_K$, let
$D_{(S,g)}=gC_S$. Since for any such pair $(S,g)$ we have $Kg\subset
S$ (and hence $Kgc\subset Sc$) and the sets $Sc$, $c\in C_S$, are
tiles (and thus are pairwise disjoint), the sets $Kh$ are pairwise
disjoint when $h=gc$ varies over $D_{(S,g)}$. By syndeticity of the
tiling, each set $C_S$ is syndetic, and so is each of the sets
$D_{(S,g)}$. It follows immediately from finite subadditivity of
$\underline d_{(F_n)}(\cdot)$ and the fact that for any $D\subset G$
and $g\in G$, $\underline d_{(F_n)}(gD)=\underline d_{(F_n)}(D)$, that
syndetic sets have positive lower $(F_n)$-density. In particular,
$\underline d_{(F_n)}(D_{(S,g)})>0$.
Finally, since there are finitely many pairs $(S,g)$, the sets
$D_{(S,g)}$ can be enumerated as $D_1,D_2,\dots,D_r$ ($r\in\na$). By
construction, the family $\{D_i:i=0,1,2,\dots,r\}$, is a partition of
$G$. This ends the proof.
\end{proof}

Recall (see Remark \ref{3.6}) that if the semigroup $G$ is not
commutative then generally speaking, the action $\sigma_g$ need not
preserve $(F_n)$-normality. Nevertheless, by a straightforward
application of shift-invariance of $\lambda$, the following holds.

\begin{cor}\label{totnor}
Let $G$ be an infinitely countable amenable cancellative semigroup. If a F{\o}lner \sq\ $(F_n)$ in $G$ satisfies, for each
$\alpha\in(0,1)$, the summability condition
$\sum_{n\in\na}\alpha^{|F_n|}<\infty$, then $\lambda$-almost every element
$x\in\{0,1\}^G$ has the property that all the images $\sigma_g(x)$
($g\in G$) are $(F_n)$-normal.
\end{cor}

Proposition \ref{ro} allows us to formulate now an (ostensibly
stronger) equivalent version of Theorem \ref{B1} as follows

\begin{thm}\label{co}
Under the assumptions of Theorem \ref{B1}, $\lambda$-almost every
$x\in\{0,1\}^G$ is $(F_n)$-generic for $\lambda$, i.e., the
convergence
\begin{equation*}
\lim_{n\to\infty}\frac1{|F_n|}\sum_{g\in F_n}f(\sigma_gx)=\int f\,d\lambda.
\end{equation*}
holds for any continuous function $f$ on $\{0,1\}^G$.
\end{thm}

\begin{rem}
\begin{enumerate}[(i)]\nopagebreak
  \item The requirement $\sum_{n=1}^\infty\alpha^{|F_n|}<\infty$ for
    each $\alpha\in(0,1)$ is much weaker than \eqref{tempered} (of
    $(F_n)$ being tempered), and is satisfied, for example, by any
    F{\o}lner \sq\ $(F_n)$ such that $|F_n|$ strictly increases as
    $n\to\infty$.
  \item On the other hand, some condition on the growth of $|F_n|$ is
    necessary. For example, for $G=\z$ consider the F{\o}lner
    \sq\ consisting of pairwise disjoint intervals: $n_1$ intervals of
    length $1$ followed by $n_2$ intervals of length $2$, followed by
    $n_3$ intervals of length $3$, etc. If a number $n_k$ is very
    large compared with $k$, then there exists a set $X_k\subset\{0,1\}^\mathbb Z$ with
    $\lambda(X_k)$ close to 1 and such that for every $x\in X_k$ 
    the restriction of $x$ to at least one of the intervals $F_n$ of
    length $k$ will be filled entirely by $0$'s. We can thus arrange
    the \sq\ $n_k$ so that the measures of the complements of the sets $X_k$
    are summable over $k$. Then, by the Borel-Cantelli Lemma, for almost
    every $x$ there will be arbitrarily far F{\o}lner sets filled
    entirely with $0$'s (instead of being filled nearly half-half by
    $0$'s and $1$'s), contradicting $(F_n)$-genericity of $x$ already
    on cylinders of length $1$. In fact, in this example the set of
    $(F_n)$-generic elements has measure zero.
\end{enumerate}
\end{rem}

\begin{rem}\label{popraw}
It is worth mentioning that the method used in the proof of
Theorem~\ref{B1} fails in proving the pointwise ergodic theorem (even
for the Bernoulli measure) for discontinuous $L^\infty$ functions. In
\cite{AJ75}, Akcoglu and del Junco proved that in any ergodic (and
aperiodic) $(\z,+)$-action the pointwise ergodic theorem along the
F{\o}lner sequence of intervals $[n,n+\lfloor\sqrt n\rfloor]$ fails
for the indicator function of some measurable set $A$. Note that, for
any $a\in(0,1)$, the \sq\ $a^{|F_n|}=a^{\lfloor\sqrt n\rfloor}$ is
summable, thus in the case of the uniform Bernoulli measure, according
to Theorem \ref{B1}, the set $A$ cannot be clopen in $\{0,1\}^\z$.
\end{rem}

\subsection{Normal elements form a first category set}\label{tsmall}

In contrast to the measure-theoretic largeness established in
Theorem~\ref{B1}, the following simple proposition de\-monstrates that
the set of $(F_n)$-normal elements in $\{0,1\}^G$ is always
topologically small (i.e., is of first category), without any
assumptions on the F{\o}lner \sq.

\begin{prop} \label{norissm}
Let $G$ be a countably infinite cancellative amenable semigroup and
let $(F_n)$ be a F{\o}lner \sq\ in $G$. Then the set $\mathcal
N((F_n))$ of $(F_n)$-normal elements is of first Baire category in
$\{0,1\}^G$.
\end{prop}

\begin{proof}
For $n\in\na$, a nonempty finite set $K\in G$, $B\in\{0,1\}^K$ and
$\varepsilon\in(0,2^{-|K|})$, let $W(F_n,B,\varepsilon)$ be the union of all
cylinders corresponding to the blocks $C$ over the F{\o}lner set $F_n$,
such that
\[
2^{-|K|}-\varepsilon\le \frac1{|F_n|}|\{g\in F_n: (\forall h\in
K)\ hg\in F_n \text{ and }C(hg)=B(h)\}|\le 2^{-|K|}+\varepsilon.
\]
Since $W(F_n,B,\varepsilon)$ is a finite union of cylinders, it is
clopen. The set $\mathcal N((F_n))$ can be written as
\[
\bigcap_{K\subset G,\, K \text{ nonempty
    finite}}\ \,\bigcap_{B\in\{0,1\}^K}\ \,\bigcap_{0<\varepsilon<2^{-|K|}}\ \,
\bigcup_{n_0\in\na}\ \,\bigcap_{n\ge n_0}\ \,W(F_n,B,\varepsilon).
\]
Note that for each $B,\varepsilon$ and $n_0$ as above, the closed set
$\bigcap_{n\ge n_0}\ \,W(F_n,B,\varepsilon)$ has empty interior,
because the set of the elements of $\{0,1\}^G$ which are constant on
complements of finite sets
is dense in $\{0,1\}^G$. Thus by the Baire theorem, the set
$\bigcup_{n_0\in\na}\ \,\bigcap_{n\ge n_0}\ \,W(F_n,B,\varepsilon)$ is
of first category and contains the set $\mathcal N((F_n))$, which ends
the proof.
\end{proof}

\begin{cor}\label{normnumberscategory}
Let $G$ be either $(\na,+)$ or $(\na,\times)$ and let $(F_n)$ be an
arbitrary F{\o}lner \sq\ in $G$. Then the set of $(F_n)$-normal numbers
in $[0,1]$ (i.e., numbers which have $(F_n)$-normal binary expansions)
is of first category.
\end{cor}

\begin{proof}
For any countably infinite semigroup $G$ in which we have a fixed
\emph{enumeration}, i.e., a bijection between $G$ and $\na$, $n\mapsto
g_n$, the formula
\[
\psi(x) = \sum_{n\in\na}2^{-n}x(g_n), \text{ where} \ \ x=(x(g))_{g\in
  G}\in\{0,1\}^G,
\]
establishes a continuous map from $\{0,1\}^G$ onto $[0,1]$. This map
is injective except on a countable set, on which it is two-to-one.
Note that every continuous map $\phi$ on a compact domain, such that
all but countably many fibers (preimages of points) are singletons and
all other fibers are of first category, preserves the first category.
Indeed, let $A$ be a first category subset of the domain, i.e.,
$A\subset B=\bigcup_n B_n$, where each $B_n$ is compact and has empty
interior. The set $C=\bigcup_n\phi^{-1}(\phi(B_n))$ contains the first
category set $B$ and differs from it by at most a countable union of
fibers (which is of first category), so $C$ is also a first category
set. By continuity, for each $n$ the set $\phi(B_n)$ is compact and,
moreover, it has empty interior (otherwise $\phi^{-1}(\phi(B_n))$
would have nonempty interior, which is impossible since $C$ is of
first category). Thus $\phi(B)=\bigcup_n\phi(B_n)$ is of first
category, and so is its subset $\phi(A)$. We conclude that the binary
expansion map $\psi$ preserves the first category. Now it remains to
apply this fact to $(\na,+)$ or $(\na,\times)$ and invoke Proposition
\ref{norissm}.
\end{proof}

\section{An effectively defined normal set}\label{cha}

Although Lebesgue-almost every number is normal (in the classical
sense) in any base $b\in\na$,
the set of \emph{computable} numbers (i.e., the numbers whose $b$-ary
expansion can be computed with the help of a Turing machine, like, for
example, the Champernowne number) has Lebesgue measure zero. It is so,
because the asymptotic Kolmogorov complexity of such expansions is
zero, while, as shown by A. A. Brudno~\cite{Br82}, a typical expansion
has Kolmogorov complexity $\log b$. Hence computable normal numbers
are highly exceptional among normal numbers.

Let $G$ be a countably infinite amenable group and let $(F_n)$ be an
arbitrary F{\o}lner \sq\ in $G$. In this subsection we describe an
``effective'' construction of an $(F_n)$-normal Cham\-per\-nowne-like
set (viewed, when convenient, as an element of $\{0,1\}^G$). We use
the term \emph{effective} to indicate that our construction is given
by an inductive algorithm which allows to determine, for every $g\in
G$, whether it belongs to the set or not, in finitely many inductive
steps. We cannot claim that our construction gives a computable set,
since we make no assumptions on computability of the group $G$ or the
F{\o}lner \sq\ $(F_n)$.

In the construction of the classical binary Champernowne number three
types of 0-1 words are involved:
\begin{enumerate}
	\item The binary words which are expansions of natural
          numbers. We will call these words ``bricks''. Notice that a
          brick never starts (on the left) with the symbol $0$, and
          there are exactly $2^{k-1}$ bricks of length $k$.
	\item The ``packages''. For each $k$, the $k$th ``package'' is
          the concatenation (in the lexicographical order) of all
          bricks of length $k$. The length of the $k$th package is
          $k2^{k-1}$.
	\item Finally, the ``chains''. For each $k$, the $k$th
          ``chain'' is formed by the packages, from the first to the
          $k$th, concatenated together (by increase of $k$). The $k$th
          chain stretches from the coordinate $1$ to the coordinate
          $\sum_{i=1}^k i2^{i-1}$.
\end{enumerate}
Once the chains are defined, the \sq\ representing the binary
Champernowne number is obtained by taking the coordinatewise limit in
$\{0,1\}^\na$ of the chains (extended to infinite 0-1 words by adding
zeros).

In the construction of the binary Champernowne number described above
one can introduce the following three modifications which do not
destroy the normality:
\begin{enumerate}
	\item one can include as bricks also the words starting with
          the symbol $0$ (the reason why they are not used is purely
          aesthetic) so that there are $2^k$ (rather than $2^{k-1}$)
          bricks of length $k$,
	\item the package of order $k$ may contain every brick of
          order $k$ repeated more than once, as long as the number of
          repetitions is the same (or nearly the same) for every
          brick; then the length of the package of order $k$ is
          $m_kk2^k$ for some \sq\ $m_k$,
	\item in the chain, one may repeat each package of order $k$
          more than once, say $n_k$ times (then the length of the
          $k$th chain equals $\sum_{i=1}^k n_im_ii2^{i-1}$).
\end{enumerate}
While the modifications described above are not necessary in the
construction of the classical Champernowne number, they contain an
idea instrumental for the proof of the following theorem.

\begin{thm}\label{groups}
Let $G$ be a countably infinite amenable group and let $(F_n)$ be an
arbitrary F{\o}lner \sq\ in $G$. Then there exists an effectively
defined $(F_n)$-normal element $x\in\{0,1\}^G$.
\end{thm}

\begin{rem}
The theorem provides $(F_n)$-normal elements even when the
cardinalities $|F_n|$ do not strictly increase, in which case Theorem
\ref{B1} does not necessarily apply.
\end{rem}

\begin{proof}[Proof of Theorem \ref{groups}]
The construction involves a congruent, deterministic, F{\o}lner system
$(\mathcal T_k)_{k\ge 0}$ of proper, syndetic tilings of $G$, starting
with the tiling $\mathcal T_0$ comprised of singletons (see Section
\ref{pre}). We can choose the system $(\mathcal T_k)$ independently of
the F{\o}lner \sq\ $(F_n)$; any such system of tilings will lead to an
$(F_n)$-normal element.

For each $k\ge 1$ and each shape $S$ of $\mathcal T_k$ let
$\B_S=\{0,1\}^S$ be the set of all possible 0-1 blocks over $S$
(clearly, $|\B_S|=2^{|S|}$) and let $\bigcup_{S\in\mathcal S_k}\B_S$
be the set of ``bricks'' of order $k$. Syndeticity of the sets $C_S$
together with the fact that $(\mathcal T_k)$ is a F{\o}lner and
deterministic system of tilings imply that for each $k\ge 1$ there
exists an index $r(k)$ such that the standard tiling of each shape
$S'$ of the tiling $\mathcal T_{r(k)}$, by the tiles of $\mathcal
T_k$, contains, for each shape $S$ of $\mathcal T_k$, at least
$2k2^{|S|}$ tiles of shape $S$. Let $\ell(S',S)\ge 2k2^{|S|}$ denote
the number of tiles of $\mathcal T_k$, having the shape $S$, in the
standard tiling of $S'$. We are now in a position to associate with
each shape $S'$ of $\mathcal T_{r(k)}$ a package of order $k$,
$P(S')\in\{0,1\}^{S'}$. Since for each $S\in\mathcal S_k$ we have
$\ell(S',S)\ge 2k2^{|S|}$, one can divide the collection of all tiles
$T$ of shape $S$, occurring in the standard tiling of $S'$, into
$2^{|S|}$ nonempty and disjoint families $\mathbb T^{(S,S')}_B$
indexed bijectively by the bricks $B\in \B_S$, and having roughly
equal cardinalities. More precisely we can arrange that, for each
$B\in \B_S$, $|\mathbb T^{(S,S')}_B|\in
[\ell(S',S)2^{-|S|}-1,\ell(S',S)2^{-|S|}+1]$ (since $2\le
\frac1k{\ell(S',S)2^{-|S|}}$, the above cardinalities differ by at
most $\frac{100}k$ percent). Then, for each tile $T$ of shape $S$
occurring in the standard tiling of $S'$ we define the restriction of
$P(S')$ to $T$ as the unique brick $B$ such that $T\in\mathbb
T^{(S,S')}_B$. This concludes the definition of the packages $P(S')$
of order $k\ge 1$. For completeness, we let $r(0)=0$ and define the
package of order $0$ as the single symbol $0$. This is consistent with
the previous conventions: the package of order $0$ has a shape
corresponding to the tiling $\mathcal T_{r(0)}=\mathcal T_0$. Since
$\mathcal T_0$ has only one shape (the singleton), the $0$th package
is a block over a singleton (i.e., a single symbol).

At this point we need to introduce some additional terminology. For a
nonempty finite set $K\subset G$ and $\varepsilon>0$, a block $C\in
\{0,1\}^F$ over another finite set $F\subset G$ is
$(K,\varepsilon)$-normal if for every block $B\in\{0,1\}^K$ one has
\[
2^{-|K|}-\varepsilon\le \frac1{|F|}|\{g\in F: (\forall h\in K)\ hg\in
F\text{ and }C(hg)=B(h)\}|\le 2^{-|K|}+\varepsilon.
\]
Summing over all blocks $B\in\{0,1\}^K$ one obtains that in order for
$C$ to be $(K,\varepsilon)$-normal, $F$ must be
$(K,2|K|\varepsilon)$-invariant.

The following fact is now easily verified:
\begin{enumerate}
	\item For any nonempty finite set $K\subset G$ and any
          $\varepsilon>0$, if $k$ is sufficiently large then every
          package of order $k$ is $(K,\varepsilon)$-normal. So is
          every concatenation of such (shifted) packages.
\end{enumerate}

In order to define the $(F_n)$-normal element $x\in\{0,1\}^G$ we first
create a (not proper) \emph{mixed tiling} $\Theta$, i.e., a partition
of $G$ into tiles belonging to different tilings from the
sub\sq\ $(\mathcal T_{r(k)})_{k\ge 0}$ (this will be possible due to
the fact that we are working with a congruent system of tilings). Then
we will define $x$ as follows: $x$ restricted to a tile $T$ of
$\Theta$ equals the (appropriately shifted) package associated to the
shape of $T$ (if $T$ belongs to the tiling $\mathcal T_{r(k)}$ then
the order of the package is $k$). In this manner $x$ becomes an
infinite concatenation of packages of various orders. We remark that
working with a mixed tiling is equivalent to working with chains: one
can define the $k$th chain as the part of $x$ covered by the tiles of
$\Theta$ belonging to the tilings $\mathcal T_{r(0)},\mathcal
T_{r(1)},\dots,\mathcal T_{r(k)}$. Conversely, whenever a Champernowne
set is defined via the concept of chains, as a concatenation of
packages of different orders, then the tiles of $\Theta$ are simply
the domains of these packages.

It remains to describe how we define the mixed tiling $\Theta$. The
procedure will depend on the a priori given F{\o}lner \sq\ $(F_n)$
(which so far was not involved in the construction).

For each $k\ge 1$ let $n_k$ be such that the F{\o}lner sets $F_n$ with
$n>n_k$ are $(\mathbf S_k,\frac1k)$-invariant, where $\mathbf
S_k=\bigcup_{S'\in\mathcal T_{r(k)}} S'{S'}^{-1}$ (then $F_n$ is also
$(\mathbf S_i,\frac1k)$-invariant for all $i\le k$). We begin by
defining $\Theta$ on the $\mathcal T_{r(1)}$-saturation (denoted by
$\mathbf F_1$) of the union $F_1\cup F_2\cdots\cup F_{n_1}$ simply as
$\mathcal T_0$. Notice that $\Theta$ remains undefined on the
complement of $\mathbf F_1$ which is a union of complete tiles of
$\mathcal T_{r(1)}$. Inductively, let $k\ge 2$ and suppose that, after
step $k\!-\!1$, $\Theta$ remains undefined on a union of complete
tiles of $\mathcal T_{r(k-1)}$. In the $k$th step we define $\Theta$
as $\mathcal T_{r(k-1)}$ on the yet untiled part of the $\mathcal
T_{r(k)}$-saturation $\mathbf F_k$ of the union $F_1\cup F_2\cdots\cup
F_{n_k}$. Note that $\Theta$ remains undefined on a union of complete
tiles of $\mathcal T_{r(k)}$. Continuing in this way we will define
the mixed tiling $\Theta$ on a set containing the union of all
F{\o}lner sets $F_n$. If any part of the group remains untiled, we define
$\Theta$ on that part as $\mathcal T_0$. This concludes the
construction of the mixed tiling $\Theta$.

Observe that the mixed tiling $\Theta$ has the following properties:
\begin{enumerate}
	\item[(2)] Each $F_n$ is covered only by tiles of those
          shapes $S'$ for which $F_n$ is
          $(S'{S'}^{-1},\frac1k)$-invariant (where $k$ is the largest
          index such that $n>n_k$). This implies that $F_n$ differs
          from its $\Theta$-saturation by at most $\frac1k|F_n|$
          elements.
	\item[(3)] For each $k\ge 1$, $\Theta$ uses only finitely many
          tiles belonging to $\mathcal T_{r(k)}$.
\end{enumerate}

As we have already explained earlier, $\Theta$ determines some
$x\in\{0,1\}^G$. It remains to verify the $(F_n)$-normality of $x$.
Let $K\subset G$ be a nonempty finite set and let us fix some
$\varepsilon>0$. It is enough to show that, for $n$ sufficiently
large, $x|_{F_n}$ is $(K,3\varepsilon)$-normal. Pick
$k\ge\frac1\varepsilon$ so large that all packages of orders larger
than or equal to $k$ are $(K,\varepsilon)$-normal (see (1) above).
Choose $n\ge n_k$. In order to determine the parameter $\delta$ for
which $x|_{F_n}$ is $(K,\delta)$-normal we first replace $F_n$ by its
$\Theta$-saturation. By (2), this affects the estimation of $\delta$
by at most $\varepsilon$. Next, we remove from this saturation all
tiles of orders smaller than $k$ (there are finitely many such tiles).
If $n$ is large enough, this last step also affects the estimation of
$\delta$ by at most $\varepsilon$. Now it remains to examine the
restriction of $x$ to a set on which it is a concatenation of packages
of orders at least $k$. By the choice of $k$, this restriction is
$(K,\varepsilon)$-normal. It follows that $x$ restricted to $F_n$ is
$(K,3\varepsilon)$-normal, as required. This concludes the proof.
\end{proof}

Via Theorem \ref{sgrvsgr} and Lemma \ref{ll}, the above construction
applies also to cancellative semigroups.

\begin{cor}
Let $G$ be a countably infinite amenable cancellative semigroup and let $(F_n)$ be a F{\o}lner \sq\ in $G$. Then there exists an
effectively defined $(F_n)$-normal subset of $G$.
\end{cor}

\section{Multiplicative normality}\label{mnor}

In this section we will focus on the action of $(\na,\times)$ on the
symbolic space $\{0,1\}^\na$. In this case the shift action,
henceforth called the \emph{multiplicative shift} and denoted by
$(\rho_n)_{n\in\na}$ is defined on $\{0,1\}^\na$ as follows:
\[
\text{if \ }x=(x_j)_{j\in\na}\text{ \ then \ }\rho_n(x) =
(x_{jn})_{j\in\na}.
\]
In other words, $\rho_n$ maps each binary \sq\ to its sub\sq\ obtained
by reading its every $n$th term. Clearly, the classical
$(\frac12,\frac12)$-Bernoulli measure on the symbolic space
$\{0,1\}^\na$ is invariant under both the additive and multiplicative
shift actions, and in both cases it is the unique measure of maximal
entropy. In fact we are dealing here with the case of a \sq\ of
independent identically distributed random variables, which
corresponds to the Bernoulli process regardless of the applied action,
as long as the action ``permutes'' the indices (we use quotation
marks, because our ``permutations'' are not surjective). Note that
both shift actions are ergodic (in fact mixing) and their
Kolmogorov-Sinai entropies equal the entropy of the generating
partition $\{[0],[1]\}$, i.e., to $\log2$, and so are the
\tl\ entropies of both shift actions. For a treatment of entropy 
for actions of amenable groups see for example \cite{Ol85}.

\subsection{F{\o}lner \sq s in $(\na,\times)$}\label{Folners}

The semigroup $(\na,\times)$ is a free Abelian semigroup generated by
the set of primes. We will denote the set of primes by $\PP$ and view
$(\na,\times)$ as the direct sum\footnote{Recall that a direct sum is
  the subset of the Cartesian product (with addition acting
  coordinatewise) consisting of points with at most finitely many
  nonzero coordinates.}
\[
\mathbb G=\bigoplus_{p\in\PP}\na_p,
\]
where, for each $p\in\PP$, $\na_p$ is the same additive semigroup
$(\na\cup\{0\},+)$. The isomorphism is given by
\[
(k_1,k_2,\dots,k_r)\mapsto p_1^{k_1}p_2^{k_2}\cdots p_r^{k_r},
\]
where $p_1=2,\ p_2=3\ ,p_3=5,$ etc. are consecutive prime numbers.
Notice that $\mathbb G$ is \emph{additive}, i.e., in
this representation multiplication of natural numbers is interpreted
as addition of vectors.

In order to deal with the actions of $(\na,\times)$ it is crucial to
identify convenient choices of F{\o}lner \sq s in this semigroup. A
natural choice of a F{\o}lner \sq\ in $\mathbb G$ is given by
\emph{anchored} (i.e., containing the origin) rectangular boxes of the form
\begin{equation}\label{box}
F=\{0,1,\dots,k_1\}\times\{0,1,\dots,k_2\}\times\cdots\times
\{0,1,\dots,k_d\}\times\{0\}\times\{0\}\times\cdots.
\end{equation}
The parameter $d$ (i.e., the largest index $i$ such that $k_i>0$) will
be referred to as the \emph{dimension} of $F$. The number $k_i$ will
be called the \emph{size of $F$ in the $i$th direction}. Let
now \begin{equation}\label{boxn}
  F_n=\{0,1,\dots,k^{(n)}_1\}\times\{0,1,\dots,k^{(n)}_2\}\times
\cdots\times\{0,1,\dots,k^{(n)}_{d^{(n)}}\}\times\{0\}\times\{0\}\times
\cdots
\end{equation}
be a \sq\ of anchored rectangular boxes. With this notation, $(F_n)$
is a F{\o}lner \sq\ in $\mathbb G$ if and only if $\lim_n d_n=\infty$
and $\lim_n k^{(n)}_i=\infty$ for each $i\in\na$). Any such F{\o}lner
\sq\ will be called \emph{anchored rectangular}. The verification of
the F{\o}lner property is straightforward. Preferably, the F{\o}lner
sets should increase with respect to inclusion, which means that the
\sq s $(d_n)$ and $(k^{(n)}_i)$ for each $i$ should be nondecreasing
and the sum $k_1^{(n)}+k_2^{(n)}+\cdots+k_{d_n}^{(n)}$ should be
strictly increasing. Such increasing F{\o}lner \sq s will be called
\emph{nice}. Not every nice F{\o}lner \sq\ $(F_n)$ is tempered.
However, since the cardinalities $|F_n|$ strictly increase, Theorem
\ref{B1} applies.

Every nice F{\o}lner \sq\ occurs as a sub\sq\ of a specific \emph{nice and
slow} F{\o}lner \sq, such that at each step the sum
$k_1^{(n)}+k_2^{(n)}+\cdots+k_{d_n}^{(n)}$ increases by~$1$. The
choice of a nice and slow F{\o}lner \sq\ is equivalent to fixing a
``\sq\ of directions'' $(i_n)$, in which every natural number appears
infinitely many times, and letting
$k^{(n)}_i=|\{j\in\{1,\dots,n\}:i_j=i\}|$. There are several fairly
natural options for choosing the \sq\ $(i_n)$, for example:
\begin{align}
&1;\,1,2;\,1,2,3;\,1,2,3,4;\,1,2,3,4,5;\,\dots\ \ \text{the
    \emph{staircase type}},\label{tt}\\
&1,2,1,3,1,2,1,4,1,2,1,3,1,2,1,5,\dots\ \ \text{the \emph{Toeplitz
      type}}.
\end{align}

Now we can translate all this to the multiplicative representation
$(\na,\times)$. In $(\na,\times)$ we have the natural partial order
given by $m\preccurlyeq M\ \iff\ m|M$. The set $\na$ equipped with
this order is a \emph{directed set} (i.e., every two elements have a
common upper bound; in this case a common multiple). A \sq\ of natural
numbers $(L_n)$ \emph{multiplicatively tends to infinity}, if for any
$m\in\na$ there exists $n_0$ such that $m\preccurlyeq L_n$ for all
$n\ge n_0$. If, in addition, the sequence $L_n$ strictly increases
with respect to the multiplicative order, we will say that $(L_n)$
\emph{multiplicatively increases to infinity}. For an anchored
rectangular box $F\subset\mathbb G$ (see \eqref{box}) the number $L =
p_1^{k_1}p_2^{k_2}\cdots p_d^{k^d}$ will be called the \emph{leading
  parameter} of $F$. Interpreting $F$ as a subset of $(\na,\times)$,
notice that $F=\{m:m\preccurlyeq L\}$ (i.e., $F$ is the set of all
divisors of $L$). With this terminology, a \sq\ of anchored
rectangular boxes in $(\na,\times)$ is a F{\o}lner \sq\ (resp. nice
F{\o}lner \sq) if and only if the \sq\ $(L_n)$ of their leading parameters
multiplicatively tends (resp. multiplicatively increases) to infinity.
A nice F{\o}lner \sq\ $(F_n)$ in $(\na,\times)$ is nice and slow if and
only if $\frac{L_{n+1}}{L_n}$ is a prime for every $n$. Notice that
even if $(F_n)$ is nice and slow, the cardinalities $|F_n|$ grow
relatively fast. Indeed, from time to time the dimension $d_{n+1}$ of
$F_{n+1}$ has to increase, i.e., a new direction has to be included,
and then the cardinality doubles: $|F_{n+1}|=2|F_n|$. Otherwise the
cardinality is multiplied by a factor smaller than $2$, but in any
case a rectangular box of dimension $d_n-1$ is added. In particular,
$|F_{n+1}|-|F_n|>1$ for $n>1$.

Obviously, there are many other F{\o}lner \sq s in $\mathbb G$. The
rectangular boxes need not be anchored at zero, and moreover, they can
be replaced by other shapes. For instance, it is possible to create a
F{\o}lner \sq\ with $|F_n|=n$, but it is not going to be rectangular
(however, it may have a nice and slow F{\o}lner subsequence). We skip
further details. While there is no preferred ``canonical'' choice for
a F{\o}lner \sq\ in $\mathbb G$, it will be convenient for our purposes
to focus on anchored rectangular, and, in particular, on nice F{\o}lner
\sq s (mainly due to advantageous arithmetic properties of their
multiplicative interpretation).

\subsection{Multiplicative Champernowne set}\label{champer}

The construction of a Champernowne set in $(\na,\times)$ can be made
significantly more transparent than in the general case discussed in
Section \ref{cha}. This is due to the fact that the semigroup
$(\na,\times)$ admits a system of (congruent, deterministic, F{\o}lner,
syndetic) \emph{monotilings}, i.e., tilings with only one shape. In
fact, any rectangular box tiles the semigroup, while a congruent
system of tilings is obtained from a specific F{\o}lner \sq, which we
will call \emph{doubling}. This will enable us to create ``condensed''
packages which contain every brick exactly once (like in the classical
Champernowne construction). For every $k$, the $k$th chain still has
to contain more than one repetition of every package of order $k$
(this we would have to do even in the two-dimensional semigroup
$(\na^2,+)$), but we will use the least possible number of repetitions
to fill a rectangular box the size of the next order package. In this
manner we will obtain a ``compendious'' Champernowne set, which will
turn out to be normal at least with respect to the same doubling
F{\o}lner \sq\ which is used in its construction. Later we will present a
slight modification of the same construction, which produces a
``net-normal'' set, i.e., normal with respect to any nice F{\o}lner
\sq, at the cost of repeating each package of order $k$ an infinite
number of times.

We begin by formally introducing the notion of a doubling F{\o}lner
\sq. Again, we will interpret $(\na,\times)$ as the additive
semigroup~$\mathbb G$.

\begin{defn}\label{doubling}
A nice F{\o}lner \sq\ $(F_n)$ is called \emph{doubling} if $F_{n+1}$ is
a disjoint union $F_n\cup (v_n+F_n)$ for some $v_n\in \mathbb G$.
\end{defn}

Note that since $F_{n+1}$ is a anchored rectangular box, $v_n$ must be
equal to one of the vectors spanning $F_n$, i.e., if
\[
F_n=\{0,1,\dots,k^{(n)}_1\}\times\{0,1,\dots,k^{(n)}_2\}\times\cdots
\times\{0,1,\dots,k^{(n)}_{d_n}\}\times\{0\}\times\{0\}\times\cdots,
\]
then $v_n$ is of the form $(0,0,\dots,0,k^{(n)}_i+1,0,0,\dots)$, where
$k^{(n)}_i+1$ occurs as the $i$th term, $i=1,2,\dots,d_n$, or
$v_n=(0,0,\dots,0,1,0,0,\dots)$, where $1$ occurs at a position larger
than $d_n$.

Any doubling F{\o}lner \sq\ can be obtained by the following procedure.
As before, in the construction of a nice and slow F{\o}lner \sq, we fix
a \sq\ of directions $(i_n)_{n\ge 1}$ in which each natural $i$ is
repeated infinitely many times. We begin with the ``zero F{\o}lner
set'' $F_0=\{0\}$. Once the F{\o}lner set $F_n$ is determined, the next
one, $F_{n+1}$, instead of growing in \emph{by a unit} in the
direction $i_{n+1}$, is \emph{doubled} in that direction. The
cardinality of $F_n$ will hence be equal to $2^n$. For example, if
$(i_n)$ is the staircase \sq\ $1;\,1,2;\,1,2,3;\dots$, the first six
F{\o}lner sets are
\begin{align*}
F_0&=\{0\}\\
F_1&=\{0,1\}\\
F_2&=\{0,1,2,3\}\\
F_3&=\{0,1,2,3\}\times\{0,1\}\\
F_4&=\{0,1,2,3,4,5,6,7\}\times\{0,1\}\\
F_5&=\{0,1,2,3,4,5,6,7\}\times\{0,1,2,3\}\\
F_6&=\{0,1,2,3,4,5,6,7\}\times\{0,1,2,3\}\times\{0,1\}.
\end{align*}
(for convenience we skip the infinite product of singletons $\{0\}$
that should follow to the right in the formula for each of the above
sets).

\subsubsection{The construction}\label{constru}

We shall now construct an $(F_n)$-normal element $x\in\{0,1\}^\na$ for
a doubling F{\o}lner sequence. We describe how we build the bricks and
packages, and how we place them in $x$ (we will use the language of
chains rather than that of mixed tilings). The bricks of order $k$
will simply be blocks over the F{\o}lner set $F_k$ (of cardinality
$2^k$), i.e., the bricks will belong to $\{0,1\}^{F_k}$. We accept as
bricks of order $k$ all blocks over $F_k$. Thus there will be
$2^{2^k}$ different bricks, which, when concatenated together (each
used exactly once), produce a package of cardinality $2^{2^k+k}$.
Since the sizes in all directions of all our objects (F{\o}lner sets,
bricks, packages, etc.) are powers of $2$, we can arrange the package
so that it is a block over $F_{2^k+k}$ (the index $2^k+k$ plays the
role of $r(k)$ from the general construction). We define the $0$th
chain as the concatenation of $4$ copies of the package of order zero
arranged to fill a block over $F_3$ (see Figure \ref{fig1} below). For
$k\ge 1$ we assume inductively that the $(k\!-\!1)$st chain is a block
over the same set as the package of order $k$ (for $k=1$ this holds).
This assumption guarantees that the $(k\!-\!1)$st chain is saturated
with respect to the tiling number $r_k$, so it can be concatenated
together with packages of order $k$ without gaps or overlaps. Now we
can build the $k$th chain. It consists of
\begin{enumerate}
	\item[(i)] the $(k\!-\!1)$st chain occupying the ``lower
          left'' corner, i.e., containing the origin, and
	\item[(ii)] $2^{2^k+1}-1$ shifted copies of the $k$th order
          package,
\end{enumerate}
so that the chain has cardinality $2^{2^{k+1}+k+1}$. The chain can be
arranged to be a block over $F_{2^{k+1}+k+1}$, i.e., over the same set
as the package of order $(k+1)$, as required in the induction.

Figure \ref{fig1} corresponds to the (mentioned above) ``staircase
type'' doubling F{\o}lner \sq. It shows the package of order $0$ with
the initial ``zero'' brick of order $0$ shaded, and next to it the
$0$th chain, which is a concatenation of four such packages. In the
next line we show the package of order $1$ with the initial ``zero''
brick of order $1$ shaded, and next to it the $1$st chain which is a
concatenation of the preceding chain (shaded) and seven identical
packages of order $1$. The last picture shows the package of order $2$
with the initial ``zero'' brick of order $2$ shaded. The $2$nd chain
is too large to be shown. It is a concatenation of the $1$st chain and
$31$ copies of the package of order $2$, and it is a block over
$F_{11}$ (which is four-dimensional).

\begin{figure}
\begin{flushleft}
\begin{tikzpicture}
\matrix (mA) [matrix of nodes,inner sep=0,nodes={draw,inner sep=1mm}]
{
|[fill=black!15]|0&1&\\
};
\end{tikzpicture} \ $0$th package \ \ \ \
\begin{tikzpicture}
\matrix (mA) [matrix of nodes,inner sep=0,nodes={draw,inner sep=1mm}]
{
0&1&0&1\\
0&1&0&1\\
};
\end{tikzpicture} \ $0$th chain

\bigskip
\begin{tikzpicture}
\matrix (mA) [matrix of nodes,inner sep=0,nodes={draw,inner sep=1mm}]
{
1&0&1&1\\
|[fill=black!15]|0&|[fill=black!15]|0&0&1\\
};
\end{tikzpicture}\ $1$st package \ \ \ \
\begin{tikzpicture}
\matrix (mA) [matrix of nodes,inner sep=0,nodes={draw,inner sep=1mm}]
{
1&0&1&1&1&0&1&1\\
0&0&0&1&0&0&0&1\\
1&0&1&1&1&0&1&1\\
0&0&0&1&0&0&0&1\\
};

\matrix (mB) [matrix of nodes,fill=white,inner sep=0,nodes={draw,inner
    sep=1mm}] at ($(mA.south west)+(1,.3)$)
{
1&0&1&1&1&0&1&1\\
0&0&0&1&0&0&0&1\\
|[fill=black!15]|0&|[fill=black!15]|1&|[fill=black!15]|0&|[fill=black!15]|1&1&0&1&1\\
|[fill=black!15]|0&|[fill=black!15]|1&|[fill=black!15]|0&|[fill=black!15]|1&0&0&0&1\\};

\draw[dashed](mA.north east)--(mB.north east);
\draw[dashed](mA.north west)--(mB.north west);
\draw[dashed](mA.south east)--(mB.south east);
\end{tikzpicture}\ $1$st chain
\bigskip

\begin{tikzpicture}
\matrix (mA) [matrix of nodes,inner sep=0,nodes={draw,inner sep=1mm}]
{
1&1&1&0&1&1&1&1\\
1&1&0&1&1&1&1&0\\
1&0&1&0&1&0&1&1\\
1&0&0&0&1&0&0&1\\};

\matrix (mB) [matrix of nodes,fill=white,inner sep=0,nodes={draw,inner
    sep=1mm}] at ($(mA.south west)+(1,.3)$)
{
0&1&1&0&0&1&1&1\\
0&1&0&1&0&1&1&0\\
0&0&1&0&0&0&1&1\\
|[fill=black!15]|0&|[fill=black!15]|0&|[fill=black!15]|0&|[fill=black!15]|0&0&0&0&1\\};

\draw[dashed](mA.north east)--(mB.north east);
\draw[dashed](mA.north west)--(mB.north west);
\draw[dashed](mA.south east)--(mB.south east);
\end{tikzpicture}\ $2$nd package
\end{flushleft}
\caption{}\label{fig1}
\end{figure}

The chains converge to an element $x\in\{0,1\}^{\mathbb G}$. The set
$\{g\in \mathbb G:c(g)=1\}$ and the real number with binary expansion
$x$ will be called the \emph{multiplicative Champernowne set} and
\emph{multiplicative Champernowne number}, respectively.

\subsubsection{$(F_n)$-normality of $x$}\label{normc}

Recall that given a nonempty finite set $K$ and $\varepsilon>0$, for
some large $k_0$, packages of orders $k\ge k_0$ are multiplicatively
$(K,\varepsilon)$-normal. So, to prove $(F_n)$-normality of $x$ we
need to check that, as $n$ increases, we have $\frac{a_n}{|F_n|}\to
1$, where $a_n$ is the cardinality of the portion $F_n'$ of $F_n$ such
that $x|_{F_n'}$ is a concatenation of packages of orders $k\ge k_0$.
First, we will check this for indices $n$ of the special form
$n=r(k)=2^k+k$. For such an $n$, $F_n$ is filled with the $k$th chain,
consisting of the $(k\!-\!1)$st chain and many (precisely,
$2^{2^k+1}-1$) packages of order $k$ (having the same size as the
$(k\!-\!1)$st chain), so these packages ``dominate'' in $x|_{F_n}$
(precisely, $\frac{a_n}{|F_n|}\ge 1-\frac1{2^{2^k+1}}$). If a large
$n$ is not of this form, then, for some $k\ge k_0$, we have $2^k+k< n
< 2^{k+1}+k+1$, and the block $x|_{F_n}$ is a concatenation of the
$k$th chain and some number of packages of order $k+1$, so
$\frac{a_n}{|F_n|}>\frac{a_{2^k+k}}{|F_{2^k+k}|}$. This completes the proof
of $(F_n)$-normality of $x$.

\begin{rem}
We can also deduce multiplicative normality of $x$ with respect to the
nice and slow F{\o}lner \sq\ of which $(F_k)$ is a subsequence (to
obtain such a nice and slow F{\o}lner \sq, instead of doubling a
direction we increase it by $1$ several times). We omit the details.
On the other hand, $x$ is definitely not normal for some other nice
F{\o}lner \sq s. For instance, if the F{\o}lner sets increase in the
first direction much faster than in other directions (elongated
shapes) then the symbol $0$ will prevail. We skip the details again.
\end{rem}

\subsection{Net-normal sets}\label{univers}

The notion of an anchored rectangular box or F{\o}lner \sq\ is meaningful
not only in $\mathbb G$, but also in $\na^d$ ($d\in\na$) with addition. The elements
of such a \sq\ are $d$-dimensional anchored rectangular boxes given by
\begin{equation}\label{box2}
F=\{0,1,\dots,k_1\}\times\{0,1,\dots,k_2\}\times\cdots\times\{0,1,\dots,k_d\}
\end{equation}
(cf. \eqref{box}). Denoting by $G$ either $\mathbb G$ or $\na^d$ for
some $d\in\na$, let $\mathcal F_G$ stand for the family of all
anchored rectangular boxes in $G$. In either case, this family,
ordered by inclusion\footnote{When working with $(\na,\times)$ rather
  than with $\mathbb G$, the above order on $\mathcal F_G$ coincides
  with the multiplicative order $\preccurlyeq$ (see Section
  \ref{Folners}) applied to the respective leading parameters.}, is a
directed set: any two such boxes are contained in a third one. So any
function with domain $\mathcal F_G$ is a \emph{net}. With slight abuse
of terminology, the directed set $\mathcal F_G$ will be called the
\emph{F\o lner net} (formally, this term should refer to the identity
function on $\mathcal F_G$).

Now we introduce the notion of net-normality.

\begin{defn}
Let $G$ be either $\mathbb G$ or $\na^d$ for some $d\in\na$. A set
$A\subset G$ (as well as its indicator function $\mathbbm
1_A\in\{0,1\}^G$) is \emph{net-normal} if for any finite set $K\subset
G$ and every block $B\in\{0,1\}^K$, the net of averages (indexed by
$F\in\mathcal F_G$)
\[
\frac1{|F|}\mathsf N(B,x,F)
\]
converges\footnote{A net $\iota\mapsto a_\iota$ of real numbers,
  indexed by a directed set $(I,\ge)$, \emph{converges} to a limit $a$
  if for every $\varepsilon>0$ there exists $\iota_0$ such that
  $|a_\iota-a|<\varepsilon$ for every $\iota\ge\iota_0$ in $I$.} to
$2^{-|K|}$ (comp. with \eqref{fnbnorm}). If $G=\mathbb G$ is
interpreted as the multiplicative semigroup $(\na,\times)$, a
net-normal set $A\subset\na$ (and its indicator function $\mathbbm
1_A\in\{0,1\}^\na$) will be called \emph{multiplicatively net-normal}.
\end{defn}

\begin{prop}\label{34}
A set $A\subset\na$ is multiplicatively net-normal if and only if it is 
$(F_n)$-normal with respect to every
anchored rectangular F{\o}lner \sq\ $(F_n)$ in $(\na,\times)$. Also,
$A$ is multiplicatively net-normal if and only if it is $(F_n)$-normal
with respect to every nice (i.e., anchored rectangular and increasing
by inclusion) F{\o}lner \sq\ $(F_n)$ in $(\na,\times)$. 
\end{prop}
\begin{proof} Every anchored rectangular F{\o}lner \sq\ is a \emph{subnet} of the F{\o}lner net, hence net-normality implies normality with respect to any anchored rectangular 
F\o lner \sq. The fact that normality with respect to every nice F\o lner \sq\ implies net-normality follows from the trivial observation that the
failure of convergence of any countable net can be detected along some
increasing sub\sq\ of that net (in our case, an increasing sub\sq\ of the F\o lner 
net is a nice F\o lner \sq).
\end{proof}

It is natural to inquire about the existence of net-normal sets and
their typicality, in $\na^d$ and in $\mathbb G$. Curiously enough, it
turns out that the answers are different for $\na^d$ and $\mathbb G$.
This difference is captured by the following two theorems.

\begin{thm}
For any $d\ge 1$, almost every (with respect to the Bernoulli measure
$\lambda$) element of $\{0,1\}^{\na^d}$ is net-normal.
\end{thm}
\begin{proof}
It is not hard to check that the proof of Theorem \ref{B1} works also
for countable F{\o}lner nets. It now suffices to notice that in $\na^d$
the sum $\sum_{F\in\mathcal F_{\na^d}}e^{-|F|}$ is finite.
\end{proof}

The above argument fails for $\mathbb G$, because the sum
$\sum_{F\in\mathcal F_{\mathbb G}}e^{-|F|}$ diverges (for example,
there are infinitely many anchored rectangular boxes of cardinality
$2$). In fact, we have the following theorem.

\begin{thm}\label{unizero}
The collection of all net-normal elements in $\{0,1\}^{\mathbb G}$ has
measure zero for the Bernoulli measure $\lambda$.
\end{thm}

\begin{proof}
For $\lambda$-almost every $x\in\{0,1\}^{\mathbb G}$ we will construct
a nice (in fact doubling) F{\o}lner \sq\ $(F_n(x))_{n\in\na}$ for which
$x$ is not $(F_n(x))$-normal. By Proposition \ref{34} this will imply that
any such $x$ is not net-normal. For even $n$ the definition of
$F_n(x)$ will depend on $x$ and will apply to a subset of full measure
of the set of points $x$ for which $F_{n-1}(x)$ was defined. For odd
$n$, $F_n(x)$ will be defined for all points $x$ for which
$F_{n-1}(x)$ is defined (at step 1 this will be the whole space
$\{0,1\}^\mathbb G$). For even $n$ the rectangular box $F_n(x)$ will
grow (relatively to $F_{n-1}(x)$) in a ``random'' (i.e., depending on
$x$) direction. However, for $F_n(x)$ to be a F{\o}lner \sq, the
rectangles must grow in every direction infinitely many times. This
property will be guaranteed by judicial (deterministic) choice of the
directions at odd steps of the construction.

We start by defining $F_1(x)$ (for every $x\in X_1=\{0,1\}^\mathbb G$)
as the ``zero rectangle'':
\[
F_1(x)=\{0\}\times\{0\}\times\{0\}\times\cdots.
\]
Next, for each $x$ and every $k\ge 1$ we consider the rectangle which
is ``doubled'' in the $k$th direction:
\[
F_1(x,k) = \{0\}\times\{0\}\times\{0\}\times\cdots\times\{0\}
\times\{0,1\}\times\{0\}\times\cdots,
\]
where $\{0,1\}$ appears at the $k$th position in the product. Note
that the sets $F_1(x,k)\setminus F_1(x)$ (which at this step of
construction are singletons) are disjoint for different $k$'s, and
hence the functions $x\mapsto x_{g_k}$ where $g_k\in F_1(x,k)\setminus
F_1(x)$ form an i.i.d. \sq\ of random variables. Thus, there exists a
full measure set $X_2\subset \{0,1\}^\mathbb G$, such that for every
$x\in X_2$ there exists $k$ such that $x_{g_k}=0$. We let $k_1(x)$ be
the smallest such $k$ and we define $F_2(x)$ as $F_1(x,k_1(x))$. In
this manner, for almost every $x$, we have guaranteed at least half of
the symbols $x_g$, $g\in F_2(x)$, to be zeros. From now on we consider
only the points $x\in X_2$.

Next, we produce $F_3(x)$ by doubling $F_2(x)$ in the direction
provided (for example) by the staircase \sq\ \eqref{tt}. Since the
first term of the staircase \sq\ is~$1$, we simply double the first
coordinate:
\[
F_3(x) =\begin{cases}
\{0,1\}\times\{0\}\times\{0\}\times\cdots\times\{0\}\times\{0,1\}
\times\{0\}\times\cdots& \text{ if }k_1(x)>1\\
\{0,1,2,3\}\times\{0\}\times\{0\}\times\cdots&\text{ if }k_1(x)=1.
\end{cases}
\]
This time for any $x$ some two ``random'' symbols $x_g$ with $g\in
F_3(x)\setminus F_2(x)$ appear in the block $x|_{F_3(x)}$. We let
$X_3=X_2$. At the fourth step, for each $x$ and every
$k>k_1(x)$,\footnote{The requirement $k>k_1(x)$ is inessential. We put
  it only to reduce the variety of possible formulas for $F_3(x,k)$.}
we consider the rectangle which is ``doubled'' in the $k$th direction:
\[
F_3(x,k)=\begin{cases}\{0,1\}\times\{0\}\times\cdots\times\{0\}\times
\{0,1\}\times\{0\}\times\cdots\times\{0\}\times\{0,1\}\times\{0\}\times
\cdots\\
\ \ \ \ \ \ \ \ \ \ \ \ \ \ \ \ \ \ \ \ \ \ \ \ \ \ \ \ \ \ \ \ \ \ \ \
\ \ \ \ \ \ \ \ \ \ \ \ \ \ \ \ \ \ \ \ \ \ \ \ \ \ \ \ \ \ \ \ \ \ \
\ \,\text{ if }k_1(x)>1\\
\{0,1,2,3\}\times\{0\}\times\{0\}\times\cdots\times\{0\}\times\{0,1\}
\times\{0\}\times\cdots\text{ if }k_1(x)=1,
\end{cases}
\]
where the last appearance of $\{0,1\}$ takes place at the position $k$
in the product. As before, there exists a set of full measure,
$X_4\subset X_3$, such that for every $x\in X_4$ there exists $k$ for
which all symbols $x_g$ with $g\in F_3(x,k)\setminus F_3(x)$ are zeros
(again, it is essential that the sets $F_3(x,k)\setminus F_3(x)$ are
disjoint for different $k$'s). We let $k_2(x)>k_1(x)$ be the smallest
such $k$ and define $F_4(x)=F_3(x,k_2(x))$. In this way, we have
guaranteed at least the fraction $\frac12+\frac18$ of zeros in the
block $x|_{F_4(x)}$.

Continuing in this way, at the odd steps we will double the rectangles
in the directions provided by the staircase \sq, and at the even steps
(restricting to a full measure set) we will double the rectangles so
that all the symbols $x_g$ with $g\in F_n(x)\setminus F_{n-1}(x)$
(which constitutes half of $F_n(x)$) will be zeros.

It is clear that eventually, for $\lambda$-almost every $x$ (more
precisely for $x\in\bigcap_n X_n$), we will obtain a doubling F{\o}lner
sequence $(F_n(x))$, such that the lower $(F_n(x))$-density of zeros
in $x$ is at least $\frac12+\frac18+\frac1{32}+\cdots=\frac23$. Thus
$x$ is not $(F_n(x))$-normal.
\end{proof}

We remark that the set of net-normal elements of $\{0,1\}^G$ is an
intersection of the sets of $(F_n)$-normal elements for a family of
F{\o}lner \sq s $(F_n)$, hence, by Proposition \ref{norissm}, it is of
the first category. Further, in view of the Theorem \ref{unizero}, in
the case $G=\mathbb G$, it is not only \tl ly, but also
measure-theoretically small. Nevertheless, we will prove in Theorem \ref{uni}
that this set is nonempty, and moreover, as follows from Remark \ref{mage}
below, it is even uncountable.

We begin with a preparatory lemma.

\begin{lem}\label{corone}
Fix some nonempty finite set $K\subset \na$ and $\varepsilon>0$. Let
$Z(K,\varepsilon)$ denote the union of all anchored rectangular boxes
which are not multiplicatively $(K,\varepsilon)$-invariant. Then
$Z(K,\varepsilon)$ has density zero with respect to any (not
necessarily rectangular) F{\o}lner \sq\ $(F_n)$ in $\mathbb G$.
\end{lem}

\begin{proof}
For sake of convenience we will write $[0,n]$ instead of
$\{0,1,2,\dots,n\}$. Intuitively, a rectangular box is not
multiplicatively $(K,\varepsilon)$-invariant if it is ``narrow'' in
some direction, and narrow sets have density zero. More precisely, let
$\bar K=[0,k_1]\times[0,k_2]\times\cdots\times[0,k_q]$ be the smallest
rectangular box containing $K$. If a rectangular box
$B=[0,b_1]\times[0,b_2]\times\cdots\times[0,b_r]$ is not
$(K,\varepsilon)$-invariant, then it is not $(\bar
K,\varepsilon)$-invariant, i.e., there is an index
$i\in\{1,2,\dots,q\}$ such that $b_i<\alpha_i$, where
$\alpha_i=\frac{2k_i}{\varepsilon}$. Thus, the union
$Z(K,\varepsilon)$ of all such rectangles $B$ is contained in the
finite union $\bigcup_{i=1}^q X_i$, where $X_i$ is the set of all
vectors in $\mathbb G$ whose $i$th coordinate is smaller than
$\alpha_i$. It is clear that each set $X_i$ has density zero with
respect to any F{\o}lner \sq\ $(F_n)$, because its size in one of the
directions is bounded. The proof is complete since density zero is
preserved under finite unions.
\end{proof}

\begin{thm}\label{uni}
There exists an effectively defined net-normal set $A\subset\mathbb
G$.
\end{thm}

\begin{rem}\label{mage}
Given one net-normal set $A$ we can easily produce a Cantor set of
net-normal elements of $\{0,1\}^\mathbb G$, by altering the indicator
function $\mathbbm 1_A$ in all possible ways along some infinite
subset of $\mathbb G$ which has density zero for all nice F{\o}lner \sq
s (an example of such a subset is provided by any finitely-generated
sub-semigroup; see also Lemma \ref{corone} above).
\end{rem}

\begin{proof}[Proof of Theorem \ref{uni}]
The construction is a modification of the construction of an
$(F_n)$-normal element for a doubling F{\o}lner \sq\ $(F_n)$ (see
Subsection \ref{constru}). The bricks and packages will be the same
(they depend on the choice of the \sq\ $(F_n)$). The mixed tiling will
be different: this time, for each $k\ge 1$ it will contain infinitely
many tiles of $\mathcal T_{r(k)}$.
We can now describe the modification of the mixed
tiling $\Theta$ (or, equivalently, of the chains) appearing in the
construction \ref{constru}. We continue to use the notation
$r(k)=2^k+k$ and keep denoting by $\mathcal T_{r(k)}$ the tiling by
shifted copies of $F_{2^k+k}$. Step $0$ is unchanged: the $0$th chain
is the concatenation of $4$ packages of order zero arranged to fill a
block over $F_3=F_{r(1)}$. In the language of tilings, this defines
$\Theta$ on $F_{r(1)}$ (as a partition into $4$ rectangles), which
clearly is a $\mathcal T_{r(1)}$-saturated set. For $k\ge 1$ assume
that at the steps $1,\dots,k\!-\!1$ we have defined $\Theta$ on a
$\mathcal T_{r(k)}$-saturated set. Now, at the step $k$, we consider
the set $Z(F_k,\frac1k)$, and its saturation $\mathbf Z_k$ with
respect to the tiling $\mathcal T_{r(k+1)}$. Part of $\mathbf Z_k$ has
been tiled in preceding steps (by tiles of orders $\mathcal T_{r(i)}$
with $i<k$), and this part is $\mathcal T_{r(k)}$-saturated. We now
tile the remaining part of $\mathbf Z_k$ by the tiles of $\mathcal
T_{r(k)}$. Due to the congruency of the system of tilings $(\mathcal
T_k)$, in this manner we tile exactly the set $\mathbf Z_k$ (which is
$\mathcal T_{r(k+1)}$-saturated), so that the inductive assumption is
fulfilled for $k\!+\!1$. Notice that the sets $\mathbf Z_k$ eventually
fill up the whole group, thus the mixed tiling $\Theta$ is well
defined on $\mathbb G$ and it determines an element $x=\mathbbm
1_A\in\{0,1\}^\mathbb G$.

By Proposition~\ref{34}, it remains to verify multiplicative normality of $x$ with respect to
any nice F{\o}lner \sq\ $(H_n)$. As in \eqref{normc}, we need to show
that for each $k_0$ we have $\frac{a_n}{|H_n|}\to 1$, where $a_n$ is
the cardinality of the portion $H_n'$ of $H_n$ such that $x|_{H'_n}$
is a concatenation of packages of orders $k\ge k_0$, equivalently, the
portion of $H_n$ tiled by the tiles belonging to $\mathcal T_{r_k}$
with $k\ge k_0$. In other words, we need to show the convergence
$\frac{b_n}{|H_n|}\to 0$, where $b_n$ is the cardinality of the
portion $H_n\setminus H'_n$ of $H_n$ tiled by the tiles belonging to
$\mathcal T_{r_k}$ with $k<k_0$. This convergence follows directly
from three facts:
\begin{itemize}
	\item tiles belonging to $\mathcal T_{r_k}$ with $k<k_0$
          appear only in $\mathbf Z_{k_0}$,
	\item by Lemma \ref{corone}, $Z(F_{k_0},\frac1{k_0})$ has
          $(H_n)$-density zero,
	\item for any F{\o}lner \sq\ $(H_n)$ in any countably infinite
          group $G$, if a set has $(H_n)$-density zero, then so does
          its saturation with respect to any proper (i.e., having
          finitely many shapes) tiling of the group, in particular,
          $\mathbf Z_{k_0}$ has $(H_n)$-density zero. \qedhere
\end{itemize}
\end{proof}

\begin{rem}
Notice that for any countably infinite amenable (semi)group $G$ it
is impossible to find an element $x\in\{0,1\}^G$ which is normal with
respect to all F{\o}lner \sq s. For instance, if $A$ is $(F_n)$-normal
for some F{\o}lner \sq\ $(F_n)$ in $(\na,+)$, then its complement $A^c$
contains arbitrarily long intervals, which constitute a F{\o}lner
\sq\ disjoint from $A$. A similar argument applies to any amenable
semigroup. The existence of a multiplicatively net-normal subset of
$(\na,\times)$ shows that the restriction to anchored rectangular
boxes is a well balanced level of generality.
\end{rem}

\section{Combinatorial and Diophantine properties of additively and
  multiplicatively normal sets in $\na$}\label{s4}

In this section we will be focusing on the combinatorial and
Diophantine richness of normal sets in $(\na,+)$ and $(\na,\times)$.
Before starting the discussion we review some terminology.

\begin{enumerate}[(i)]
\item We say that a set $S\subset\na$ is \emph{additively
  (multiplicatively) large} if \emph{there exists} a F{\o}lner
  \sq\ $(F_n)$ in $(\na,+)$ (resp. $(\na,\times)$) for which
  $\overline d_{(F_n)}(S)>0$.
\item A set $S$ in a semigroup $G$ is called \emph{thick} if it
  contains a right translate of every finite set. The family of thick
  sets in $G$ is denoted by $\mathcal T(G)$. Note that $S\in\mathcal
  T(\na,+)$ if and only if $S$ contains arbitrarily long intervals,
  and that $S\in\mathcal T(\na,\times)$ if and only if $S$ contains
  arbitrarily large sets of the form
  $a_n\{1,2,\dots,n\}=\{a_n,2a_n,\dots,na_n\}$.
\item We say that $S\subset\na$ is \emph{additively normal} if it is
  $(F_n)$-normal for some F{\o}lner \sq\ $(F_n)$ in $(\na,+)$. If
  $(F_n)=(\{1,2,\dots,n\})$, we will call $S$ a \emph{classical normal
  set}. Similarly, a set $S$ is called \emph{multiplicatively normal}
  if it is $(F_n)$-normal for some F{\o}lner \sq\ $(F_n)$ in
  $(\na,\times)$ (there is no classical notion in this
  case).\footnote{We remark that in contrast to the set of classical
    normal numbers (which is of first category, see Corollary
    \ref{normnumberscategory}), the set of additively normal numbers
    is residual. Indeed, given a nonempty finite set $K\subset\na$ and
    $\varepsilon>0$, it is easy to see that the set $S(K,\varepsilon)$
    of all 0-1 sequences $x\in\{0,1\}^\na$, such that there exists an
    interval $I\subset\na$ for which $x|_I$ is
    $(K,\varepsilon)$-normal, is open and dense. The countable
    intersection $\bigcap_{K,\varepsilon}S(K,\varepsilon)$ over all
    nonempty finite sets $K$ and all rational $\varepsilon\in(0,1)$ is
    residual and consists of additively normal \sq s. Residuality of
    the set of additively normal numbers in $[0,1]$ now follows by a
    proof similar to that of Corollary \ref{normnumberscategory}.

An analogous argument establishes the residuality of the set of
multiplicatively normal numbers.}

\item Let $(n_i)_{i=1}^\infty$ be a \sq\ of (not necessarily distinct)
  positive integers. The set
\[
FS(n_i)_{i=1}^\infty =
\{n_{i_1}+n_{i_2}+\cdots+n_{i_k}:\ i_1<i_2<\ldots<i_k,\ k\in\na\}
\]
is called an \emph{additive IP-set}. Likewise, the set
\[
FP(n_i)_{i=1}^\infty= \{n_{i_1}n_{i_2}\cdots
n_{i_k}:\ i_1<i_2<\ldots<i_k,\ k\in\na\}
\]
is called a \emph{multiplicative IP-set}.
\end{enumerate}

\subsection{Multiplicative versus additive density and normality --
  some basic observations}

Recall that $(F_n)$-normality (additive or multiplicative) of an
element $x\in\{0,1\}^\na$ is defined as the property that for any
finite set $K\subset\na$ and every block $B\in\{0,1\}^K$ the
(additive or multiplicative) shifts of $B$ \emph{occur} in $x$ with
$(F_n)$-density $2^{-|K|}$. We emphasize that the additive and
multiplicative shifts of a block are quite different. For example, if
$w$ is a word over $\{1,2,\dots,k\}$, its additive shift occurs at a
position $n$ of some $x$ if $x|_{\{n+1,n+2,\dots, n+k\}}=w$, while its
multiplicative shift occurs at $n$ if $x|_{\{n,2n,3n,\dots,kn\}}=w$.

Theorem \ref{B1} implies that for any F{\o}lner \sq\ $(F_n)$ in
$(\na,+)$ such that $|F_n|$ increases, $\lambda$-almost every
$x\in\{0,1\}^\na$ is $(F_n)$-normal and, similarly, for any F{\o}lner
\sq\ $(K_n)$ in $(\na,\times)$ such that $|K_n|$ increases,
$\lambda$-almost every $x$ is $(K_n)$-normal. So $\lambda$-almost
every $x$ is both additively $(F_n)$-normal and multiplicatively
$(K_n)$-normal. On the other hand, the two notions of normality are
``in general position'': additively normal sets can be
multiplicatively trivial (have multiplicative density $0$ or $1$), and
vice-versa, multiplicatively normal sets can have additive density $0$
or~$1$. More precisely, the following holds.

\begin{thm}\label{orto}
For any F{\o}lner \sq\ $(K_n)_{n\in\na}$ in $(\na,\times)$ there exists
a set $A\subset\na$ with $(K_n)$-density $1$ and having universal
additive density $0$ (here \emph{universal} means that the additive
density can be computed with respect to an arbitrary F{\o}lner \sq\ in
$(\na,+)$).
\end{thm}

\begin{proof}
First observe that for any $m$ the set $m\na$ (being a multiplicative
shift of a set of $(K_n)$-density $1$) has $(K_n)$-density~$1$. Let
$n_m$ be such that for each $n> n_m$ the fraction of multiples of $m!$
in $K_n$ is larger than $1-\frac1m$. The set $A$ is defined as the
union $K_1\cup K_2\cup\cdots\cup K_{n_2}$ to which we add all
multiples of $2!$ contained in the union $K_{n_2+1}\cup
K_{n_2+2}\cup\cdots\cup K_{n_3}$, all multiples of $3!$ contained in
the union $K_{n_3+1}\cup K_{n_3+2}\cup\cdots\cup K_{n_4}$, etc. It is
obvious that the $(K_n)$-density of $A$ equals $1$. On the other
hand, since $A$ has gaps which tend to infinity in length, its
additive density is equal to $0$ (for any additive F{\o}lner \sq).
\end{proof}

Note that the complementary set $A^c$ has $(K_n)$-density $0$ and
universal additive density~$1$. Further, given a F{\o}lner \sq\ $(F_n)$
in $(\na,+)$, let $B\subset\na$ be a set which is both
multiplicatively $(K_n)$-normal and additively $(F_n)$-normal. Then
$A\cap B$ is multiplicatively $(K_n)$-normal, while it has universal
additive density zero, and on the other hand, $A^c\cap B$ is
additively $(F_n)$-normal and has multiplicative $(K_n)$-density~$0$.
These examples justify our claim above that the notions of
multiplicative and additive normality are in ``general position''.

\begin{rem}\label{orto1}
A statement symmetric to Theorem \ref{orto}, in which one fixes a
F{\o}lner \sq\ $(F_n)$ in $(\na,+)$ (for instance the classical one) and
looks for a set of $(F_n)$-density~$1$ and universal multiplicative
density $0$, does not hold. As a matter of fact, any set of upper
density $1$ with respect to the classical F{\o}lner \sq\ in $(\na,+)$
has density $1$ with respect to some F{\o}lner \sq\ in $(\na,\times)$.
Indeed, in the proof of \cite[Theorem 6.3]{BeM16} it is shown that if
$A$ has classical upper density $1$, so does $A/n\cap A$ for every $n$
(see Definition \ref{def A/n} below). By an obvious iteration, we get
that $A\cap A/2\cap A/3\cap\dots\cap A/n$ is nonempty, which implies
that $A$ contains arbitrarily large sets of the form
$a_n\{1,2,\dots,n\}$, i.e., $A$ is multiplicatively thick (see (ii)
above). This, in turn, implies that for some F{\o}lner \sq\ $(K_n)$ in
$(\na,\times)$ one has $d_{(K_n)}(A)=1$.

On the other hand, there are sets $A\subset(\na,+)$ with
$d(A)=1-\varepsilon$ such that $A$ has universal multiplicative
density zero (take for example all numbers not divisible by some large
$n$).
\end{rem}

\subsection{Elementary combinatorial properties of additively and
  multiplicatively normal sets}

The above Theorem \ref{orto} and Remark \ref{orto1} hint that, in
general, the combinatorial properties of additively and
multiplicatively normal sets are distinct. We will see below that this
is indeed the case.

In this subsection we will focus on properties of
additively/multiplicatively normal sets which follow from the fact
that these sets are additively/multi\-plica\-tively thick. Since we
are interested in properties of normal sets, in the statements of our
theorems we will make the ostensibly stronger assumption that the sets
in question are normal rather than just thick. Note that, since every
thick set obviously contains an $(F_n)$-normal set for some F{\o}lner
\sq\ $(F_n)$, in all theorems in this subsection the normality and
thickness assumptions are in fact equivalent.

For example, it is not hard to see that every thick, in particular
every normal, set contains an IP-set (this applies to both additive
and multiplicative setups). Now, IP-sets can be defined as solutions
of (an infinite) system of certain equations, and in our quest for
patterns in normal sets, it is natural to inquire which Diophantine
equations and systems thereof are always solvable in normal sets. The
following two theorems shed some light on this question.

\begin{thm}\label{tratra}
If $S$ is a multiplicatively normal set then any homogeneous system of
finitely many polynomial equations (with several variables) which is
solvable in $\na$ is solvable in $S$.
\end{thm}

\begin{proof}
Since $S\in \mathcal T(\na,\times)$ (i.e., is multiplicatively thick),
$S$ contains arbitrarily long sets of the form $a_n\{1,2,\dots,n\}$.
If a given homogeneous system is solvable in $\na$ then it is solvable
in $\{1,2,\dots,n\}$ for some $n$ and hence, due to homogeneity,
also in $a_n\{1,2,\dots,n\}$.
\end{proof}

\begin{rem}
Note that it follows from Theorem \ref{tratra} that any
multiplicatively normal set contains, for any $m\in\na$, ``finite-sums
sets'' of the form
\[
FS(n_i)_{i=1}^m =
\{n_{i_1}+n_{i_2}+\cdots+n_{i_k}:\ i_1<i_2<\cdots<i_k\le
m,\ k\in\{1,2,\dots,m\}\}.
\]
Indeed, these sets can be described as solutions of finite homogeneous
systems of linear equations\footnote{The set $FS(n_i)_{i=1}^m$ is the
  solution of the following system of equations (with variables
  $n_T$):
\[
n_T = \sum_{i\in T}n_i,
\]
where $T$ ranges over all nonempty finite subsets of the set $\{1, 2,
..., m\}$ (this applies also to $m=\infty$).}. On the other hand, we
will now show that, in general, multiplicatively normal sets need not
contain additive IP-sets $FS(n_i)_{i=1}^\infty$ or shifts thereof.
Take any F{\o}lner \sq\ $(F_n)$ in $(\na,\times)$ and let $A$ be the
set of $(F_n)$-density $1$ constructed in the proof of Theorem
\ref{orto} ($A$ has universal additive density zero). For each $n$,
this set contains only finitely many numbers not divisible by $n$. On
the other hand, it is well known (and also easy to see) that every
additive IP set contains, for arbitrarily large $n$, infinitely many
numbers divisible by $n$ as well as infinitely many numbers not
divisible by $n$. Thus any (shifted or not) additive IP-set contains
infinitely many numbers not divisible by $n$, and hence cannot be
contained in $A$.
\end{rem}

\begin{rem}
We remark that additively normal sets (even the classical ones) need
not contain multiplicative IP-sets. In fact, they do not need to
contain triples of the form $\{a,b,ab\}$ (see \cite{Fi05}).
\end{rem}

\begin{thm}\label{5.1}
Let $A\overset{\to}x=0$ be a partition-regular (see Introduction) system of finitely many
linear equations with $n$ variables. Then, for any additively normal
set $S$ one can find a solution $\overset{\to}x=(x_1,x_2,\dots,x_n)$
with all entries in $S$.
\end{thm}

\begin{proof}
The proof is short but uses same facts from Ramsey theory,
\tl\ dynamics and \tl\ algebra in the Stone--\v Cech compactification
$\beta\na$ viewed as a semi\tl\ semigroup obtained by an extension of
the operation in $(\na,+)$. Since this theorem forms only a rather
small fragment of a big picture, in order to save space, we will be
using some terms and results without giving all the needed details
(but remedying this by providing pertinent references).

First, note that our additively normal set is thick and hence is a
member of a minimal idempotent in $(\beta\na,+)$. Further, any member
of a minimal idempotent in $(\beta\na,+)$ is a central set (see, for
example, Definition 5.8 and Lemma 5.10 in \cite{Be10}). Now it only
remains to invoke the theorem due to Furstenberg which states that any
central set contains solutions to any partition-regular system
$A\overset{\to}x=0$ (\cite[Theorem~8.22]{Fu81}).\footnote{See
  \cite[Theorem 4.1]{Fi11} for a more general result of this kind,
  obtained by a different method.}
\end{proof}

One can actually show that a system of linear equations is
partition-regular if and only if it is solvable in any additively
normal set (equivalently, in any thick set). For sake of simplicity we
prove this equivalence in the case of one equation with three
variables. Note that any such equation (which has at least one
solution) can be written as $ia+jb=kc$ with $i,j,k\in\na$ and $a,b,c$
as unknowns.

\begin{thm}\label{part-reg}
Let $i,j,k$ be three natural coefficients. The following conditions
are equivalent:
\begin{enumerate}
	\item $k\in\{i,j,i+j\}$,
	\item the equation $ia+jb=kc$ is partition-regular,
	\item the equation $ia+jb=kc$ is solvable in any thick set,
	\item the equation $ia+jb=kc$ is solvable in any additively
          normal set.
\end{enumerate}
\end{thm}

\begin{proof}
Equivalence of (1) and (2) is well known. As a matter of fact, a
necessary and sufficient condition for partition-regularity of an
equation $i_1a_1 +i_2a_2+\dots+i_na_n = 0$ is that some subset of
coefficients sums up to zero, see for example \cite{GRS90}.
Conditions~(3) and~(4) are equivalent since every additively normal
set is thick, while every thick set contains the union $\bigcup_n F_n$
of some F{\o}lner \sq\ and---within this union---an additively normal
set. Solvability of partition-regular linear equations in thick (and
hence additively normal) sets is our Theorem \ref{5.1}. It remains to
consider coefficients for which (1) does not hold and construct a
thick set $A\subset\na$, which contains no solutions, i.e., is such
that $(iA+jA)\cap kA=\emptyset$.

Since $k\notin\{i,j,i+j\}$ there exist a rational number $\delta>0$
such that
\[
k[1,1+\delta]\cap (i[1,1+2\delta]\cup
j[1,1+2\delta]\cup(i+j)[1,1+\delta])=\emptyset.
\]
Let
\[
A=\bigcup_{n=1}^\infty I_n, \text{ \ where \ }I_n=r_n[1,1+\delta],
\]
where the numbers $r_n\in\na$ are such that $r_n\delta\in\na$ and grow
geometrically with a large ratio. Obviously, $A$ is a thick set.
Choose any $a,b\in A$. If $a,b$ belong to the same interval $I_n$,
then $ia+jb\in (i+j)r_n[1,1+\delta]$. If $a,b$ belong to two different
intervals, say $I_m, I_n$ with $m<n$, then, since $r_m$ is much
smaller than $r_n$, $ia+jb$ is either in $ir_n[1,1+2\delta]$ or in
$jr_n[1,1+2\delta]$. In any case, $ia+jb\notin kI_n$ and, due to the
fast growth of $r_n$, $ia+jb\notin kI_l$ for any other $l$. So,
$(iA+jA)\cap kA=\emptyset$, as needed.
\end{proof}

\begin{rem}
We will show later (see Corollary \ref{aaaa}) that any
\emph{classical} normal set $A$ has the stronger property that any
equation $ia+jb=kc$ with $i,j,k\in\na$ is solvable in $A$.
\end{rem}

We conclude this subsection with a simple observation that additively
normal sets always contain at least some modest amount of
multiplicative structure.

\begin{thm}
Any additively normal set $A$ contains ``consecutive product sets'' of
the form $\{y_1,y_1y_2,\dots,y_1y_2\cdots y_k\}$ with arbitrarily
large $k$ and $y_n\ge 2$.
\end{thm}

\begin{proof}
The result follows from the (almost obvious) fact that any thick set
contains arbitrarily large product sets.
\end{proof}

\subsection{Covering property of translates of normal sets}

The special case (for $(\z,+)$) of the following result is implicit in
\cite{BeW85}. We give a short proof for arbitrary countably infinite
amenable groups (and cancellative semigroups).

\begin{lem}\label{AA}
Let $G$ be a countably infinite amenable group in which we fix
arbitrarily a F{\o}lner \sq\ $(F_n)$. If $A$ is an $(F_n)$-normal set
and $B\subset G$ is infinite, then the set $BA$ has $(F_n)$-density
$1$.
\end{lem}

\begin{proof}
Observe that if $K\subset G$ is nonempty finite then $KA$ has
$(F_n)$-density precisely $1-2^{-|K|}$. Indeed, $g\notin KA$ is
equivalent to $K^{-1}g\cap A=\emptyset$, i.e., the indicator function
$\mathbbm 1_A|_{K^{-1}g}=0$. By normality of $A$, the last equality
holds for elements $g$ whose $(F_n)$-density is $2^{-|K|}$. If $B$ is
infinite, the lower $(F_n)$-density of $BA$ is larger than $1-2^{-k}$
for any $k$, so $BA$ has $(F_n)$-density $1$.
\end{proof}

\begin{cor}\label{wniosek1}
By Theorem \ref{sgrvsgr}, Lemma \ref{AA} holds in countably infinite
amenable cancellative semigroups, in particular in $(\na,+)$
and $(\na,\times)$. Moreover, we also have that $B^{-1}\!A$ defined as
the set of such $g\in G$ that $bg\in A$ for some $b\in B$, has
$(F_n)$-density $1$.
\end{cor}

\begin{rem}\label{vitaly}
The following useful observation generalizes \cite[Theorem 2]{BeW85}:
If $C\subset G$ has positive upper $(F_n)$-density then $bA\cap C$ has
$(F_n)$-density zero for at most finitely many $b\in\na$. Indeed, if
$k$ is such that $\overline d_{(F_n)}(C)>2^{-k}$, then for any
$K\subset\na$ with $|K|=k$ one has $\overline d_{(F_n)}(KA\cap C)>0$.
If there were $k$ different elements $b_1,\dots,b_k\in\na$ satisfying
$d_{(F_n)}(b_iA\cap C)=0$, then the set $K=\{b_1,\dots,b_k\}$ would
violate the last inequality.
\end{rem}

\begin{exam}
The following example, in the classical setup of $(\na,+)$, shows that
for an additively (in particular, classical) normal set $A$ the
complement of $B+A$ need not be finite, even if $B=A$. Start the
construction by choosing a finite word $w_1$ with good normality
properties. Let $\tilde w$ denote the block obtained from $w$ by
switching zeros and ones and writing the symbols in reverse order. The
concatenated word $v_1=w_1\tilde w_1$ also has good normality
properties and is antisymmetric, i.e., it satisfies $v_1(k) =
1-v_1(n_1-k)$, for all $0\le k< n_1$, where $n_1$ is the length of
$v_1$. Note that if $A$ is any set whose indicator function $\mathbbm
1_A$ starts with $v_1$ then $A+A$ misses $n_1$. Let $w_2$ be a word
much longer than $w_1$ and with much better normality properties.
Define $v_2$ as the concatenation $v_1w_2\tilde w_2v_1$. This word
starts with $v_1$, is antisymmetric and has nearly as good normality
properties as $w_2$. If $\mathbbm 1_A$ starts with $v_2$ then $A+A$
misses both $n_1$ and $n_2=|v_2|$. Continuing in this fashion we will
end up with an infinite set $A$ which is normal and such that $A+A$
misses infinitely many integers.
\end{exam}

\begin{rem}
On the other hand, thickness alone easily implies that $A-A=\na$
(where $A-A$ is understood as the set of positive differences of
elements from $A$).
\end{rem}

\subsection{Divisibility properties of classical normal sets. First applications}

Until the end of Section \ref{s4} we will be dealing with classical
normal sets in $(\na,+)$ (and also, briefly, with net-normal sets in
$(\na,\times)$). As we will see, they exhibit especially rich
combinatorial structure (not shared by general additively or
multiplicatively normal sets). In this subsection we focus on general
linear equations with three variables in classical normal sets.

\begin{defn}\label{def A/n}
Given $A\subset\na$, and $n\in\na$, denote by $A/n$ the set $\{m:nm\in
A\}$ (formally, this is $\frac1n(A\cap n\na)$ or, invoking the
multiplicative shift, $\mathbbm 1_{A/n}=\rho_n(\mathbbm 1_A)$).
\end{defn}

\begin{lem}\label{A/n}
If $A$ is a classical normal set, so is $A/n$ for any $n\in\na$.
\end{lem}

\begin{rem}
\begin{enumerate}[(i)]
	\item Lemma \ref{A/n} says that $(x_k)\in\{0,1\}^\na$ is
          classical normal if and only if, for every $n\in\na$, the
          \sq\ $(x_{nk})$ is classical normal. This result was proved
          in D. Wall's thesis \cite{Wa49}. We provide a different,
          ergodic proof.
	\item A nontrivial fact which is implicitly used in the proof
          is the \emph{divisibility property} of the classical
          F{\o}lner \sq\ $F_n=\{1,2,\dots,n\}$ in $(\na,+)$: for any
          $k\in\na$, $(F_n/k)$ is essentially \emph{the same} F{\o}lner
          \sq. For example, for $k=3$, $(F_n/k)$ is
          $(\emptyset, \emptyset,
          F_1,F_1,F_1,F_2,F_2,F_2,F_3,F_3,F_3,\dots)$.
\end{enumerate}
\end{rem}

\begin{proof}[Proof of Lemma \ref{A/n}]
Suppose $A/n$ is not normal, i.e., some word $w$ having length $k$
does not occur in the indicator function $\mathbbm 1_{A/n}$ of $A/n$
with the correct frequency $2^{-k}$. This means that the ``scattered''
block $\hat w$ (in which the entries of $w$ appear along the
arithmetic progression $\{n,2n,\dots,kn\}$) occurs in $\mathbbm 1_A$
starting at coordinates $m\in n\na$ with, say, upper density different
from $2^k\frac1n$. Let $y$ be the periodic \sq\ $y(i)=1\iff n|i$. Now
consider the pair $(\mathbbm 1_A,y)$ (it is convenient to imagine this
pair as a two-row \sq\ with $\mathbbm 1_A$ written above $y$). This
pair is a \sq\ over four symbols in $\{0,1\}^2$. This means that
the scattered block $(\hat w,\hat v)$ where $\hat v$ denotes the block
of just $1$'s at each bottom position $\{n,2n,\dots,kn\}$, occurs in
$(\mathbbm 1_A,y)$ with the upper density different from $2^k\frac1n$.
There is a subsequence $n_i$ such that the upper density is achieved
along intervals $\{1,\dots,n_i\}$, and moreover, the corresponding
\sq\ of normalized counting measures supported by the sets
\[
\{(\mathbbm 1_A,y),\sigma((\mathbbm 1_A,y)),\sigma^2((\mathbbm
1_A,y)),\dots,\sigma^{n_i}((\mathbbm 1_A,y))\}
\]
(where $\sigma((x,y))=(\sigma(x),\sigma(y))$) converges in the
weak-star topology to a shift-\im\ $\mu$ on $(\{0,1\}^2)^\na$. Since
$\mathbbm 1_A$ is normal (i.e., generic for the Bernoulli measure
$\lambda$) and $y$ is periodic (hence generic for the unique
\inv\ probability measure $\xi$ on the periodic orbit of $y$), the
marginal measures of $\mu$ are $\lambda$ and $\xi$. Now, $\mu([\hat
  w]\times[\hat v])\neq 2^k\frac1n = \mu([\hat w])\xi([\hat v])$,
which means that $\mu\neq \lambda\times\xi$. This contradicts
disjointness of Bernoulli measures from periodic measures (which is a
particular case of disjointness between K-systems and entropy zero
systems, see \cite{Fu67}).
\end{proof}

We can now derive another fact in the classical case.

\begin{thm} \label{A+A}
If $A$ is a classical normal set and $n,m\in\na$ are coprime, then
both $nA+mA$ and $nA-mA$ (restricted to $\na$) have density $1$.
\end{thm}

\begin{proof}
The theorem follows from the fact that, relatively, in every residue
class $\!\!\!\!\mod n$, (i.e., in the set $n\na+i$ for each
$i=0,1,\dots,n-1$), the set $nA\pm mA$ has density $1$. Indeed, since
$m,n$ are coprime, the set $(mA\mp i)/n$ is infinite (this follows
already from the thickness of $A$). Then, by Lemma \ref{AA}, the set
$A\pm(mA\mp i)/n$, which we can write as $(nA\pm mA-i)/n$, has density
$1$. Hence $nA\pm mA$ has density $1$ in the residue class of $i$.
\end{proof}

\begin{cor}\label{aaaa}
For any $i,j,k\in\na$ there are $a,b,c\in A$ solving the equation
$ia+jb=kc$.\footnote{For a more general result of this type, presented
  in a different language and with a different proof see \cite[Theorem
    1.3.2]{Fi11}.}
\end{cor}

\begin{proof}
Restricting to $n\na$, where $n=\mathsf{LCD}(i,j,k)$, we can assume
that some two coefficients, for example $i$ and $j$, are coprime (the
other cases can be treated similarly). By Theorem \ref{A+A}, $iA+jA$
has relative density $1$ in $n\na$, while $kA$ has positive relative
density in $n\na$, so $(iA+jA)\cap kA\neq\emptyset$.
\end{proof}

Actually, one has a more general fact. In the theorem below we use
the following terminology: a set $B\subset\na$ is called
\emph{divisible} if it contains multiples of every natural number $n$
(note that then $B/n$ is infinite for each $n$), and it is called
\emph{substantially divisible} if $B/n$ has positive upper density for
every $n$.

\begin{thm}\label{ixjykz}
Let $A,B,C$ be subsets of $\na$ and assume that $A$ is classical
normal, $B$ is divisible, and $C$ is substantially divisible.
Fix any $i,j,k\in\na$. Then the equation $ia+jb = kc$
is solvable with $a\in A, b\in B, c\in C$.
\end{thm}

\begin{rem}
The assumptions are satisfied when $A,B,C$ are classical normal sets
(the special case $A=B=C$ was treated in Corrolary \ref{aaaa}).
\end{rem}

\begin{proof}[Proof of Theorem \ref{ixjykz}]
Note that the set $kC/i$ can be interpreted in two ways: as $k\cdot
C/i$ or as $(kC)/i$ with the latter set being possibly larger than the
former. Nevertheless, both sets have positive upper density. Also,
regardless of the interpretation, the set $jB/i$ is infinite. By Lemma
\ref{AA}, $(A+jB/i)\cap kC/i$ has positive upper density (the same as
$kC/i$). Multiplying by $i$ we obtain that $(iA+jB)\cap kC$ has
positive upper density, in particular is nonempty. So, there exist
(many) desired solutions.
\end{proof}

\subsection{Solvability of certain equations in net-normal sets}\label{ucomb}\hfill

\noindent
Motivated by the preceding subsection, let us now turn to the
multiplicative semigroup $(\na,\times)$ and multiplicative normality.
The analogue of the equation $ia+jb=kc$ reads $a^ib^j=c^k$. To see
this analogy even better, let us view $(\na,\times)$ again as the
direct sum $\mathbb G$ which is an additive semigroup. Now the
multiplicative equation $a^ib^j=c^k$ takes on the familiar additive
form $ia+jb=kc$. The problem we immediately encounter in this
``infinite-dimensional'' semigroup is that if $A$ is multiplicatively
normal (even net-normal) then the set $A/n$ (multiplicatively this is
the set $\{m:m^n\in A\}$) need not be multiplicatively normal. In
fact, it can even be empty, because the set $n{\mathbb G}$
(multiplicatively this is the set of $n$th powers) has universal
multiplicative density zero. Below we provide an easy example of
failure for the multiplicative equation $a^2b^2=c^3$, regardless of
the F{\o}lner \sq\ in $(\na,\times)$.

\begin{exam}
Let $(F_n)$ be a F{\o}lner \sq\ in $(\na,\times)$ and let $A$ be an
$(F_n)$-normal set (alternatively, it can be net-normal). By removing
from $A$ all squares (note that the set of squares is a set of
universal multiplicative density $0$), we can assume that $A$ contains
no squares. Then the elements $c^3$ with $c\in A$ are not squares
either. Thus $A$ contains no solutions of $a^2b^2=c^3$.
\end{exam}

This is why we will restrict our attention only to the case with
$i=j=1$, i.e., consider only equations of the form $ab=c^k$.

\begin{thm}\label{xyz3}
Let $A$ and $B$ be net-normal sets and let $C$ contain infinitely many
pairs $(n,rn)$, where $r$ is fixed, while $n$ tends to infinity with
respect to the multiplicative order $\preccurlyeq$ (i.e., for any
$k\in\na$, large enough $n$ is a multiple of $k$; this holds, for
instance, if $C$ is $(F_n)$-normal with respect to some fixed F{\o}lner
\sq\ in $(\na,\times)$). Then for any natural $k$ there exist $a\in A,
b\in B$ and $c\in C$ such that $ab=c^k$.
\end{thm}

\begin{cor}
If $A$ is net-normal then for any natural $k$ the equation $ab=c^k$
is solvable in $A$.
\end{cor}

\begin{proof}[Proof of Theorem \ref{xyz3}]
We continue to switch freely between the sets $A, B, C$ and their
indicator functions denoted $\mathbbm 1_A, \mathbbm 1_B, \mathbbm
1_C$. We view $(\na,\times)$ again as the additive semigroup $\mathbb
G$. Thus our task becomes to find solutions of the equation $a+b = kc$
with $a\in A, b\in B, c\in C$. From now
on, adjectives ``small'', ``nearly'', ``close'', etc will refer to
quantities (error terms, distances) that are estimated above by
functions of $\varepsilon$ tending to zero as $\varepsilon\to 0$.
Fix a small $\varepsilon>0$.
By the assumption, we can find in $C$ two elements $n$ and $n+r$, where
($r\in\mathbb G$ is fixed a priori), with $n$ multiplicatively so
large that any anchored rectangle $F$ with a leading parameter
multiplicatively larger than or equal to $3n$ is
$(kr,\varepsilon)$-invariant and has the property that both $A$ and
$B$ have in $F$ a proportion nearly $\frac12$ (it is here that we are
using net-normality of $A$ and $B$).

Now suppose that there are no triples $a\in A, b\in B, c\in C$
satisfying $a+b = kc$. This implies that within the rectangle $F$ with
the leading parameter $kn$, $A$ is disjoint from $kn-B$. Since the
proportion of both sets in $F$ is nearly $\frac12$, these two sets are
in fact nearly complementary within $F$, i.e., we can write $\mathbbm
1_A(m)=1\iff \mathbbm 1_B(kn-m)=0$ and this will be true except for a
small percentage of $m$'s in $F$. The same holds with $n+r$ replacing
$n$ within the rectangle $F'$ with the leading parameter $k(n+r)$, in
particular, also in $F$ (because by $(kr,\varepsilon)$-invariance, $F$
is negligibly smaller than $F'$). This implies that the configuration
of symbols in $\mathbbm 1_A$ (and also in $\mathbbm 1_B$) within $F$
is nearly invariant under the shift by $kr$, i.e., in most places
$m\in F$ the symbols at $m$ and $m+kr$ are the same. This contradicts
net-normality of $\mathbbm 1_A$ (and likewise of $\mathbbm 1_B$): if
$F$ is large enough then the proportion of pairs of identical symbols
at positions $m$ and $m+kr$ with $m\in F$ should be close to
$\frac12$, not to $1$.
\end{proof}

\begin{exam}\label{ex9}
Using an idea similar to that utilized in the proof of Theorem
\ref{part-reg}, we will show that assuming multiplicative normality
with respect to just one nice F{\o}lner \sq\ may be insufficient for
the solvability of the equation $ab=c^3$. We continue to use the
additive notation of $\mathbb G$. Let $L_n$ be a multiplicatively
increasing to infinity \sq\ of natural numbers. We assume that
$5L_n\preccurlyeq L_{n+1}$ (recall that multiplicatively this means
$L^5_n|L_{n+1}$). Let $F_n$ be the rectangle with the leading
parameter $3L_n$ and let $B_n$ be $F_n$ with the rectangle with the
leading parameter $2L_n$ removed. Let $B=\bigcup_n B_n$. Note that as
soon as $L_n$ is high-dimensional, say of a large dimension $d$
(multiplicatively, this means that $L_n$ is a product of [powers of]
$d$ different primes) then $B_n$ constitutes the large fraction
$1-(\frac23)^d$ of $F_n$. It is now obvious that $B$ has
$(F_n)$-density $1$. Consider the sum $a+b$ of two elements of $B$.
Let $n$ be the maximal index such that $B_n$ contains either $a$ or
$b$. Then $a+b$ belongs to the rectangle with the leading parameter
$6L_n$ with the rectangle with the leading parameter $2L_n$ removed
(call this difference $C_n$). It is easy to see (it suffices to
consider the one-dimensional case) that the union $\bigcup_n C_n$ is
disjoint from $3B$. We have shown that $a+b=3c$ has no solutions in
$B$. Since $B$ has $(F_n)$-density~$1$, it now suffices to intersect
it with any $(F_n)$-normal set to get an $(F_n)$-normal set without
the considered solutions.
\end{exam}

It is now natural to ask: are all multiplicative equations $ab=c^k$
solvable in classical normal sets? Here the answer is known to be
negative. In \cite{Fi05}, A.~Fish constructed normal sets of the form
$A=\{n: f(n) = -1\}$, where $f$ is a multiplicative function
(so-called random Liouville function) $f:\na\to\{-1,1\}$. In such sets
there are clearly no solutions of the equations $ab=c^k$ for any odd
$k$.\footnote{On the other hand, the equation $ab=c^2$ is solvable in
  any classical normal set. This follows from the fact that classical
  normal sets contain geometric progressions of length 3, see
  Theorems~\ref{classicalarerich} or \ref{ssss} below.}

\subsection{Pairs $\{a+b, ab\}$ in classical additively normal sets}\label{secmn}

In this subsection we establish yet another nontrivial property of classical normal sets.

\begin{thm}\label{xy x+y}
Let $A$ be a classical normal set. For given $a\in\na$ define
\[
S_a=\{b: a+b\in A, ab\in A\}.
\]
Then for every $a\in\na$ either $S_a$ or $S_{a^2}$ has positive upper
density. In particular, $A$ contains pairs $\{a+b,ab\}$ with
arbitrarily large $a$ and $b$.
\end{thm}

\begin{rem}\label{stip}
The property stipulated in Theorem \ref{xy x+y} does not necessarily
hold for general additively normal sets. Indeed, one can construct an
additively thick set which does not contain pairs $\{a+b,ab\}$
\cite[Theorem 6.2]{BeM16}. Clearly, such a thick set contains an
additively normal set with no pairs $\{a+b,ab\}$.
\end{rem}

\begin{proof}[Proof of Theorem \ref{xy x+y}]
It follows from the definition of the set $A/n$ that $b\in S_a\iff
b\in A/a\cap(A-a)$. Fix some $a\ge 2$ and suppose that both $S_a$ and
$S_{a^2}$ have density zero. This can be written as
\[
A\cap(A/a+a)\approx\emptyset \text{ \ and \ }A\cap
(A/a^2+{a^2})\approx\emptyset,
\]
where $\approx$ means equality up to a set of density zero. Since
every set in the above intersections has density $\frac12$, we get
$A/a+a\approx \na\setminus A$ and $A/a^2+a^2\approx \na\setminus A$,
and in particular
\[
A/a+a\approx A/a^2+a^2.
\]
Multiplying both sides by $a$ we obtain
\[
(A\cap a\na) + a^2\approx (A/a\cap a\na) + a^3.
\]
Since $A/a$ and $A-a$ are nearly disjoint (the intersection has zero
density), we also have that $(A/a\cap a\na) + a^3$ is nearly disjoint
from $(A\cap a\na)+a^3-a$. Plugging this into the last displayed
formula we conclude that $(A\cap a\na)$ is nearly disjoint from
$(A\cap a\na)+a^3-a^2-a$. Dividing both sets by $a$, we get that $A/a$
is nearly disjoint from $A/a+a^2-a-1$. Since $A/a$
has density $\frac12$, we have proved that the indicator function of $A/a$ has the
property that for $n$'s of density $1$ its values at $n$ and at $n+r$
(where $r=a^2-a-1$) are different. This contradicts Lemma \ref{A/n}
(normality of $A/a$), as in normal sets the density of such $n$s
should be $\frac12$.
\end{proof}

\subsection{Multiplicative configurations in classical normal sets}

In this subsection we show that every classical normal set contains
(up to scaling) all configurations which are known to be present in
multiplicatively large sets. The following theorem is the main
technical result allowing us to prove this fact.

\begin{thm}\label{winner}
Let $A\subset\na$ be a classical normal set. Then, for any F{\o}lner
\sq\ $(K_n)$ in $(\na,\times)$ there exists a set $E$ of
$(K_n)$-density $\frac12$ such that for any nonempty finite subset
$\{n_1,n_2,\dots,n_k\}\subset E$ the intersection $A/n_1\cap
A/n_2\cap\cdots\cap A/n_k$ has positive upper density in $(\na,+)$.
\end{thm}

The key role in the proof of Theorem~\ref{winner} will be played by
the following theorem (cf. ~\cite[Theorem~4.19]{Be97}
and~\cite[Theorem~2.1]{Be85}).

\begin{thm}\label{4.24}
Let $(F_n)$ be a F{\o}lner sequence in $(\na,+)$, let $a\in(0,1)$, and
let ${\mathcal F}=\{A_1,A_2,\dots\}$ be a countable family of subsets
in $\na$ such that $d_{(F_n)}(A)\ge a$ for all $A\in{\mathcal F}$.
Then there exists an invariant mean $L$ on the space $B_{\mathbb
  C}(\na)$ of bounded complex-valued functions such that
\begin{enumerate}[(i)]
\item $L(\mathbbm 1_A)=d_{(F_n)}(A)$ for every $A\in{\mathcal F}$,
\item for any $k\in\na$ and any $n_1,n_2,\dots,n_k\in\na$,
\[
\overline d_{(F_n)}(A_{n_1}\cap A_{n_2}\cap\dots\cap A_{n_k})\ge
L(\mathbbm 1_{A_{n_1}}\cdot \mathbbm 1_{A_{n_2}}\cdot \ldots \cdot \mathbbm 1_{A_{n_k}}),
\]
\item there exists a compact metric space $X$, a regular measure $\mu$
on $\B(X)$ (the Borel $\sigma$-algebra of $X$), and sets $\tilde
A_n\in \B(X)$, $n\in\na$, such that for any $n_1,n_2,\dots,n_k\in\na$
one has
\[
L(\mathbbm 1_{A_{n_1}}\cdot \mathbbm 1_{A_{n_2}}\cdot \ldots \cdot \mathbbm 1_{A_{n_k}})=
\mu(\tilde A_{n_1}\cap \tilde A_{n_2}\cap\dots\cap \tilde A_{n_k}).
\]
\end{enumerate}
\end{thm}

\begin{proof}
In the proof, when convenient, we will view $L$ as a finitely additive
measure on the family $\p(\na)$ of all subsets of $\na$.
Let $\s$ be the (countable) family of all finite intersections of the
form $A_{n_1}\cap A_{n_2}\cap\dots\cap A_{n_k}$, where $A_{n_j}\in
{\mathcal F}$, $j=1,\dots,k$. By using the diagonal procedure we
arrive at a subsequence $(F_{n_i})$ of our F{\o}lner sequence $(F_n)$,
such that for any $S\in\s$ the limit
\[
L(S)=\lim_{i\to\infty}\frac{|S\cap F_{n_i}|}{|F_{n_i}|}=
\lim_{i\to\infty}\frac1{|F_{n_i}|}\sum_{m\in F_{n_i}}\mathbbm 1_S(m)
\]
exists. Notice that $L(A)=d_{(F_n)}(A)$ for any $A\in{\mathcal F}$,
and that for any $n_1,n_2,\dots,n_k\in\na$ we have
\begin{multline*}
\overline d_{(F_n)}\left(\bigcap_{j=1}^k A_{n_j}\right)=
\limsup_{n\to\infty}\frac{\left|\left(\bigcap_{j=1}^k A_{n_j}\right)
  \cap F_n\right|}{|F_n|}\ge\\
\lim_{i\to\infty}\frac{\left|\left(\bigcap_{j=1}^k A_{n_j}\right) \cap
  F_{n_i}\right|}{|F_{n_i}|}=L\left(\bigcap_{j=1}^k A_{n_j}\right) .
\end{multline*}

Extending by linearity, we get a linear functional $L$ on a subspace
$V\subset B_\re(\na)$. By invoking the Hahn-Banach Theorem\footnote{
	We remark that for our applications we need only a ``restricted'' version
	of Theorem~\ref{4.24} which deals with functional $L_{\mathcal A}$ on
  $\mathcal A$ and does not need appealing to the Hahn--Banach
  Theorem.},
we can extend $L$ from $V$ to $B_\re(\na)$. This $L$ naturally extends to a
functional on the space $B_{\mathbb C}(\na)$, which satisfies
conditions (i) and (ii).

We move now to proving (iii). Let $\mathcal A$ be the uniformly closed
and closed under conjugation algebra of functions on $\na$, which is
generated by indicator functions $\mathbbm 1_A$ of sets $A\in{\mathcal F}$.
Then $\mathcal A$ is a separable $C^*$-subalgebra of
$\ell^\infty(\na,\|\cdot\|_\infty)$, and, by the Gelfand
Representation Theorem, ${\mathcal A}\cong C(X)$, where $X$ is a
compact metric space. The restriction $L_{\mathcal A}$ of the mean
$L$, which we constructed above, induces a positive linear functional
$\tilde L$ on $C(X)$, which by the Riesz Representation Theorem is
given by a Borel measure $\mu$.

Note that the isomorphism ${\mathcal A}\cong C(X)$ sends indicator
functions of subsets of $\na$ to indicator functions of subsets of $X$
(because the isomorphism provided by the Gelfand transform preserves
algebraic operations, and the indicator functions are the only ones
which satisfy the equation $f^2=f$). Let $\tilde A_j$ be the subsets
of $X$ which correspond to sets $A_j\in\mathcal F$ (note that since
$\mathbbm 1_{\tilde A_j}\in C(X)$ for each $j$, the sets $\tilde A_j$ are
measurable). Clearly, we have
\[
L_{\mathcal A}(A_{n_1}\cap A_{n_2}\cap\dots\cap A_{n_k})=
\mu(\tilde A_{n_1}\cap \tilde A_{n_2}\cap\dots\cap \tilde A_{n_k})
\]
for any $n_1,n_2,\dots,n_k$. This completes the proof.
\end{proof}

The last result which is needed for the proof of
Theorem~\ref{winner}, is the following theorem.

\begin{thm}[see Lemma~5.10 in~\cite{Be06}]\label{4.25}
Let $(K_n)$ be a F{\o}lner sequence in $(\na,\times)$, let \xbm\ be a
probability space, and let $A_j$, $j\in\na$, be measurable sets in $X$,
satisfying $\mu(A_j)\ge a$ for some $a>0$. Then there exists a
set $E\in\na$ with $\overline d_{(K_n)}(E)\ge a$, such that for any
nonempty finite set $F\subset E,$ one has $\mu\left(\bigcap_{j\in
  F}A_j\right)>0$.
\end{thm}

\begin{proof}[Proof of Theorem~\ref{winner}]
The result in question follows from Theorem~\ref{4.24} applied to the
F{\o}lner \sq\ $F_n=\{1,2,\dots,n\}$ in $(\na,+)$ and $\mathcal
F=\{A/n:n\in\na\}$ (and then we apply Theorem~\ref{4.25}). Note that
by Lemma \ref{A/n}, each $A/n$ is a
classical normal set and hence has density $a=\frac12$.
\end{proof}

It was shown in \cite{Be05} that multiplicatively large sets in $\na$
have very rich combinatorial structure (which is quite a bit richer
than that of additively large sets). For example, any multiplicatively
large set contains not only arbitrarily long geometric and arithmetic
progressions, but also all kinds of more complex structures which
involve both the addition and multiplication operations. Theorem
\ref{winner} allows us to conclude that classical normal sets in $\na$
are, in a way, as combinatorially rich as multiplicatively large sets.
For example, we can combine it with the following theorems.

\begin{thm}[Theorem 3.10 in~\cite{Be05}]\label{310be05}
Let ${\mathcal S}^{\mathsf a},{\mathcal S}^{\mathsf m}$ be two
families of finite subsets of $\na$ with the following properties:
\begin{enumerate}[(i)]
\item Any additively large set in $\na$ contains a configuration of
  the form $a + F$, where $F\in{\mathcal S}^{\mathsf a}$.
\item Any multiplicatively large set in $\na$ contains a configuration
  of the form $bF$, where $F\in{\mathcal S}^{\mathsf m}$.
\end{enumerate}
Then any multiplicatively large set $E$ contains a configuration of
the form $bF_2 (a + F_1)$, where $F_1\in{\mathcal S}^{\mathsf a}$ and
$F_2\in{\mathcal S}^{\mathsf m}$.
\end{thm}

\begin{thm}[Theorem 3.11 in~\cite{Be05}]\label{311be05}
Let $E\subset\na$ be a multiplicatively large set. Let $S_1,S_2\subset\na$ be
two infinite sets and let $IP^a(S_1)$ and $IP^m(S_2)$ be the additive and multiplicative
IP sets generated by $S_1$ and $S_2$, respectively. Then for any $n\in\na$, there exist
$a,b\in\na$, $d\in IP^a(S_1)$, and $q\in IP^m(S_2)$ such that
\[
\{bq^j(a + id),\, 0\le i,j\le n\}\subset E.
\]
\end{thm}
Then we get the following result.

\begin{thm}\label{classicalarerich}
Let $\mathcal S^{\mathsf a}$ and $\mathcal S^{\mathsf m}$ be two
families of finite sets in $\na$ which have the following properties:
\begin{enumerate}[(i)]
	\item any additively large set in $\na$ contains a
          configuration of the form $a+F_1$ for some $F_1\in\mathcal
          S^{\mathsf a}$ and $a\in\na$,\footnote{An example of
            $\mathcal S^{\mathsf a}$ is the family
            $\{\{r,2r,\dots,nr\}:r\in\na\}$ with any fixed $n$ (this
            follows the classical Szemer\'edi Theorem).}
	\item any multiplicatively large set in $\na$ contains a
          configuration of the form $bF_2$ for some $F_2\in\mathcal
          S^{\mathsf m}$ and $b\in\na$.\footnote{An example of
            $\mathcal S^{\mathsf m}$ is the family
            $\{\{q,q^2,\dots,q^n\}:q\ge 1\}$ with any fixed $n$ (see
            \cite[Theorem 3.11]{Be05}).}
\end{enumerate}
Then any classical normal set $A\subset\na$ contains a configuration
$bF_2(a+F_1)$ with $F_1\in\mathcal S^{\mathsf a}$, $F_2\in\mathcal
S^{\mathsf m}$ and $a,b\in\na$.

In particular, any classical normal set $A$ contains, for any
$n\in\na$, configurations of the form $\{q^j(a+id):\ 0\le i,\ j\le
n\}$ with some $q>1,\ a,d\in\na$.
\end{thm}

\begin{proof}
Theorem~\ref{310be05} tells us that any multiplicatively large set $E$
contains a configuration of the form $bF_2(a+F_1)$ with
$a,b\in\na,\ F_1\in\mathcal S^{\mathsf a}, F_2\in\mathcal S^{\mathsf
  m}$. For a classical normal set $A\subset\na$, we can take as $E$
the set given by Theorem~\ref{winner}. Thus, $E$ contains a set
$\{n_1,n_2,\dots,n_k\}$ of the above form $bF_2(a+F_1)$. Then the
intersection $A/n_1\cap A/n_2\cap\cdots\cap A/n_k$ has positive
additive upper density. In particular, this intersection contains some
natural number $c$, and then $cbF_2(a+F_1)\subset A$.

To get the last statement of the theorem, one has to use Theorem~\ref{311be05}
which guarantees the existence of configurations of
the form $\{bq^j(a+id):\ 0\le i,\ j\le n\}$ with some
$q>1,\ a,b,d\in\na$. Observe that we may write $bq^j(a+id)$ as
$q^j(a'+id')$.
\end{proof}

Similarly, we can invoke another theorem.

\begin{thm}[Theorem 3.15 in~\cite{Be05}]\label{315be05}
Let $E\subset\na$ be a multiplicatively large set. For any $k\in\na$
there exist $a,b,d\in\na$ such that $\{b(a + id)^j, 0\le i,j \le k\}
\subset E$.
\end{thm}

Then one gets the following result.

\begin{thm}\label{ssss}
Any classical normal set contains, for any $n\in\na$, sets of the form
$\{b(a+id)^j;0\le i,j\le n\}$ with some $a,b,d\in\na$.
\end{thm}

\begin{rem}
While Theorems \ref{classicalarerich} and \ref{ssss} guarantee that
any classical normal set contains arbitrarily long \emph{finite}
geometric progressions, it need not contain infinite geometric
progressions. Indeed, it is not hard to construct a set of density
zero which contains, for any $b,q\in\na$, a number of the form $bq^j$
for some $j$. Removing this set from a classical normal set results in
a desired example.
\end{rem}

\begin{rem}
Note that Theorem \ref{winner}, and thus Theorems
\ref{classicalarerich} and \ref{ssss}, are not valid for general
additively normal sets in $(\na,+)$. For example, one can show (see
\cite[Theorem 3.5]{BBHS06}) that there exist additively thick sets
which do not contain geometric progressions of length $3$,
$\{c,cr,cr^2\}$, where $r\in\mathbb Q\setminus\{1\}$ (cf. Remark
\ref{stip}).
\end{rem}

Many of the results of Section \ref{s4} are valid in a wider setup,
where one replaces normal sets with more general sets having strong
enough randomness properties. See for example \cite{Fi11}, where
configurations in so-called weakly mixing sets are studied.

\section{$(F_n)$-normal Liouville numbers}\label{fifa}

Let us recall that an irrational number $x$ is called a
\emph{Liouville number} if for every natural $k$ there exists a
rational number $\frac pq$ such that $|x-\frac pq|<\frac1{q^k}$.
Clearly, ``for every $k$'' can be equivalently replaced by ``for
arbitrarily large $k$'' (if $\frac pq$ is good for $k$, it is also
good for all $k'<k$). It is well known that the set $\mathcal L$ of
Liouville numbers is residual (dense $G_\delta$) but its Lebesgue
measure equals zero. This can be expressed concisely by saying that
this set is T-large and M-small. On the other hand, Theorem~\ref{B1}
and Corollary~\ref{normnumberscategory} imply that the set $\mathcal
N((F_n))$ of $(F_n)$-normal numbers is M-large and T-small (this
applies to both additive and multilplicative normality; we recall that
Theorem \ref{B1} requires a mild assumption on $(F_n)$ which is
satisfied e.g. when the \sq\ $|F_n|, \ n=1,2\dots$ is strictly
increasing). Thus, it is a priori not clear whether the sets $\mathcal
L$ and $\mathcal N((F_n))$ have a nonempty intersection. In this
section we will show that if $(F_n)$ is any F{\o}lner \sq\ in $(\na,+)$
or any nice F{\o}lner \sq\ in $(\na,\times)$ (see Section
\ref{Folners}) then $\mathcal L\cap\mathcal N((F_n))$ is not only
nonempty but in fact uncountable (contains a Cantor set). For results
dealing with Liouville numbers in the context of classical normality
see e.g. \cite{Bu02}.

\begin{thm}\label{new}
For every F{\o}lner \sq\ $(F_n)$ in $(\na,+)$ there exists an
$(F_n)$-normal Liouville number.
\end{thm}

The proof will be preceded by some generalities about F{\o}lner \sq s
in $(\na,+)$. Recall that two F{\o}lner \sq s $(F_n)$ and $(F'_n)$ in
an amenable semigroup $G$ are called equivalent if
$\frac{|F_n\triangle F'_n|}{|F_n|}\to 0$, and that if $(F_n)$,
$(F_n')$ are equivalent F{\o}lner \sq s then the notions of
$(F_n)$-normality and $(F_n')$-normality coincide.

\begin{lem}\label{foint}
Let $(F_n)$ be an arbitrary F{\o}lner \sq\ in $(\na,+)$. There exists a
\sq\ of natural numbers $(\ell_n)$ tending to infinity and a F{\o}lner
\sq\ $(F_n')$ equivalent to $(F_n)$ such that each set $F'_n$ is a
disjoint union of intervals, each of length at least~$\ell_n$.
\end{lem}

\begin{proof}
Fix a \sq\ $(\varepsilon_\ell)_{\ell\ge 1}$ decreasing to zero. For
each $\ell$ there exists $n_\ell$ such that for every $n\ge n_\ell$,
the set $F_n$ is $(K_\ell,\frac{\varepsilon_\ell}{2\ell})$-\inv, where
$K_\ell$ stands for $\{1,2,\dots,\ell\}$. Then, by Lemma \ref{estim},
the $K_\ell$-core of $F_n$, which we denote by $F_{n,K_\ell}$, is an
$\varepsilon_\ell$-mo\-di\-fi\-cation of $F_n$. For each $n$ we define
$\ell_n$ as the unique $\ell$ satisfying the inequalities $n_\ell\le
n<n_{\ell+1}$. We set $F'_n=F_n$ for $n<n_1$, and for $n\ge n_1$,
$F'_n=F_{n,K_{\ell_n}}+K_{\ell_n}$. Now, for each $n$, $F'_n$ is an
$\varepsilon_{\ell_n}$-modification of $F_n$ (hence $(F_n')$ is a
F{\o}lner \sq\ equivalent to $(F_n)$), and it is a union of (not
necessarily disjoint) intervals of length $\ell_n$. The ``connected
components''\footnote{By a connected component of a set $F\subset\na$
  we mean an interval $I=\{a,a+1,\dots,b\}\subset F$ such that
  $a-1\notin F$ (this includes the case $a-1=0$) and $b+1\notin F$.}
of $F'_n$ are disjoint intervals of lengths at least $\ell_n$, as
required.
\end{proof}

We will establish now some technical facts about (a subclass of)
Liuoville numbers which will be utilized in the proof of Theorem
\ref{new}.

\begin{defn}\label{repe}
We will call a binary \sq\ $w\in\{0,1\}^\na$ \emph{repetitive} if it
is the limit of a \sq\ of words $w_k$ ($k\ge 1$) defined
inductively, as follows:
\begin{enumerate}
	\item $w_1=u_1$ is an arbitrary nonempty 0-1 word,
	\item for $k>1$, $w_k=w_{k-1}w_{k-1}\dots w_{k-1}u_k$, where
          $w_{k-1}$ is repeated at least $k-1$ times, and $u_k$ is an
          arbitrary nonempty 0-1 word.
\end{enumerate}
\end{defn}

\begin{prop}\label{liouville}
Any not eventually periodic repetitive \sq\ $w$ is the binary
expansion of a Liouville number $x$.
\end{prop}

\begin{proof}
Given $k$, consider the rational number $\frac pq$ represented by the
periodic \sq\ $w_kw_kw_k\dots$\,. Then $q<2^{|w_k|}$, hence
$\frac1{q^k}>2^{-k|w_k|}$. The difference $|x - \frac pq|$ is a number
whose first nonzero binary digit appears at a position larger than
$k|w_k|$, which means that $|x - \frac pq|\le2^{-k|w_k|}$, hence $|x -
\frac pq|<\frac1{q^k}$. Since $w$ is not eventually periodic, $x$ is
irrational, and thus it is a Liouville number.
\end{proof}

Notice that since in Definiton \ref{repe} the word $u_k$ is completely
arbitrary, in particular it may have the form $v_kv_k\dots v_k$ (where
the word $v_k$ and the number of repetitions are also arbitrary).
Using this observation, we can isolate a special class of repetitive
\sq s.

\begin{defn}\label{bal}
Let $(v_k)$ be a \sq\ of nonempty binary words. A repetitive \sq\ $w$
is said to be \emph{balanced} with respect to $(v_k)$ if, for each
$k\ge 1$, $u_k=v_kv_k\dots v_k$, where the number of repetitions is
such that the following two conditions hold:
\begin{align}
&\delta_k=\frac{\max\{|v_k|,|v_{k+1}|,|v_{k+2}|,|w_{k-1}|\}}{|w_k|}\to
  0,\label{piec}\\ &1-\gamma_k=\frac{|u_k|}{|w_k|}\to 1\label{szesc}.
\end{align}
\end{defn}
It is easy to see that given any \sq\ of nonempty words $(v_k)$, one
can construct a repetitive \sq\ $w$ which is balanced with respect to
$(v_k)$. One just needs to apply large enough number of repetitions of
$v_k$ in $u_k$ (depending on the lenghts $|v_k|$, $|v_{k+1}|$ and
$|v_{k+2}|$).

\begin{lem}\label{kkk}
Let $(v_k)$ be a \sq\ of nonempty binary words and let $w$ be a
repetitve \sq\ balanced with respect to $(v_k)$. For each $k\ge 2$
define $\varepsilon_k=2(\delta_k+\gamma_k)$ (see Definition
\ref{bal}). If \,$W$ is a subword of $w_{k+2}$ with $|W|\ge|w_k|$
then, for some $r,s,t\ge 0$ satisfying
$\frac{r|v_k|+s|v_{k+1}|+t|v_{k+2}|}{|W|}\ge 1-\varepsilon_k$, $W$
contains $r+s+t$ nonoverlapping subwords of which $r$ are copies of
$v_k$, $s$ are copies of $v_{k+1}$ and $t$ are copies of $v_{k+2}$.
\end{lem}

\begin{proof}
Note that $w_{k+2}$ has the following structure:
$(w_{k+1})^a(v_{k+2})^{a'}$, and likewise
$w_{k+1}=(w_k)^b(v_{k+1})^{b'}$, $w_k=(w_{k-1})^c(v_k)^{c'}$ ($a\ge
k+1,b\ge k,c\ge k-1,a',b',c'\ge 1$). By successive substitution (two
times), we obtain that $w_{k+2}$ is a concatenation of (shifted)
copies of $v_k,v_{k+1},v_{k+2}$ and $w_{k-1}$. So is any subword $W$
of $w_{k+2}$, except that the copies covering the ends of $W$ may
extend beyond $W$ in which case the concatenation representing $W$
includes (at most two) \emph{end words} $V_1,V_2$ which are subwords
of either $v_k,v_{k+1},v_{k+2}$ or $w_{k-1}$.

To finish the proof we need to show that,
\[
\frac{p|w_{k-1}|+|V_1|+|V_2|}{|W|}<\varepsilon_k,
\]
where $p$ is the number of copies of $w_{k-1}$ in the concatenation
representing $W$. The fraction $\frac{p|w_{k-1}|}{|W|}$ is largest
precisely when $W=(w_{k-1})^c(v_k)^{c'}(w_{k-1})^c$ and then we have
\[
\frac{p|w_{k-1}|}{|W|}=\frac{2c|w_{k-1}|}{|w_k|+c|w_{k-1}|}<
2\frac{c|w_{k-1}|}{|w_k|}=2\gamma_k.
\]
The joint length of the end words not larger than
$2\max\{|v_k|,|v_{k+1}|,|v_{k+2}|,|w_{k-1}|\}$, so
$\frac{|V_1|+|V_2|}{|W|}<2\delta_k$. We have shown that the joint
length of the (nonoverlapping) copies of $v_k,v_{k+1}$ and $v_{k+2}$
which are subwords of $W$ is least
$(1-2\delta_k-2\gamma_{k-1})|W|=(1-\varepsilon_k)|W|$ and this is
precisely what we needed to show.
\end{proof}

\begin{proof}[Proof of Theorem \ref{new}]
Fix a F{\o}lner \sq\ $(F_n)$ in $(\na,+)$. In view of Lemma \ref{foint}
we can assume without loss of generality that if $\ell_n$ denotes the
length of the shortest connected component of $F_n$ then the
\sq\ $(\ell_n)$ tends to infinity. For each natural $j$ we define $t_j$
as the largest element of the set
\[
\bigcup_{\{n:\,\ell_n<j\}}F_n,
\]
i.e., $t_j$ is such that if $F_n$ has at least one connected component
shorter than $j$ then $F_n\subset\{1,2,\dots,t_j\}$.

Let $v\in\{0,1\}^\na$ be a classical normal \sq\ and let
$v_k=v|_{\{1,2,\dots,k\}}$. Note that the words $(v_k)$ are
\emph{asymptotically normal} in the following sense: for any nonempty
finite $K\subset\na$ and any $\varepsilon>0$, if $k$ is sufficiently
large then $v_k$ is $(K,\varepsilon)$-normal.

Let $w$ be a repetitive \sq\ which is balanced with respect to
$(v_k)$. By choosing the numbers of repetitions of $v_{k+2}$ in
$u_{k+2}$ (see Definition \ref{bal}) sufficiently large, we can
arrange that $|w_{k+2}|\ge t_{|w_{k+1}|}$, for each $k$. For each $n$
let $k_n$ be the unique integer satisfying the inequalities
$|w_{k_n}|\le\ell_n<|w_{k_n+1}|$. Notice that since the numbers
$\ell_n$ tend to infinity with $n$, so do the numbers $k_n$. By the
definition of the numbers $t_j$ and since $\ell_n<|w_{k_n+1}|$, we
have
$F_n\subset\{1,2,\dots,t_{|w_{k_n+1}|}\}\subset\{1,2,\dots,|w_{k_n+2}|\}$.
Thus, for any connected component $I$ of $F_n$, the word $W=w|_I$ is a
subword of length at least $|w_{k_n}|$ of $w_{k_n+2}$. Now,
Lemma~\ref{kkk} implies that at least the fraction
$1-\varepsilon_{k_n}$ of $w|_I$ is a constituted by nonoverlapping
copies of the words $v_{k_n}, v_{k_n+1}$ and $v_{k_n+2}$. Since
$\varepsilon_{k_n}\to 0$, it is now obvious that the blocks
$w|_{F_n}$\footnote{We use the term ``block'' because $F_n$ need not
  be an interval.} are asymptotically normal as $n$ grows to infinity,
i.e., that $w$ is $(F_n)$-normal. In particular, the number $x$ (whose
binary expansion is $w$) is irrational\footnote{Rational numbers are
  neither additively nor multiplicatively normal because their
  additive as well as multiplicative orbits are finite.}, hence it is
an $(F_n)$-normal Liouville number.
\end{proof}

\begin{rem}\label{huq}
If in the above construction we vary the classical normal element
$v$ (used to define the words $v_k$), while keeping the numbers of
repetitions of $v_k$ in $u_k$ unchanged, we obtain a continuous and
injective map $v\mapsto w$ sending classical normal \sq s to
$(F_n)$-normal repetitive \sq s. Moreover, since every $(F_n)$-normal
number is irrational, also the map $w\mapsto x$ (where $x$ is the
number whose binary expansion is $w$) is injective and continuous.
Thus, for every compact set $C$ consisting of classical normal \sq s,
the restriction to $C$ of the composition $v\mapsto w\mapsto x$ is a
homeomorphism of $C$ onto its image. Since the set of classical normal
\sq s contains a Cantor set, so does the set of $(F_n)$-normal
Liouville numbers.
\end{rem}

We now turn to constructing Liouville numbers which are
(multiplicatively) normal with respect to nice F{\o}lner \sq s. As we
shall see, repetitive \sq s are naturally well fitted for this kind of
normality. Recall (see Section \ref{Folners}) that for $m,M\in\na$ we
write $m\preccurlyeq M$ when $m|M$. If $m\preccurlyeq M$ and $M\neq
m$, we will write $m\prec M$. Recall also that a nice F{\o}lner
\sq\ $(F_n)$ in $(\na,\times)$ corresponds to a multiplicatively
increasing \sq\ $(L_n)$ of the leading parameters, i.e., natural
numbers such that, for each~$n$, $L_n\prec L_{n+1}$ and $F_n =\{m:
m\preccurlyeq L_n\}$.

\begin{lem}\label{misiu}
Given $k\ge 1$ and $\varepsilon>0$, there exists an
$m_{k,\varepsilon}$ such that for any $m$ and $M$ satisfying
$m_{k,\varepsilon}\preccurlyeq m\preccurlyeq M$, the interval
$\{m+1,\dots,(k+1)m\}$ contains at most a fraction $\varepsilon$ of
all divisors of $M$, i.e,
\[
\frac{|\{i:i\preccurlyeq M, \ m+1\le i\le(k+1)m\}|}{|\{i:i\preccurlyeq
  M\}|}\le\varepsilon.
\]
\end{lem}

\begin{proof}
Let $p$ be the smallest prime number strictly larger than $k$. Let
$r\in\na$ be such that $\frac1r\le\varepsilon$, and put
$m_{k,\varepsilon}=p^r$. Let $m$ be any multiple of
$m_{k,\varepsilon}$ and let $M$ be any multiple of $m$. The set of all
divisors of $M$ (which can be visualized as the anchored rectangular
box with the leading parameter $M$, see Section \ref{Folners}) splits
into disjoint union of one-dimensional sets of the form $aI=\{a, ap,
ap^2,\dots, ap^s\}$, where $a$ is not a multiple of $p$, and $p^s$ is
the largest power of $p$ dividing $M$. Clearly, $s\ge r$. Since $p\ge
k+1$, at most one element from any set $aI$ may fall in
$\{m+1,\dots,(k+1)m\}$. Thus at most the fraction
$\frac1s\le\frac1r\le\varepsilon$ of all divisors of $M$ may fall in
$\{m+1,\dots,(k+1)m\}$.
\end{proof}

\begin{thm}\label{lnormal}
For any nice F{\o}lner \sq\ $(F_n)$ in $(\na,\times)$ there exists an
$(F_n)$-normal Liouville number.
\end{thm}

\begin{proof}
The proof relies on choosing an arbitrary $(F_n)$-normal
0-1-\sq\ $\tilde w$ and modifying it on a set of $(F_n)$-density zero.
Clearly, then the modified \sq\ $w$ maintains $(F_n)$-normality. On
the other hand, we will make the \sq\ $w$ repetitive. Since
the number $x$ whose binary expansion is $w$ is multiplicatively
normal is not rational, Proposition~\ref{liouville} will imply that
$x$ is the desired $(F_n)$-normal Liouville number.

Given $k\ge 1$, Lemma \ref{misiu} applied for $k$ and
$\varepsilon=2^{-k}$ provides a number $m_{k,2^{-k}}$. Let $n_k$ be
the smallest index $n$ such that $m_{k,2^{-k}}\in F_n$ (i.e.,
$m_{k,2^{-k}}\preccurlyeq L_n$)
and let $m_k=\mathsf{LCM}(m_{k,2^{-k}}, L_{n_k-1})$. In this manner, we have
assured that $L_{n_k-1}\preccurlyeq m_k\preccurlyeq L_{n_k}$. Since
$m_k$ is a multiple of $m_{k,2^{-k}}$, the following holds:
\begin{equation}\label{misi}
m_k\preccurlyeq M \text{ \ \ implies \ \ }\frac{|\{i:i\preccurlyeq M,
  \ m_k+1\le i\le(k+1)m_k\}|}{|\{i:i\preccurlyeq M\}|}\le2^{-k}.
\end{equation}
Further, it is obvious that $m_k$ can be replaced by $m_{k'}$ with any
$k'\ge k$ ($m_{k'}$ has the above property with $k'$ thus also with
$k$). Hence, passing if necessary to a sub\sq, we can assume that
$m_{k+1}>(k+1)m_k$ for each $k$. Although the property
$L_{n_k-1}\preccurlyeq m_k\preccurlyeq L_{n_k}$ may be lost, we still
have for any natural indices $k$ and $n$, either $m_k\preccurlyeq L_n$
or $L_n\preccurlyeq m_k$.

Now we are in a position to define $w$. We let $u_1=w_1=\tilde
w|_{\{1,\dots,m_1\}}$. Next, we define $w_2=w_1w_1u_2$, where
$u_2=\tilde w|_{\{2m_1+1\dots,m_2\}}$. Notice that the coordinates on
which $w_2$ disagrees with $\tilde w_{\{1,\dots,m_2\}}$ (if any) are
contained in the interval $\{m_1+1,\dots 2m_1\}$. Then we define
$w_3=w_2w_2w_2u_3$, where $u_3=\tilde w|_{\{3m_2+1,m_3\}}$. Similarly,
the coordinates where $w_3$ disagrees with $\tilde w_{\{1,m_3-1\}}$
(if any) are contained in the union $\{m_1+1,\dots
2m_1\}\cup\{m_2+1,\dots 3m_2\}$. Continuing in this way we will define
a \sq\ of words $w_k$ converging to a \sq\ $w$ which agrees with
$\tilde w$ on the complement of the set
\begin{equation}\label{union}
\bigcup_{k\ge 1} \{m_k+1,\dots (k+1)m_k\}.
\end{equation}
According to Definition \ref{repe}, $w$ is a repetitive \sq. It
remains to show that the $(F_n)$-density of the union \eqref{union} is
zero. Given an $n\in\na$, we divide the indices $k$ into three classes
(some of them possibly empty): $k\in\mathbb S_n$ if
$km_k\le|F_n|^{\frac13}$, $k\in\mathbb L_n$ if $m_k\ge L_n$ and
$\mathbb M_n=\na\setminus(\mathbb S_n\cup\mathbb L_n)$.
\begin{itemize}
	\item For $k\in\mathbb S_n$ we have
\[
\frac{|\{m_k+1,\dots,(k+1)m_k\}\cap
  F_n|}{|F_n|}=\frac{km_k}{|F_n|}\le|F_n|^{-\frac23}.
\]
Because $|\mathbb S_n|\le|F_n|^{\frac13}$, we have
\[
\frac1{|F_n|}{\Bigl|\bigcup_{k\in\mathbb S_n} \{m_k+1,\dots
  (k+1)m_k\}\cap F_n \bigr|}\le|F_n|^{-\frac13}.
\]
\item For $k\in\mathbb L_n$, $F_n$ is disjoint from $\{m_k+1,\dots
  (k+1)m_k\}$, hence
\[
\frac1{|F_n|}{\Bigl|\bigcup_{k\in\mathbb L_n} \{m_k+1,\dots
  (k+1)m_k\}\cap F_n \bigr|}=0.
\]
\item For $k\in\mathbb M_n$, we have $L_n>m_k$, in particular
  $L_n\not\preccurlyeq m_k$ and thus $m_k\preccurlyeq L_n$. By
  \eqref{misi}, we have
\[
\frac1{|F_n|}|\{m_k+1,\dots (k+1)m_k\}\cap F_n|\le 2^{-k}.
\]
\end{itemize}

Putting the above three cases together, we get
\[
\frac1{|F_n|}{\Bigl|\bigcup_{k\in\na} \{m_k+1,\dots (k+1)m_k\}\cap F_n
\bigr|}\le|F_n|^{-\frac13}+\sum_{k\in\mathbb M_n} 2^{-k}.
\]
Since $|F_n|\to\infty$, the right hand side tends to zero with $n$
(note that every $k$ eventually falls in $\mathbb S_n$).
\end{proof}

\begin{rem}
Denote by $\mathbb D$ the complement in $\na$ of the union
\eqref{union}. Then $\mathbb D$ has $(F_n)$-density $1$ in
$(\na,\times)$, and $w|_\mathbb D=\tilde w|_\mathbb D$, where $\tilde
w$ is the $(F_n)$-normal \sq\ chosen at the beginning of the proof of
Theorem \ref{lnormal}. The construction of $w$ uses only the subwords
of $\tilde w$ appearing in $\tilde w|_\mathbb D$, hence the mapping
$\tilde w|_\mathbb D\mapsto w$ is injective (and obviously it is also
continuous). It is easy to see that there exists a Cantor set
consisting of $(F_n)$-normal elements $\tilde w$ on which the map
$\tilde w\mapsto\tilde w|_\mathbb D$ is injective. On this Cantor set,
the map $\tilde w\mapsto w$ is injective and continuous. Arguing as in
Remark \ref{huq}, we get that the map $w\mapsto x$ is injective and
continuous on this Cantor set. This implies that the set of
(multiplicatively) $(F_n)$-normal Liouville numbers contains a Cantor
set.
\end{rem}

\begin{rem}
The technique employed in the proof of Theorem \ref{lnormal} can be
utilized to obtain Liouville numbers with other properties. Let
$(F_n)$ be a nice F{\o}lner \sq\ in $(\na,\times)$ and let $P$ be any
property satisfied by a nonempty set of numbers and preserved under
zero $(F_n)$-density modifications of the binary expansions (for
example, the property of being generic for some multiplicatively
invariant, not necessarily Bernoulli, measure). Then there exist
Liouville numbers with property $P$.
\end{rem}

We conclude this section (and the paper) with an open problem.
\begin{ques}
Do there exist net-normal Liouville numbers?
\end{ques}
We remark that our technique does not allow us to produce such numbers.
Indeed, for any fixed \sq\ of intervals of the form
$\{m_k+1,\dots,km_k\}$, the union \eqref{union} has upper density at
least $\frac12$ for a suitable nice F{\o}lner \sq. To prove this, it
suffices to indicate for any leading parameter $L$ a multiple $pL$
such that half of divisors of $pL$ belong to one of the intervals
$\{m_k+1,\dots,km_k\}$. To this end, choose $k\ge 2L$ and a prime
number $p$ in $\{m_k+1,\dots,2m_k\}$ (such $p$ exists by Bertrand's
postulate). Then at least half of the divisors of $pL$ have the form
$pl$, where $l\preccurlyeq L$ (in particular $l\le L$) and then $m_k<
p\le pl \le 2m_kL \le km_k$.
\bigskip

\appendix
\section*{Appendix}

In this appendix we briefly discuss the original proof in \cite{Bo09}
of the fact that the set of normal numbers in $[0,1]$ has full
Lebesgue measure, and the controversies it generated. The proof has
two parts. In the first part Borel defines a number $x\in[0,1]$ to be
\emph{simply normal in base $b$} if the frequency of every digit
$0,\,1,\,\dots,\,b-1$ in the expansion of $x$ equals $\frac1b$. He
then shows that the set of numbers in $[0,1]$ which are simply normal
in base $b$ is of full Lebesgue measure. One can view this result as a
special case of the Strong Law of Large Numbers (SLLN). The proof is
based on what is now known as the Borel-Cantelli Lemma. We remark that
in this part it is inessential that $\z$ is a group. What matters is
that the functions $X_i=\lfloor b^ix\rfloor \mod b$ (which express the
digits in the base $b$ expansion of $x$) form a countable family of
independent identically distributed random variables and that the
averaging sets $F_n$ (in this case $\{1,2,\dots,n\}$) strictly
increase in cardinality. The F{\o}lner property and the inclusions
$F_n\subset F_{n+1}$ are not used. In the second part of the proof,
Borel defines a number $x$ to be \emph{completely normal} if for
every $k,m\ge 1$ numbers $b^mx$ (considered modulo 1) are simply normal in
base $b^k$. As a countable intersection of sets of full measure, the
set of completely normal numbers also has full measure. Then Borel
writes (for $b=10$):
\medskip

\begin{minipage}[c]{0.9\textwidth}
{\small ``La propri\'et\'e caract\'eristique d'un nombre normal est la
  suivante: \emph{un groupement quelconque de $p$ chiffres consdcutifs
    \'etant consid\'er\'e, si l'on d\'esigne par $c_n$ le nombre de
    fois que se rencontre ce groupement dam les n premiers chiffres
    d\'ecimaux, on~a:}
\begin{equation}
\lim_{n\to\infty} \frac{c_n}n = \frac1{10^p}\ .\text{ \ ''}\tag{*}
\end{equation}
\smallskip

(The characteristic property of a normal number is the following: for
any grouping of $p$ consecutive digits being considered, denoting by
$c_n$ the number of times this grouping occurs in the first $n$
decimal digits, one has (*).)}
\end{minipage}
\medskip

This ``characteristic property'' is exactly normality in terms of our
Definition \ref{clnor} (adapted to base 10). Borel does not prove
equivalence between his definition of ``complete normality'' and the
``propri\'et\'e caract\'eristique''. Perhaps Borel intentionally
skipped the proof (considering it fairly obvious), but this omission
triggered a long-lasting controversy (and confusion). In particular,
Champernowne \cite{Ch33}, Koksma~\cite{Ko37}, Copeland and Erd\H os
\cite{CE46}, Hardy and Wright \cite{HW45} explicitly or implicitly
used the unproved equivalence. To illustrate how far from obvious this
equivalence was at that time, let us quote what Donald D. Wall claimed
in his dissertation \cite{Wa49} (written in 1949 under the supervision
of Derrick H. Lehmer):
\medskip

\begin{minipage}[c]{0.9\textwidth}
{\small ``Actually, there seems to be little reason to believe that
  the classes are identical.''}
\end{minipage}
\medskip

\noindent In fact, Wall believed to be close to finding a
counterexample:
\medskip

\begin{minipage}[c]{0.9\textwidth}
{\small ``Certain aspects of the problem are discussed in some detail
  here, and the main result is a new method of constructing some class
  II numbers -- a method which seems to give hope of finding a class
  II number which is not in class III.''}
\end{minipage}
\medskip

Eventually the equivalence was established by I. Niven and H. S.
Zuckerman in 1951 \cite{NZ51} (see also \cite{Ca52}). Today this
equivalence is no longer controversial. Once we understand that
normality (in the sense of Borel's ``characteristic property'') of a
\sq\ $x$ implies normality of $x$ restricted to any infinite
arithmetic progression\footnote{Ironically, this
  implication was first proved by Wall in his dissertation, but
  apparently he has not realized that it solves the ``equivalence
  problem''. In modern times the implication follows immediately from
  the fact that K-systems (in particular Bernoulli systems) and
  systems with entropy zero (in particular periodic) are disjoint in
  the sense of Furstenberg. }
\medskip

Borel was also criticized for other gaps in his proof. One such
criticism appears in the 1910 book of Georg Faber \cite{Fa10} on page
400. It seems that Faber finds it unclear that the Lebesgue measure on
$[0,1]$ corresponds to the distribution of the i.i.d. process
$\{X_i\}_{i\ge 1}$, where the $X_i$'s are the random variables defined
above.
\medskip

\begin{minipage}[c]{0.9\textwidth}
{\small ``Sodann hat Herr Borel k\"urzlich nach Aufstellung geeigneter
  Definitionen \"uber Wahrscheinlichkeit bei einer abz\"ahlbaren Menge
  von Dingen bewiesen, dass die Wahrscheinlichkeit daf\"ur, dass ein
  Punkt der obigen Menge angeh\"ort, gleich Null ist. Die Vergleichung
  des obigen Satzes mit dem Borelschen Resultat legt die Frage nahe:
\smallskip

Ist die Wahrscheinlichkeit -- nach der Borelschen Festsetzung die
eventuell zur Beantwortung dieser Fragen zu erweitern w\"are --, dass
eine Zahl einer bestimmten vorgelegten Menge vom Masse Null
angeh\"ort, immer gleich Null? Und umgekehrt: Ist eine Menge immer vom
Masse Null, wenn die Wahrscheinlichkeit, dass ein Punkt ihr
angeh\"ort, gleich Null ist?

\medskip
(Next, shortly after establishing appropriate definitions about
probability associated with a countable set, Mr. Borel has proved that
the probability for a point to be an element of the above set equals
zero. The comparison of the above theorem with Borel's result suggests
the following question:

Is the probability -- which, according to Borel's definition, possibly
has to be extended in order to answer these questions -- that a number
belongs to a certain given set of mass zero, always zero? And
conversely: Does a set always have mass zero, if the probability that
a point belongs to it is zero?\footnote{We thank Christoph
  Kawan for helping us with the translation.})}
\end{minipage}
\medskip

Because of Faber's somewhat antiquated style and terminology, we are
not exactly sure what is bothering him, but from today's perspective,
the equivalence between the above two meanings of a null set leaves no
doubts. It is worth mentioning that Faber provides his own, different
proof of SLLN. A reference to Faber's proof is made in the following
passage in the survey \cite{Do96} by Joseph L. Doob, where he
indicates that Borel has actually proved only convergence in measure
rather than almost everywhere:

\medskip\medskip

\begin{minipage}[c]{0.9\textwidth}{\small
``Classical elementary probability calculations imply that this
    sequence of averages converges \underline{in measure} to 1/2, but
    a stronger mathematical version of the law of large numbers was
    the fact deduced by Borel---in an unmendably faulty proof---that
    this sequence of averages converges to 1/2 for (Lebesgue measure)
    \underline{almost every} value of $x$. A correct proof was given a
    year later by Faber, and much simpler proofs have been given
    since. [Fr\'echet remarked tactfully: <<Borel's proof is
      excessively short. It omits several intermediate arguments and
      assumes certain results without proof.>>]''}
\end{minipage}
\bigskip

\noindent So, what is actually wrong with Borel's proof of the SLLN? A
careful examination of Borel's proof reveals the following:
\begin{enumerate}
	\item On pages 250--252, it is proved, under the (implicit)
          assumption that a \sq\ of sets $A_n$ is independent, that if
          the \sq\ of probabilities $\mathbb P(A_n)$ is summable then
          the upper limit $\bigcap_{m\ge 1}\bigcup_{n\ge m} A_n$ has
          measure zero (which is a special case of what is today
          called the Borel--Cantelli Lemma).
	\item On page 259, in the proof of the fact that simply normal
          numbers form a set of full measure (in other words, in the
          proof of the SLLN for 0-1 valued random variables), the
          Borel--Cantelli Lemma is applied to sets $A_n$ which are not
          independent.
	\item The proof of the full version of the Borel--Cantelli
          Lemma is missing. Without it, Borel's proof indeed
          establishes (as pointed out by Doob) only the version of the
          Law of Large Numbers which involves the convergence in
          measure.
\end{enumerate}

\noindent So, formally speaking, Borel's proof does contain a gap. But
does that mean that the proof is ``unmendably faulty''? We are
inclined to accept Fr\'echet's assessment, that the proof was just
excessively short.

\bigskip
We conclude with a comment concerning the possibility of adapting
Borel's method to more general amenable groups.

The key property of $\z$ which is behind the equivalence between the
two Borel's definitions of normality is that $\z$ admits, for each
$k$, a \emph{monotiling} (tiling with one shape) with the shape being
the interval $\{0,1,\dots,k-1\}$. It is plausible that for
monotileable groups\footnote{A group $G$ is \emph{monotileable} if it
  admits monotilings with arbitrarily large shapes, by which we mean
  that any finite set $K\in G$ is eventually a subset of the shape of
  some monotiling.} Borel's definition of normality and his proof that
$(F_n)$-normal elements form a set of full measure $\lambda$ can be
adapted with not too much effort to a large class of F{\o}lner \sq s.
But it seems impossible to extend Borel's definition of normality to
elements of $\{0,1\}^G$, where $G$ is any infinitely countable
amenable group or semigroup. Although the notion of \emph{simple}
normality can be naturally defined in this case, and moreover, by
essentially the same proof as in the case of $\{0,1\}^\z$, one can
show that almost every element $x\in\{0,1\}^G$ is simply normal, it is
not clear what is the analog of the operation of changing the base
from $b$ to $b^k$. One would need to find a large finite set $S$ which
tiles the group (i.e., is a shape of a monotiling $\mathcal T$) and
then treat the blocks $B=x|_T$ (where $T=Sc$, $c\in C_S$, are the
tiles of $\mathcal T$) as new symbols (from the alphabet
$\{0,1,\dots,b-1\}^S$) associated to the centers $c$ of the tiles. It
is not known which groups (except residually finite) admit monotilings
with arbitrarily large shapes. In fact, it is an open problem whether
all countable amenable groups are monotileable. This is the reason
why in the proof of Theorem \ref{B1} we must use tiling with many
shapes which complicates the proof of this theorem.
\medskip

\bigskip
\textbf{Acknowledgements.}
The first author gratefully acknowledges the support of the NSF under
grant DMS-1500575. The research of the second author is supported by
the NCN (National Science Center, Poland) Grant 2013/08/A/ST1/00275.
The work of the third author was supported by the grant number 426602
from the Simons Foundation.

\end{document}